\documentclass[11pt]{article}  

\usepackage{hyperref}

\usepackage{stmaryrd}
\usepackage[tight]{minitoc}
\usepackage{enumitem}
\usepackage{url}
\usepackage{amsfonts}
\usepackage{amsthm}
\usepackage{amssymb,amsmath}
\usepackage{latexsym}
\usepackage{graphics}
\usepackage{epic}
\usepackage{epsfig}
\usepackage{psfrag}
\usepackage{pict2e}  
\usepackage{subfigure}
\usepackage{bbm}
\usepackage{dsfont}
\usepackage{color}
\usepackage{caption}
\usepackage{tikz}
\usepackage{calc}
\definecolor{gray5}{gray}{0.8}
\definecolor{gray1}{gray}{0.4}
\definecolor{gray2}{gray}{0.6}
\definecolor{gray3}{gray}{0.7}
\definecolor{gray4}{gray}{0.3}

\numberwithin{equation}{section}

\newtheorem     {thm}{Theorem}[section]

\newtheorem     {lem}[thm]{Lemma}
\newtheorem     {prop}[thm]{Proposition}
\newtheorem     {cor}[thm]{Corollary}

\newtheorem     {rem}[thm]{Remark}
\newtheorem     {assu}[thm]{Assumption}

\newcommand     {\mf}[1] {#1}%attention new def
\renewcommand     {\b}[1] {#1}%attention new def

\definecolor{darkred}{rgb}{0.9,0.1,0.1}

 \newcounter{hypo}

\def\cmb#1{\marginpar{\raggedright\tiny{\textcolor{blue}{#1}}}}

\begin{document}

\vspace{-3cm}

\title{Markovian tricks for non-Markovian trees:  contour process \\ \Large \it Extinction and Scaling limits} 

\date{}
\author{\textsc{Bertrand Cloez$^{1}$}, \textsc{Benoit Henry$^{2,3}$}}
\footnotetext[1]{MISTEA, INRA, Montpellier SupAgro, Univ. Montpellier}
\footnotetext[2]{IECL, Universit\'e de Lorraine, Site de Nancy, B.P. 70239, F-54506 Vandoeuvre-l\`es-Nancy Cedex, France}
\footnotetext[3]{CNRS, IECL, UMR 7502, Vand{\oe}uvre-l\`es-Nancy, F-54506, France\\ E-mail:
 \texttt{Bertrand.cloez@inra.fr}, \texttt{benoit.henry@univ-lorraine.fr}}

\maketitle

\begin{abstract}
In this work, we study a family of non-Markovian trees modeling populations where individuals live and reproduce independently with possibly time-dependent birth-rate and lifetime distribution. To this end, we use the coding process introduced by Lambert. We show that, in our situation, this process is no longer a L\'{e}vy process but remains a Feller process and we give a complete characterization of its generator. This allows us to study the model through Markov processes techniques. On one hand, introducing a scale function for such processes allows us to get necessary and sufficient conditions for extinction or non-extinction and to characterize the law of such trees conditioned on these events. On the other hand, using Lyapounov drift techniques, we get  another set of, easily checkable, sufficient criteria for extinction or non-extinction and some tail estimates for the tree length. Finally, we also study scaling limits of these trees and observe that the Bessel tree appears naturally.

\end{abstract} 
\bigskip



\tableofcontents 
    
\section{Introduction}
\label{sec:intro}%\cmb{revoir dans tt le papier les domaines/generateur} \cmbb{je ne comprend pas}

In this paper we study non-Markovian trees modeling time-inhomogeneous populations where individuals live and reproduce independently but where birth-rate $b$ and lifetimes distributions $K$ are time-dependent. In addition, we assume that the population starts from a single individual called the \emph{root} or the \emph{ancestor} and that individuals only give birth to a single child at time.
In particular, as the time-homogeneous case is known has splitting trees, we simply call such tree \emph{inhomogeneous splitting trees} (IST). The homogeneous case has been introduced in \cite{GK} but one of the main steps in the study of these trees was done by Lambert \cite{Lambert2010}. The inhomogeneous case has also been studied recently by Lambert and Popovic \cite{LP13} and by Lambert and Stadler \cite{LS13}. In \cite{LS13}, the authors are interested in the so-called \emph{reconstructed phylogenetic tree} of this model (among various others) which summarizes de genealogical relations between the individuals. They question whether this reconstructed tree is a CPP  or not. A CPP is a kind of random tree where all the tips lay at a same distance from the root and where the coalescent times of the tips are independent random quantities. In particular, they obtain explicit expression for the branch length distribution which we can recover in our context. We also refer the interested reader to \cite{LP13,L17}.

A powerful technique to study such model, on which this work relies, is based on coding processes, or contour processes. Contour processes are central objects in the study of trees in probability. They allow to substitute the study of trees with the study of real-valued functions which may be considerably easier in many situations. In particular, one of the most famous result involving contour processes is the proof of the convergence of conditioned Galton-Watson trees to the Brownian tree \cite{aldousBTree} using the Harris paths (which is a particular type of contour process). This result relies on a Donsker like Theorem. A considerable amount of research involved different types of contour processes of different types of random trees have been done in \cite{aldousBTree,LeGallTree, lambertTOM,duqu} (among many other works).

\iffalse
 Two powerful techniques have recently been introduced allowing to study such models. The first one is based on a spinal decomposition of the tree and many-to-one formulas \cite{R11}. The second technique, on which this work relies, is based on coding processes, or contour processes. Contour processes are central objects in the study of trees in probability as well as in computer science. Coding processes allow to substitute the study of trees with the study of real valued functions which may be considerably easier in some situations. In particular, one of the most famous result involving contour processes is the proof of the convergence of conditioned Galton-Watson trees to the Brownian tree \cite{aldousBTree} using the Harris paths (which is a particular type of contour process). This result relies on a Donsker like Theorem. A considerable amount of research involved different types of contour processes of different types of random trees \cite{aldousBTree,LeGallTree, lambertTOM,duqu}.
 \fi

 In this work, we use the so-called Jump Chronological Contour Process (JCCP) which was introduced in \cite{Lambert2010}. In \cite{Lambert2010}, Lambert studies time-homogeneous splitting trees and introduce the JCCP as a coding process for such tree. Roughly speaking, this contour process is constructed by going along the tree (on the form of all our figures) from left to right and by writing a decreasing, linear (with slope $1$) function when we are going down (namely we follow the life of an individual) and we write a jump when we are going up (namely at birth). In the homogeneous case, it turns out that this stochastic process is a L\'evy process. In particular, using the well-developed theory of fluctuations of L\'evy processes, this property allowed to understand many properties of the splitting trees and of related models \cite{Lambert2010,CLR,Ricthese,R11}. In more recent work, Lambert and Uribe Bravo \cite{lambertTOM} demonstrate that JCCPs can be used to code for a far more general family of random trees known as \emph{Totally Ordered Measured trees} (TOM trees). In our case, the JCCP is (generally) no longer a L\'evy process, and all the arguments used in the homogeneous case fall down.
 
 \medskip
 
 \iffalse
 Our model is constructed around two functions: 
 \begin{itemize}
\item  $t\in\mathbb{R}_{+}\mapsto b(t)\in\mathbb{R}_{+}$ where $b(t)$ is the instantaneous birth-rate of the population at $t$ time.
\item $t\in\mathbb{R}_{+}\mapsto K(t,dy)\in \mathcal{M}_{1}(\overline{\mathbb{R}}^{\ast}_{+})$ is the instantaneous lifetime distribution of the individuals at time $t$, where $\mathcal{M}_{1}(\overline{\mathbb{R}}^{\ast}_{+})$ is the space of probability measure on the extended real line excluding $0$.
 \end{itemize}
 \fi 
 
 In our situation, we show that the JCCP looses its space homogeneity. The main consequence is that the theory of fluctuation of L\'evy process can't be used in this case preventing to extend the methods of the homogeneous case \cite{B98}. However, it remains a Feller process, and we derive its generator:
 \begin{equation}
 \label{eq:genintro}
 Lf(x)=f'(x)+b(x)\int_{\mathbb{R}_{+}}(f(x+y)-f(x))\, K(x,dy),\quad x\in\mathbb{R}_{+},
  \end{equation}
 with a full characterization of its domain. Parameters $b,K$ denotes respectively the birth rate and the lifetime distribution; a full construction of the tree and details on these parameters is given in Section~\ref{sec:model} below. As a consequence, we can exploit the rich literature on Markov processes \cite{EK86,RY,RW00,MT09,KS} to derive new properties of the tree (even in the time-homogeneous case).
 
 One of the typical method to study the fluctuation of Markov processes (in dimension $1$) is the scale functions theory. Unfortunately, scale function theories only exist in the case of L\'evy or diffusion processes. In this work, we introduce a new type of scale function which is relevant for the contour process and is consistent with L\'evy and diffusion theories. Among other results, this allows us to obtain
 
 \begin{thm}[Probability of extinction]
 \label{thm:intro1}
 Under Assumption~\ref{ass:b,K} below, let $E(v)$ be the probability of extinction of the population starting from an individual of fixed lifetime $v$. Then, under Assumption~\ref{ass:b,K}, $E$ satisfies
 \[
 \ E(\mf{v})= e^{-\int_0^\mf{v} b(\mf{u}) d\mf{u}} + e^{-\int_0^\mf{v} b(\mf{u}) d\mf{u}} \int_0^\mf{v} b(\mf{s}) e^{\int_0^\mf{s} b(\mf{u}) d\mf{u}} \int_{[0,+\infty)} E(\mf{s}+\mf{x})\, K(\mf{s},d\mf{x})\, d\mf{s},\quad \forall v\geq 0.
 \]
 \end{thm}
 See Corollary~\ref{cor:ext} for details and proof. This last result allows obtaining necessary and sufficient conditions for the extinction of the population.
 Another interesting application of our scale functions is the conditioning of trees. Indeed, we are able to obtain the distribution of IST conditioned either to extinction or non-extinction. A surprising aspect is that homogeneous splitting tree conditioned to non-extinction becomes non-homogeneous in time (see Theorem~\ref{prop:htrans} below).
 
 In this work, we also obtain more practical criteria of extinction than those who derive from Theorem \ref{thm:intro1} (which are optimal but sometimes abstract). To this end, we use Lyapounov drift methods. In particular, denoting $m(x)$ for the mean value of $K(x,dy)$, we show that ``weighted means'' of $( m(s)b(s)-1)$, such as
 \[
 \int_0^{\infty}  \left(( m(s)b(s)-1)  e^{-\int_s^{\infty} b(u)du} \right) ds,
 \] can be used to determine if the population extinct almost surely; see Corollary~\ref{cor:lyapd}. This nicely reminds criticality criteria in homogeneous cases as in \cite[Proposition 2.2]{Lambert2010}. Besides sufficient conditions for extinction or survival, these techniques also enables to derive bounds on the tail of the total length of the tree; see Proposition~\ref{prop:tail}. This type of result seems new (even in the time-homogeneous case) and is in general a difficult question for Markovian tree see for instance \cite{ADJ13,K17}.

\begin{figure}[ht]
\begin{center}
  \subfigure[$N=5$]{\includegraphics[width=5.5cm]{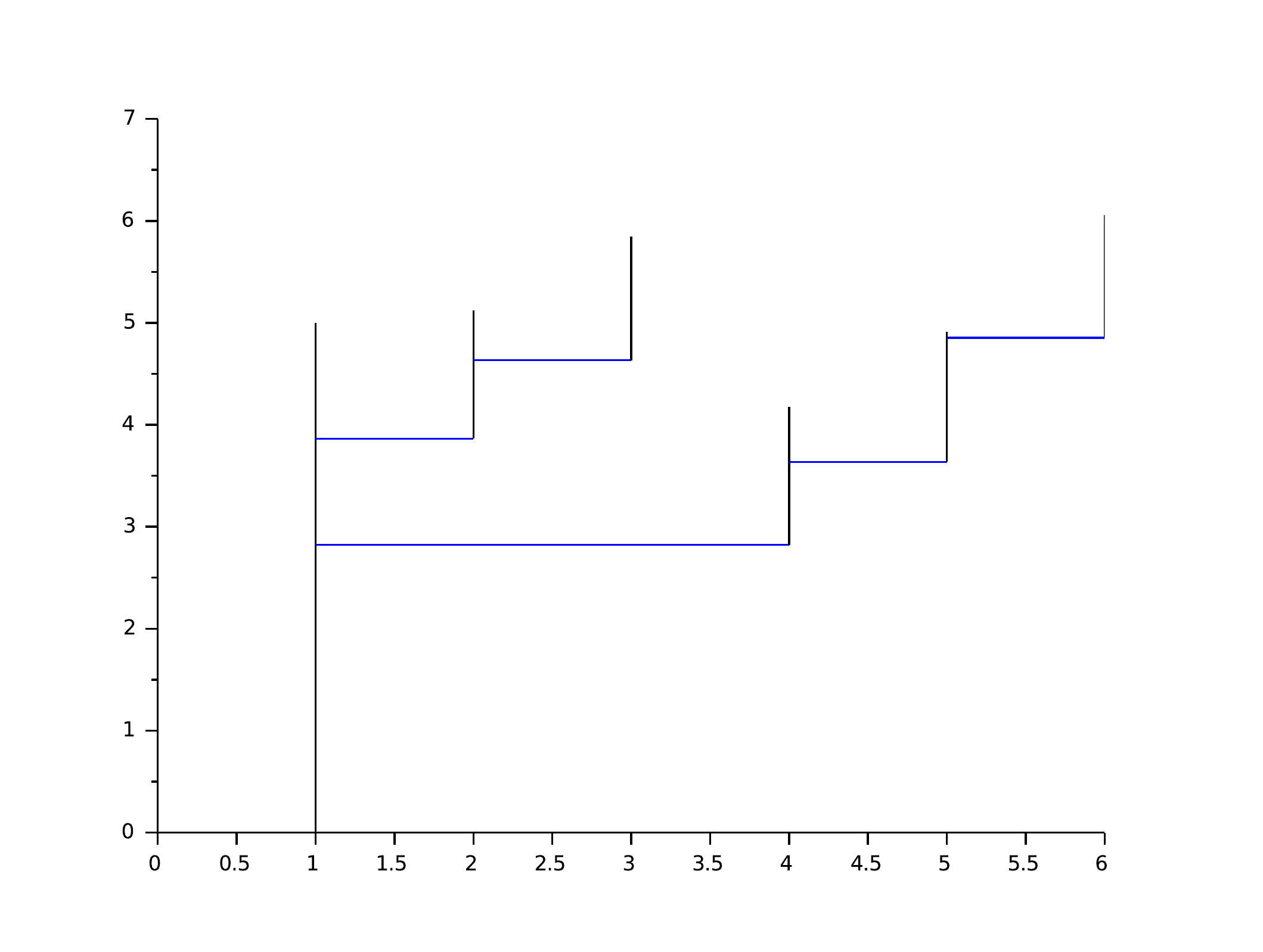}}
  \subfigure[$N=50$]{\includegraphics[width=5.5cm]{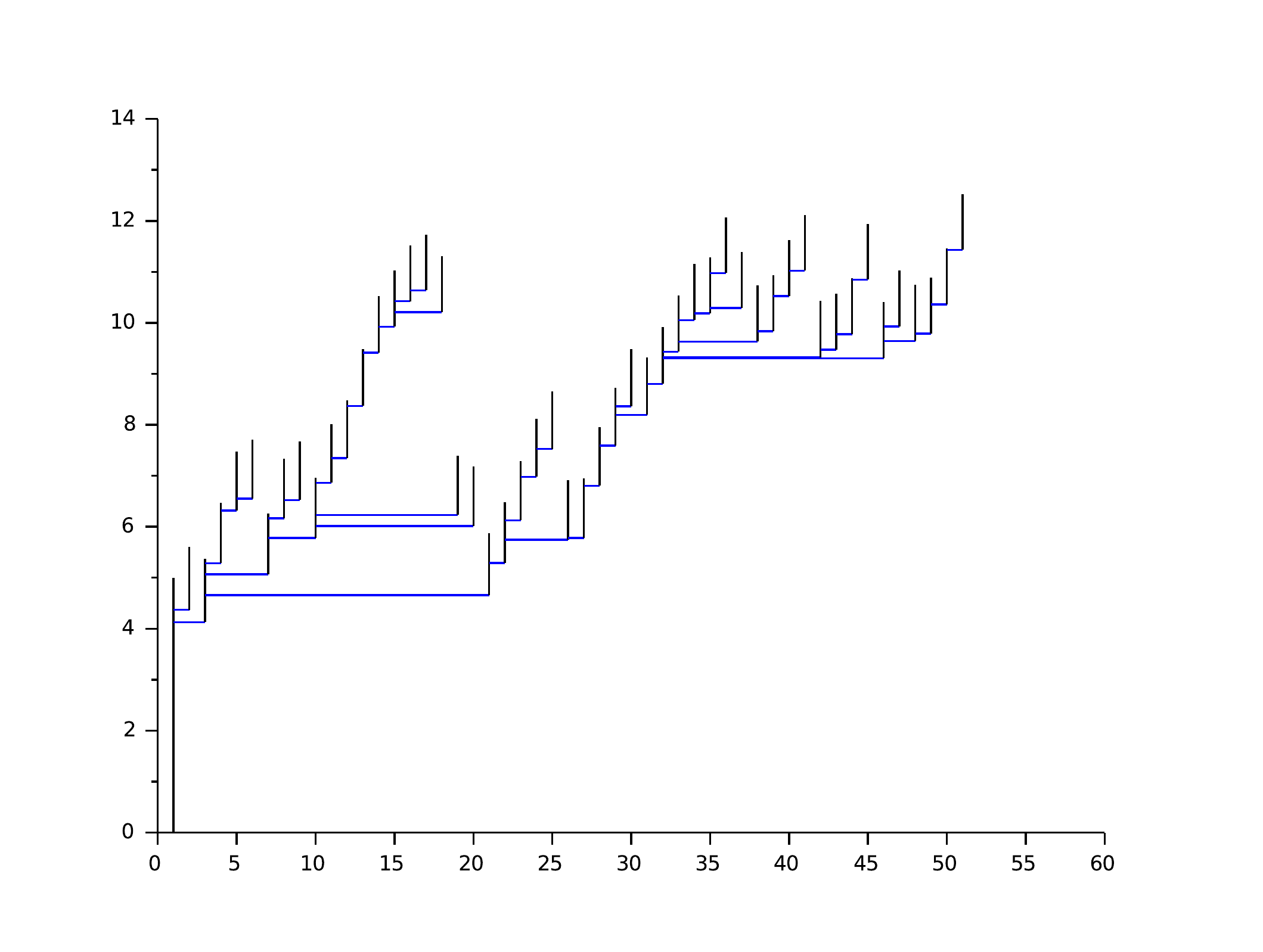}}
  \subfigure[$N=2000$]{\includegraphics[width=5.5cm]{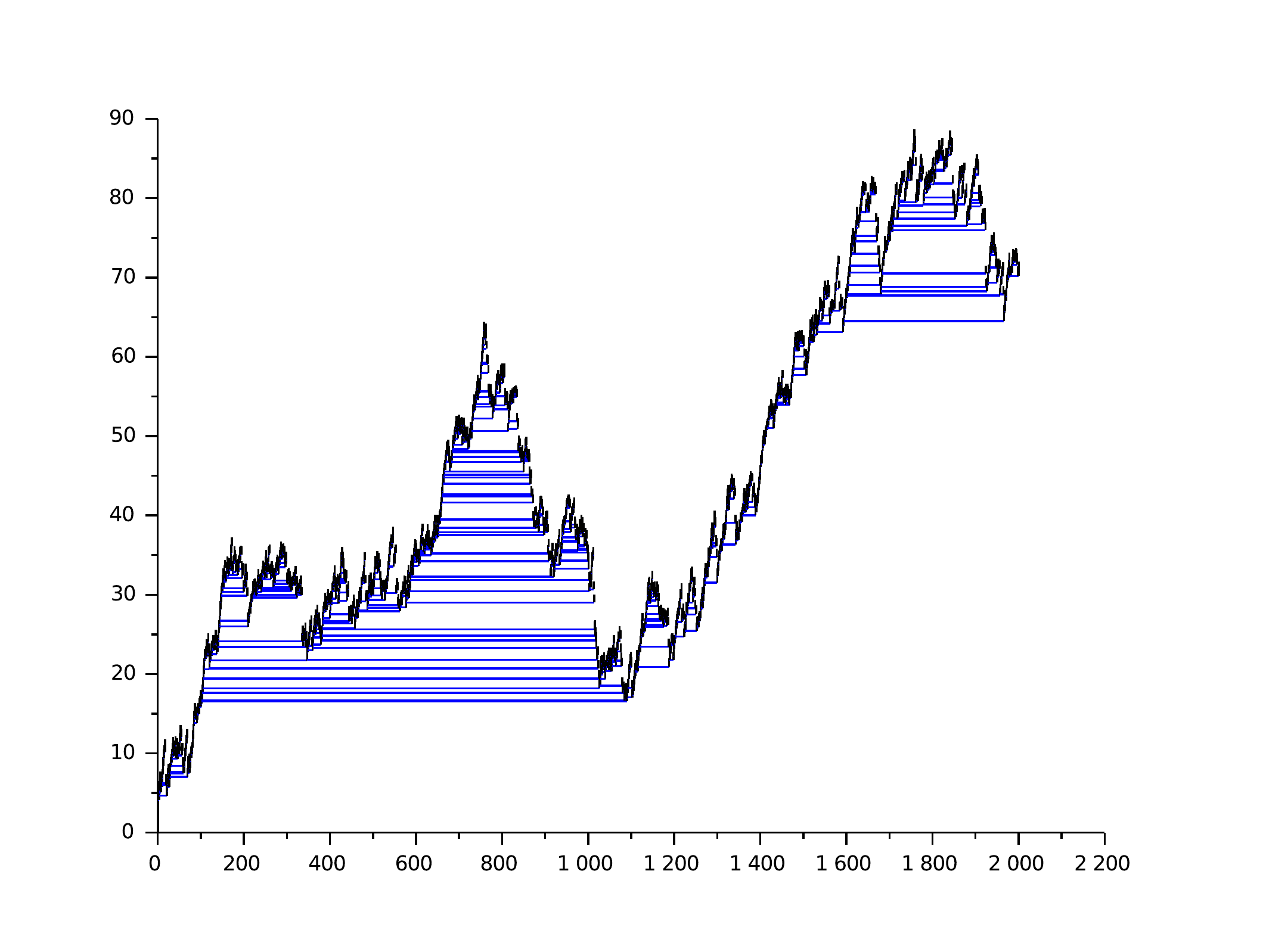}}
\end{center}
\caption{Sampled Inhomogeneous-time Splitting trees with parameters $b:t\mapsto 1$ and $K:(t,A) \mapsto \mathbf{1}_{\{1+1/t \in A \}}$ under $\mathbb{P}_{5}$. The $y-$axe represents the biologic time and the $x-$axe correspond to the individuals. The vertical black lines are the life-time interval of each individual $[B_\sigma, A_\sigma]$ and the horizontal blue lines connect parent and child.   These trees are infinite, and we only represent the first $N$ explored individuals in the sense of Section~\ref{ssec:construction}.}%.
\label{fig:Bessel-intro}
\end{figure}

Finally, since it was an historical application for the contour of trees \cite{aldousBTree}, we investigate some scaling limits of IST. Firstly, note that the JCCP, as previously and roughly defined, cannot be used to describe continuous (or $\mathbb{R}-$) trees (as defined in \cite{duqu,LeGallTree}). One has to accelerate the reading pace. A tree (continuous or not) embedded with a reading pace is called a TOM tree in \cite{lambertTOM} and the contour process set up a one-to-one map between TOM trees and c\`ad-l\`ag functions with no negative jumps (see \cite[Theorem 1]{lambertTOM}). This correspondence is continuous (see Section~\ref{sec:scaling} for topology details) and it is then enough to consider scaling limits of Markov processes generated by \eqref{eq:genintro}. It is simple to see that we can deduce (almost) any Feller Markov process from a scaling limit of such processes, thus we only answer to two simple and natural questions: fixing $b$ and $K$, when the discrete tree looks like a continuous one (by reading it fast)? Can we observe continuous tree without accelerating the reading pace?

Unfortunately, the answer to the second question is negative and it is proved in Theorem~\ref{prop:noscaling}. For the first one, surprisingly, one can not obtain any continuous tree by reading fast some IST. The only possible limit that we find is the Bessel tree in case of asymptotically critical IST . One example is illustrated in Figure~\ref{fig:Bessel-intro}. Our result reads:

\begin{thm}
\label{th:Bessel-intro}
Let $b,K$ be the birth rate and death kernel of an asymptotically critical IST $\mathbb{T}$ (as defined in Assumption~\ref{hyp:regul}). Set, for all $n\geq 1$, 
$$
\mathbb{T}_n = \{(\delta, t) \ | \ (\delta, t\sqrt{n}) \in \mathbb{T} \},
$$
embedded with the reading pace $(d_n, \lambda_n)$ given by \eqref{eq:scaledl}. Then the sequence $(\mathbb{T}_n, d_n, \lambda_n)$ converges for the Gromov-Hausdorff-Prokhorov to a $\mathbb{R}-$tree whose contour is a Bessel process absorbed at $0$; that is a Markov process with generator $A$ given by
$$
A f(x) = \frac{c}{x} f'(x) + \frac{1}{2} f''(x),
$$
for some $c\in \mathbb{R}$ and all $x>0$ and $f\in C_c^\infty(\mathbb{R}_+)$. 
\end{thm}

Again this type of convergence results is in general very difficult to prove for general non-markovian tree \cite{LSZ13,SS15}. The Markov property of the contour makes easier the proof of the previous result even it is not trivial. Indeed, due to the possible absorption, some schemes of usual proof are not possible and the previous proof is based on sharp functional results.\\

\textbf{Outline.} In Section \ref{sec:model}, we introduce the construction of the inhomogeneous splitting trees. Although this construction is quite standard, the notations of this section will be used in the sequel. Section \ref{sec:JCCP} is devoted to the study of the JCCP of the IST. In particular, Subsection \ref{ssec:construction} is a remainder of \cite{Lambert2010,lambertTOM} on the construction of the JCCP. In Subsection \ref{ssec:markov}, we show that the contour process is Markov, has the Feller property and derive its generator. This is our main result and the cornerstone of our others results. Subsection \ref{ssec:jumpProc} gives another construction of the contour seen as a stand-alone process. In Section \ref{sec:scale}, we introduce a notion of scale functions for the contour allowing to study many properties of the tree. In Section \ref{sec:lyap}, we introduce a notion of criticality for the tree and study it through the contour processes using Lyapounov drift conditions.  In the last Section~\ref{sec:scaling}, we look at scaling limit of such trees using convergence results on Markov processes.

Finally, as we use different notions of generators in the paper, we write an appendix on these notions in the end of the paper which at least defines properly the terms and notations we use.

\section{Description of the model}
\label{sec:model}
In this section, we introduce the construction of the time-inhomogeneous population model. Although this construction is classical, we recall it  for sake of presentation.  This model is described by a branching tree, where individuals live and reproduce independently from each other. The main point is that the birth-rate and the lifetime distribution of the individuals are time-dependent. More precisely, we suppose that the birth-rate is given by some measurable function $$\b{t}\in\mathbb{R}_{+}\mapsto b(\b{t})\in\mathbb{R}_{+}$$ and the lifetime distribution of the individual is given by a Markov kernel $$\b{t}\in\mathbb{R}_{+}\mapsto K(\b{t},dx)\in\mathcal{M}_{1}(\overline{\mathbb{R}}^{\ast}_{+}),$$ where $\mathcal{M}_{1}(\overline{\mathbb{R}}^{\ast}_{+})$ denotes the space of probability measure on the extended positive real line. Such model can be called inhomogeneous splitting tree (IST) in reference to the splitting trees of \cite{Lambert2010}.
Although different from Jagers-Nermann general branching processes \cite{Jagers}, IST are constructed in a similar way. However, contrary to Jagers-Nermann general branching processes, ISTs do not present renewal structures, which is the core point of the study of Jagers-Nermann branching processes.  

To give the construction of ISTs, let us introduce some notations. In the following, we denote by $\mathcal{U}$ the so-called Ulam-Harris-Neveu set, that is
\[
\mathcal{U}=\bigcup_{n\geq 0}\mathbb{N}^{n}.
\]
This set is meant to label individuals in a genealogically consistent way and to describe the discrete genealogy of the population. For two elements $v=(v_{1},\dots,v_{n})$ and $w=(w_{1},\dots,w_{m})$ in $ \mathcal{U}$, let $v.w$ be the concatenation of $v$ and $w$, that is $v.w=(v_{1},\dots,v_{n},w_{1},\dots,w_{n})$ of $\mathcal{U}$. In addition, for any positive integer $k$, we denote $v^{k}=(v_{1},\dots,v_{k})$ with the convention that $v^{k}=v$ if $k\geq n$. 

As for standard splitting trees, ISTs are random chronological trees. A chronological tree $\mathbb{T}$ is defined as a subset of $\, \mathcal{U}\times \mathbb{R}_{+}$ such that $(\sigma,\b{t})\in\mathbb{T}$ if and only if the $\sigma$-labeled individual is alive at time $\b{t}$. Hence, every individual $v\in\mathcal{U}$ in the tree is formalized by a subset of the form $\{v \}\times(\b{a},\b{b}]$ where $\b{a}$ is its birthdate and $\b{b}$ is its deathdate. From this remark, it is clear that a subset of $\,\mathcal{U}\times \mathbb{R}_{+}$ must satisfy some properties to be admissible as a chronological tree. We do not recall these properties which are well-known and refer the interested reader to the paper of Lambert \cite{Lambert2010} for more details. In the sequel, $P_{\mathcal{U}}$ (resp. $P_{\mathbb{R}_+}$) denotes the canonical projection of $\mathcal{U}\times\mathbb{R}_{+}$ to $\mathcal{U}$ (resp.\ $\mathbb{R}_+$). In particular, for a chronological tree $\mathbb{T}$, $P_{\mathcal{U}}(\mathbb{T})$ gives the discrete genealogy of $\mathbb{T}$. 

\bigskip

We now give the construction of ISTs. Let $(\mathcal{N}_{u})_{u\in\mathcal{U}}$ be an i.i.d.\ family of Poisson random measures with common intensity $b(t)\, dt$, where $dt$ refers to the Lebesgue measure on $\mathbb{R}_{+}$. We recursively define an increasing (for the inclusion order) of sequence trees.
Let
\[
\left\{
\begin{array}{l}
T_{1}=\{\xi_{\emptyset} \}\times(0,\b{\xi_{\emptyset}}] ,\\
B_{\emptyset}=0
\end{array}
\right.
\]
where $\xi_{\emptyset}$ is some random variable whose distribution is given later.
Now defines recursively,
\[
T_{n}=\bigcup_{v\in P_{\mathcal{G}}(\mathbb{T}_{n-1})}\bigcup_{i\geq 1}^{\mathcal{N}_{v}(B_{v},B_{v}+\xi_{v}]}\{v.i \}\times(B_{v.i},B_{v.i}+\xi_{(v,i)}],\\
\]
with
\[
B_{v.1}=\inf\{t>0\mid \mathcal{N}_{v}(B_{v},B_{v}+t] >0\},\quad \text{if }\mathcal{N}_{v}(B_{v},B_{v}+\xi_{v}]>0,
\]
\[
B_{v.i}=\inf\{t>B_{(v,i-1)}\mid \mathcal{N}_{v}(B_{v},B_{v}+t]>i \},\quad 2\leq i\leq \mathcal{N}_{\emptyset}(0,\xi_{\emptyset}],
\]
and $
\xi_{v.i}
$ is a random variable with conditional distribution with respect to $B_{(v,i)}$ given by $K(B(v,i),dx)$, assuming, in addition, that $
\xi_{v.i}
$  independent from any other random quantities except from $\mathcal{N}_{v}$.
Finally, the IST is given by 
\[
\mathbb{T}=\bigcup_{n\geq 1}T_{n}.
\]

Let us highlight that, for any $\sigma\in\mathbb{P}_{\mathcal{U}}(\mathbb{T})$, $B_{\sigma}$ refers to the birthdate of individual $\sigma$. Similarly, for any $\sigma\in\mathbb{P}_{\mathcal{U}}(\mathbb{T})$, let us set
\[
A_{\sigma}=B_{\sigma}+\xi_{\sigma},
\]
the death-date of individual $\sigma$.
 
This procedure defines a probability measure on the space of chronological trees. In the sequel, we consider tree starting from an ancestor whose lifetime is a fixed real number. For this reason, we denote by $\mathbb{P}_{x}$ the probability distribution on the space of chronological trees satisfying $\mathbb{P}_{x}(\xi_{\emptyset}=x)=1$. 
 
Before ending this section, let us make our general assumptions in order to have a well-behaved contour process. 

%\bigskip
\noindent
\begin{assu}[Standing Assumptions]%[Regularity assumptions]
\label{ass:b,K} \
\begin{enumerate}
\item Function $x\mapsto b(x)$ is a measurable and locally bounded function.
\item Function $x\mapsto K(x,dy)$ is weakly continuous in $\mathcal{M}_{1}(\overline{\mathbb{R}}^{\ast}_{+})$; namely for every continuous and bounded function $f$, $x\mapsto Kf(x)=\int_{\mathbb{R}_+} f(y) K(x,dy)$ is continuous. 
\end{enumerate}
\end{assu}

This assumption is only a sufficient condition to build our IST; it is not a necessary one.

\section{The contour process of an inhomogeneous splitting tree}

In this section, we introduce and characterize the contour process. Subsection~\ref{ssec:construction} is devoted to the definition of the JCCP and is essentially a remainder of \cite{Lambert2010}. We establish these properties and characterize it in case of IST in Subsection~\ref{ssec:markov}. Finally, we give another construction of a process having the same law in subsection~\ref{ssec:jumpProc}.

\label{sec:JCCP}
\subsection{Construction of the contour process}
\label{ssec:construction}
In this section, we briefly recall some facts on the construction of the JCCP which is due to Lambert in \cite{Lambert2010}. Recently, in \cite{lambertTOM}, Lambert and Uribe Bravo endowed this construction in the more general framework of TOM trees (totally ordered measured trees) which we briefly recall here because it is used in Section \ref{sec:scaling}. A TOM tree $(\mathcal{T},\leq,\lambda)$ is a real tree (see e.g. \cite{duqu,LeGallTree} for the definition) equipped with a total order $\leq$ and a $\sigma$-finite measure $\lambda$ satisfying some properties (it is locally bounded, diffuse and charges all non-empty intervals, see \cite{lambertTOM} for details).
 Given the order relation $\leq$, one can define the left of an element $x$ of $\mathcal{T}$ by
 \[
 L_{x}=\left\{y\in \mathcal{T} \mid y\leq x \right\},
 \]
 and the inverse exploration process by
\[
\begin{array}{llll}
\varphi^{-1}:&\mathcal{T} & \to & \mathbb{R}_{+}, \\
 &x & \mapsto & \lambda\left(\{y  \mid  \ y\leq x\} \right).
\end{array}
\]
As explain in \cite{lambertTOM}, the exploration process $\varphi$ is the unique c\`ad-l\`ag extension of the generalized inverse of $\varphi^{-1}$. 
  This construction is used in the next section in order to characterize the law of this process when the underlying chronological tree is an IST.
  
  \bigskip
 
In our particular case, we use the formalism introduced in \cite{Lambert2010} which endows the trees with a total order and measured structure. 
Let $\mathbb{T}$ be some chronological tree. In \cite{Lambert2010}, Lambert introduces on $\mathbb{T}$ a measure $\lambda$ and a total order relation $\leq$ which can be summarized as follows: for two elements $(\sigma,t)$ and $(\delta,s)$, we say that $(\sigma,t)\leq (\delta,s)$ if and only if $(\sigma,t)$ and $(\delta,s)$ satisfy one of the two following condition (see also Figure \ref{fig:order}):
\[
\left\{
\begin{array}{ll}
\delta\preceq \sigma \text{ and }  P_{\mathbb{R}}((\delta,s)\wedge(\sigma,t))\geq s& \quad (C1)\\
\text{or}\\
\exists n\in\mathbb{N},\ \sigma^{n}\preceq \delta \text{ and } t> B_{\sigma^{n}},& \quad (C2)
\end{array}
\right.
\]
where $\preceq$ is the classical order relation on $\mathcal{U}$. Informally speaking, the measure $\lambda$ can be thought as the length of the segments in the tree.
 We refer the reader to \cite[Section 2]{Lambert2010} for more details. In the sequel, we set
 \begin{equation}
 \label{eq:defLEngth}
 \mathcal{L}(\mathbb{T}):=\lambda(\mathbb{T})\text{ and }  \mathcal{H}(\mathbb{T}):=\sup P_{\mathbb{R}}(\mathbb{T}),
 \end{equation}
 which are respectively the total length and the height of the tree $\mathbb{T}$.
 In particular,  $\mathcal{H}(\mathbb{T})\leq \mathcal{L}(\mathbb{T})$ and $\mathcal{H}(\mathbb{T})<\infty$ means that the population gets extinct in finite time. This event is denoted by
 \begin{equation}
\label{eq:extEvent}
 \text{Ext}:=\{\mathcal{H}(\mathbb{T})<\infty \}.
\end{equation}
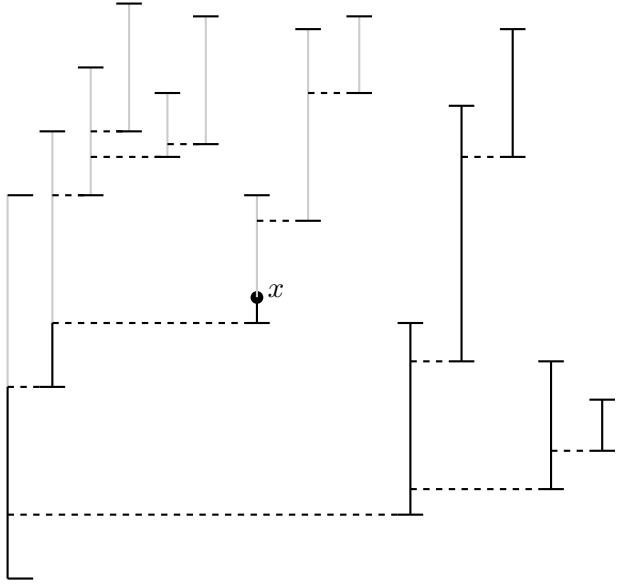
\begin{figure}

\unitlength 2mm % = 4.55pt
\linethickness{0.4pt}

%TeXCAD (http://texcad.sf.net/) Picture. File: [smalltree.tex]. Options on following lines.
%\grade{\on}
%\emlines{\off}
%\epic{\off}
%\beziermacro{\on}
%\reduce{\on}
%\snapping{\off}
%\quality{8.000}
%\graddiff{0.005}
%\snapasp{1}
%\zoom{8.0000}
\unitlength 1.7mm % = 2.845pt
\linethickness{0.8pt}
\begin{picture}(15,45)(-30,0)

\put(0,0){\line(0,1){15}}%branche racine
\color{gray5}
\put(0,15){\line(0,1){15}}
\color{black}
\put(0,0){\line(1,0){2}}%pied racine
\put(0,30){\line(1,0){2}}%tete racine
%R2
\multiput(0,15)(1,0){3}{\line(1,0){0.5}}%r to r2  2.5  

\put(3.5,15){\line(0,1){5}}% r2 15>35
\color{gray5}
\put(3.5,20){\line(0,1){15}}
\color{black}
\put(2.5,15){\line(1,0){2}}%pied r2
\put(2.5,35){\line(1,0){2}}%tete r2
%R22
\multiput(3.5,30)(1,0){3}{\line(1,0){0.5}}%r2 to r22  l=3.5 -> 6.5
\color{gray5}
\put(6.5,30){\line(0,1){10}}% r22 30>40
\color{black}
\put(5.5,30){\line(1,0){2}}%pied r22
\put(5.5,40){\line(1,0){2}}%tete r22
%R222
\multiput(6.5,35)(1,0){3}{\line(1,0){0.5}}%r22 to r222  l=6.5 -> 8.5
\color{gray5}
\put(9.5,35){\line(0,1){10}}%r222 35>45
\color{black}
\put(8.5,35){\line(1,0){2}}%pied r22
\put(8.5,45){\line(1,0){2}}%tete r22
%R221
\multiput(6.5,33)(1,0){6}{\line(1,0){0.5}}%r22 to r221  l=6.5 -> 12.5
\color{gray5}
\put(12.5,33){\line(0,1){5}}%r221 33>38
\color{black}
\put(11.5,33){\line(1,0){2}}%pied r22
\put(11.5,38){\line(1,0){2}}%tete r22
%R2211
\multiput(12.5,34)(1,0){3}{\line(1,0){0.5}}%r221 to r2212  l=12.5 -> 14.5
\color{gray5}
\put(15.5,34){\line(0,1){10}}%r221 33>38
\color{black}
\put(14.5,34){\line(1,0){2}}%pied r22
\put(14.5,44){\line(1,0){2}}%tete r22
%R21
\put(19.5,22){\circle*{1}}
\put(20.5,22){\makebox(1,1){$x$}}
\multiput(3.5,20)(1,0){15}{\line(1,0){0.5}}%r2 to r21  l=3.5 -> 21.5
\color{gray5}
\put(19.5,22){\line(0,1){8}}%r21 20>30
\color{black}
\put(19.5,20){\line(0,1){2}}%r21 20>30

\put(18.5,20){\line(1,0){2}}%pied r21
\put(18.5,30){\line(1,0){2}}%tete r21
%R211
\multiput(19.5,28)(1,0){3}{\line(1,0){0.5}}%r21 to r211  l=3.5 -> 18.5
\color{gray5}
\put(23.5,28){\line(0,1){15}}%r211 28>43
\color{black}
\put(22.5,28){\line(1,0){2}}%pied r21
\put(22.5,43){\line(1,0){2}}%tete r21
%R2111
\multiput(23.5,38)(1,0){3}{\line(1,0){0.5}}%r211 to r2111  l=3.5 -> 18.5
\color{gray5}
\put(27.5,38){\line(0,1){6}}%r21 33>38
\color{black}
\put(26.5,38){\line(1,0){2}}%pied r21
\put(26.5,44){\line(1,0){2}}%tete r21
%R1
\multiput(0,5)(1,0){31}{\line(1,0){0.5}}%r to r1  l=3.5 -> 18.5
\put(31.5,5){\line(0,1){15}}%r1 33>38
\put(30.5,5){\line(1,0){2}}%pied r1
\put(30.5,20){\line(1,0){2}}%tete r1
%R12
\multiput(31.5,17)(1,0){3}{\line(1,0){0.5}}%r1 to r12  l=3.5 -> 18.5
\put(35.5,17){\line(0,1){20}}%r21 17>37
\put(34.5,17){\line(1,0){2}}%pied r21
\put(34.5,37){\line(1,0){2}}%tete r21
%R121
\multiput(35.5,33)(1,0){3}{\line(1,0){0.5}}%r211 to r2111  l=3.5 -> 18.5
\put(39.5,33){\line(0,1){10}}%r21 5>20
\put(38.5,33){\line(1,0){2}}%pied r21
\put(38.5,43){\line(1,0){2}}%tete r21
%R11
\multiput(31.5,7)(1,0){10}{\line(1,0){0.5}}%r1 to r11  l=3.5 -> 18.5
\put(42.5,7){\line(0,1){10}}%r21 5>20
\put(41.5,7){\line(1,0){2}}%pied r21
\put(41.5,17){\line(1,0){2}}%tete r21
%R111
\multiput(42.5,10)(1,0){3}{\line(1,0){0.5}}%r1 to r11  l=3.5 -> 18.5
\put(46.5,10){\line(0,1){4}}%r21 5>20
\put(45.5,10){\line(1,0){2}}%pied r21
\put(45.5,14){\line(1,0){2}}%tete r21
\end{picture}
\caption{A tree $\mathbb{T}$, with in gray the set $\left\{y\in\mathbb{T}\mid y\leq x \right\}$ and in black its complementary. } 
\label{fig:order}
\end{figure}
Let us now recall an important result concerning finite chronological tree.
\begin{prop}[Theorem 3.1 of \cite{Lambert2010}]
 Let $\mathbb{T}$ be a chronological tree such that $\mathcal{L}(\mathbb{T})<\infty$. Then the function $(\mathcal{E}_{s}(\mathbb{T}))_{ s\in[0,\mathcal{L}(\mathbb{T})]}$ defined by
 \[
 \mathcal{E}_{s}(\mathbb{T})=\inf\left\{x\in\mathbb{T}\mid \lambda\left(\{y\in\mathbb{T}\mid y\leq x \} \right)\geq s\right\}, \quad s\in[0,\mathcal{L}(\mathbb{T})].
 \]
 is an increasing bijection from $[0,\mathcal{L}(\mathbb{T})]$ to $\mathbb{T}$. This process is called the exploration process of $\mathbb{T}$ (see Figure \ref{fig:explorationProcess}).
\end{prop}

An important consequence of the construction of the contour process which is used in the sequel is that, for any point $(\delta,t)$ of $\mathbb{T}$,
\begin{equation}
\label{eq:inclu}
\left\{ \mathcal{E}_{s}(\mathbb{T}) \mid s\leq \mathcal{E}_{(\delta,t)}^{-1}(\mathbb{T}) \right\}=\left\{x\in \mathbb{T}\mid x\leq (\delta, t) \right\}.
\end{equation}
In other words, the part of the tree which is explored by the exploration process up to the exploration of point $(\delta,t)$ is exactly the left of $(\delta,t)$.
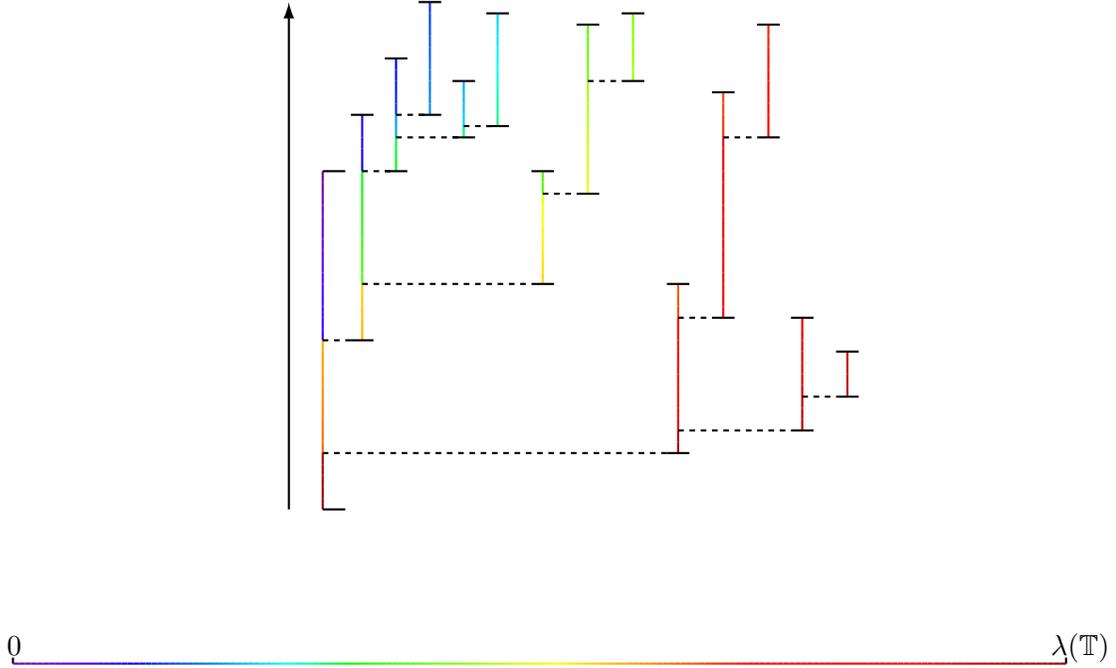
\begin{figure}[ht]
\begin{center}	
\unitlength 1.5mm % = 2.845pt
\linethickness{0.8pt}
\begin{picture}(250,50)(-35,0)
\put(0,0){\vector(0,1){45}}
\put(3,0){
\newcount\compteur \compteur=400
\multiput(0,30)(0,-1){15}{\color[wave]{\the\compteur}\global\advance\compteur by 2\line(0,-1){1}}
\multiput(3.5,35)(0,-1){5}{\color[wave]{\the\compteur}\global\advance\compteur by 2\line(0,-1){1}}
\multiput(6.5,40)(0,-1){5}{\color[wave]{\the\compteur}\global\advance\compteur by 2\line(0,-1){1}}
\multiput(9.5,45)(0,-1){10}{\color[wave]{\the\compteur}\global\advance\compteur by 2\line(0,-1){1}}
\multiput(6.5,35)(0,-1){2}{\color[wave]{\the\compteur}\global\advance\compteur by 2\line(0,-1){1}}
\multiput(12.5,38)(0,-1){4}{\color[wave]{\the\compteur}\global\advance\compteur by 2\line(0,-1){1}}
\multiput(15.5,44)(0,-1){10}{\color[wave]{\the\compteur}\global\advance\compteur by 2\line(0,-1){1}}
\multiput(12.5,34)(0,-1){1}{\color[wave]{\the\compteur}\global\advance\compteur by 2\line(0,-1){1}}
\multiput(6.5,33)(0,-1){3}{\color[wave]{\the\compteur}\global\advance\compteur by 2\line(0,-1){1}}
\multiput(3.5,30)(0,-1){10}{\color[wave]{\the\compteur}\global\advance\compteur by 2\line(0,-1){1}}
\multiput(19.5,30)(0,-1){2}{\color[wave]{\the\compteur}\global\advance\compteur by 2\line(0,-1){1}}
\multiput(23.5,43)(0,-1){5}{\color[wave]{\the\compteur}\global\advance\compteur by 2\line(0,-1){1}}
\multiput(27.5,44)(0,-1){6}{\color[wave]{\the\compteur}\global\advance\compteur by 2\line(0,-1){1}}
\multiput(23.5,38)(0,-1){10}{\color[wave]{\the\compteur}\global\advance\compteur by 2\line(0,-1){1}}
\multiput(19.5,28)(0,-1){8}{\color[wave]{\the\compteur}\global\advance\compteur by 2\line(0,-1){1}}
\multiput(3.5,20)(0,-1){5}{\color[wave]{\the\compteur}\global\advance\compteur by 2\line(0,-1){1}}
\multiput(0,15)(0,-1){10}{\color[wave]{\the\compteur}\global\advance\compteur by 2\line(0,-1){1}}
\multiput(31.5,20)(0,-1){3}{\color[wave]{\the\compteur}\global\advance\compteur by 2\line(0,-1){1}}
\multiput(35.5,37)(0,-1){4}{\color[wave]{\the\compteur}\global\advance\compteur by 2\line(0,-1){1}}
\multiput(39.5,43)(0,-1){10}{\color[wave]{\the\compteur}\global\advance\compteur by 2\line(0,-1){1}}
\multiput(35.5,33)(0,-1){16}{\color[wave]{\the\compteur}\global\advance\compteur by 2\line(0,-1){1}}
\multiput(31.5,17)(0,-1){10}{\color[wave]{\the\compteur}\global\advance\compteur by 2\line(0,-1){1}}
\multiput(42.5,17)(0,-1){7}{\color[wave]{\the\compteur}\global\advance\compteur by 2\line(0,-1){1}}
\multiput(46.5,14)(0,-1){4}{\color[wave]{\the\compteur}\global\advance\compteur by 2\line(0,-1){1}}
\multiput(42.5,10)(0,-1){3}{\color[wave]{\the\compteur}\global\advance\compteur by 2\line(0,-1){1}}
\multiput(31.5,7)(0,-1){2}{\color[wave]{\the\compteur}\global\advance\compteur by 2\line(0,-1){1}}
\multiput(0,5)(0,-1){5}{\color[wave]{\the\compteur}\global\advance\compteur by 2\line(0,-1){1}}
%\put(0,0){\line(0,1){30}}%branche racine
\put(0,0){\line(1,0){2}}%pied racine
\put(0,30){\line(1,0){2}}%tete racine
%R2
\multiput(0,15)(1,0){3}{\line(1,0){0.5}}%r to r2  2.5  
%\put(3.5,15){\line(0,1){20}}% r2 15>35
\put(2.5,15){\line(1,0){2}}%pied r2
\put(2.5,35){\line(1,0){2}}%tete r2
%R22
\multiput(3.5,30)(1,0){3}{\line(1,0){0.5}}%r2 to r22  l=3.5 -> 6.5
%\put(6.5,30){\line(0,1){10}}% r22 30>40
\put(5.5,30){\line(1,0){2}}%pied r22
\put(5.5,40){\line(1,0){2}}%tete r22
%R222
\multiput(6.5,35)(1,0){3}{\line(1,0){0.5}}%r22 to r222  l=6.5 -> 8.5
%\put(9.5,35){\line(0,1){10}}%r222 35>45
\put(8.5,35){\line(1,0){2}}%pied r22
\put(8.5,45){\line(1,0){2}}%tete r22
%R221
\multiput(6.5,33)(1,0){6}{\line(1,0){0.5}}%r22 to r221  l=6.5 -> 12.5
%\put(12.5,33){\line(0,1){5}}%r221 33>38
\put(11.5,33){\line(1,0){2}}%pied r22
\put(11.5,38){\line(1,0){2}}%tete r22
%R2211
\multiput(12.5,34)(1,0){3}{\line(1,0){0.5}}%r221 to r2212  l=12.5 -> 14.5
%\put(15.5,34){\line(0,1){10}}%r221 33>38
\put(14.5,34){\line(1,0){2}}%pied r22
\put(14.5,44){\line(1,0){2}}%tete r22
%R21
\multiput(3.5,20)(1,0){15}{\line(1,0){0.5}}%r2 to r21  l=3.5 -> 21.5
%\put(19.5,20){\line(0,1){10}}%r21 20>30
\put(18.5,20){\line(1,0){2}}%pied r21
\put(18.5,30){\line(1,0){2}}%tete r21
%R211
\multiput(19.5,28)(1,0){3}{\line(1,0){0.5}}%r21 to r211  l=3.5 -> 18.5
%\put(23.5,28){\line(0,1){15}}%r211 28>43
\put(22.5,28){\line(1,0){2}}%pied r21
\put(22.5,43){\line(1,0){2}}%tete r21
%R2111
\multiput(23.5,38)(1,0){3}{\line(1,0){0.5}}%r211 to r2111  l=3.5 -> 18.5
%\put(27.5,38){\line(0,1){6}}%r21 33>38
\put(26.5,38){\line(1,0){2}}%pied r21
\put(26.5,44){\line(1,0){2}}%tete r21
%R1
\multiput(0,5)(1,0){31}{\line(1,0){0.5}}%r to r1  l=3.5 -> 18.5
%\put(31.5,5){\line(0,1){15}}%r1 33>38
\put(30.5,5){\line(1,0){2}}%pied r1
\put(30.5,20){\line(1,0){2}}%tete r1
%R12
\multiput(31.5,17)(1,0){3}{\line(1,0){0.5}}%r1 to r12  l=3.5 -> 18.5
%\put(35.5,17){\line(0,1){20}}%r21 17>37
\put(34.5,17){\line(1,0){2}}%pied r21
\put(34.5,37){\line(1,0){2}}%tete r21
%R121
\multiput(35.5,33)(1,0){3}{\line(1,0){0.5}}%r211 to r2111  l=3.5 -> 18.5
%\put(39.5,33){\line(0,1){10}}%r21 5>20
\put(38.5,33){\line(1,0){2}}%pied r21
\put(38.5,43){\line(1,0){2}}%tete r21
%R11
\multiput(31.5,7)(1,0){10}{\line(1,0){0.5}}%r1 to r11  l=3.5 -> 18.5
%\put(42.5,7){\line(0,1){10}}%r21 5>20
\put(41.5,7){\line(1,0){2}}%pied r21
\put(41.5,17){\line(1,0){2}}%tete r21
%R111
\multiput(42.5,10)(1,0){3}{\line(1,0){0.5}}%r1 to r11  l=3.5 -> 18.5
%\put(46.5,10){\line(0,1){4}}%r21 5>20
\put(45.5,10){\line(1,0){2}}%pied r21
\put(45.5,14){\line(1,0){2}}%tete r21
}
\end{picture}
\unitlength 0.8mm

\vspace{1cm}

\begin{picture}(170,2)(0,0)
\put(0,0){\line(0,1){1}}
\put(-1,1.5){\makebox{0}}
\put(175,0){\line(0,1){1}}
\put(172.5,1.5){\makebox{$\lambda(\mathbb{T})$}}
\newcount\compteur \compteur=400
\multiput(0,0)(1,0){15}{\color[wave]{\the\compteur}\global\advance\compteur by 2\line(1,0){1}}
\multiput(15,0)(1,0){5}{\color[wave]{\the\compteur}\global\advance\compteur by 2\line(1,0){1}}
\multiput(20,0)(1,0){5}{\color[wave]{\the\compteur}\global\advance\compteur by 2\line(1,0){1}}
\multiput(25,0)(1,0){10}{\color[wave]{\the\compteur}\global\advance\compteur by 2\line(1,0){1}}
\multiput(35,0)(1,0){2}{\color[wave]{\the\compteur}\global\advance\compteur by 2\line(1,0){1}}
\multiput(37,0)(1,0){4}{\color[wave]{\the\compteur}\global\advance\compteur by 2\line(1,0){1}}
\multiput(41,0)(1,0){10}{\color[wave]{\the\compteur}\global\advance\compteur by 2\line(1,0){1}}
\multiput(51,0)(1,0){1}{\color[wave]{\the\compteur}\global\advance\compteur by 2\line(1,0){1}}
\multiput(52,0)(1,0){3}{\color[wave]{\the\compteur}\global\advance\compteur by 2\line(1,0){1}}
\multiput(55,0)(1,0){10}{\color[wave]{\the\compteur}\global\advance\compteur by 2\line(1,0){1}}
\multiput(65,0)(1,0){2}{\color[wave]{\the\compteur}\global\advance\compteur by 2\line(1,0){1}}
\multiput(67,0)(1,0){5}{\color[wave]{\the\compteur}\global\advance\compteur by 2\line(1,0){1}}
\multiput(72,0)(1,0){6}{\color[wave]{\the\compteur}\global\advance\compteur by 2\line(1,0){1}}
\multiput(78,0)(1,0){10}{\color[wave]{\the\compteur}\global\advance\compteur by 2\line(1,0){1}}
\multiput(88,0)(1,0){8}{\color[wave]{\the\compteur}\global\advance\compteur by 2\line(1,0){1}}
\multiput(96,0)(1,0){5}{\color[wave]{\the\compteur}\global\advance\compteur by 2\line(1,0){1}}
\multiput(101,0)(1,0){10}{\color[wave]{\the\compteur}\global\advance\compteur by 2\line(1,0){1}}
\multiput(111,0)(1,0){3}{\color[wave]{\the\compteur}\global\advance\compteur by 2\line(1,0){1}}
\multiput(114,0)(1,0){4}{\color[wave]{\the\compteur}\global\advance\compteur by 2\line(1,0){1}}
\multiput(118,0)(1,0){10}{\color[wave]{\the\compteur}\global\advance\compteur by 2\line(1,0){1}}
\multiput(128,0)(1,0){16}{\color[wave]{\the\compteur}\global\advance\compteur by 2\line(1,0){1}}
\multiput(144,0)(1,0){10}{\color[wave]{\the\compteur}\global\advance\compteur by 2\line(1,0){1}}
\multiput(154,0)(1,0){7}{\color[wave]{\the\compteur}\global\advance\compteur by 2\line(1,0){1}}
\multiput(161,0)(1,0){4}{\color[wave]{\the\compteur}\global\advance\compteur by 2\line(1,0){1}}
\multiput(165,0)(1,0){3}{\color[wave]{\the\compteur}\global\advance\compteur by 2\line(1,0){1}}
\multiput(168,0)(1,0){2}{\color[wave]{\the\compteur}\global\advance\compteur by 2\line(1,0){1}}
\multiput(170,0)(1,0){5}{\color[wave]{\the\compteur}\global\advance\compteur by 2\line(1,0){1}}
\end{picture}
\end{center}
\caption{A graphical representation of the exploration process. The one-to-one correspondence between $[0,\lambda(\mathbb{T})]$ and $\mathbb{T}$ is represented by corresponding colors. }
\label{fig:explorationProcess}
\end{figure}

From $(\mathcal{E}_{s}(\mathbb{T}))_{ s\in[0,\mathcal{L}(\mathbb{T})]}$ the contour process is now defined by %(see also Figure~\ref{fig:explorationProcess}) %\ref{fig:coloredContour})
\[
\mathcal{C}_{s}(\mathbb{T})=P_{\mathbb{R}_{+}}(\mathcal{E}_{s}(\mathbb{T})) \mathds{1}_{s\leq \mathcal{L}(\mathbb{T})},\quad s\in \mathbb{R}_{+}.
\]
This means that the contour process is the height in the tree of the exploration process at a given time until the exploration process hits the point $(\emptyset,0)$ at time $\mathcal{L}(\mathbb{T})$. After time $\mathcal{L}(\mathbb{T})$, the contour remains equal to $0$.

\begin{rem}[Truncated tree]
	\label{rem:trunc}
In the following, we consider trees $\mathbb{T}$ with infinite total length (i.e.\ $\mathcal{L}(\mathbb{T})=\infty)$). For such tree, one cannot define properly the contour process. To avoid this technicality, we consider truncated trees above some threshold associated to tree $\mathbb{T}$. More precisely, for some time $T$, the truncated tree $\mathbb{T}^T$ of $\mathbb{T}$ above level $T$ is defined by
\[
(\sigma,s)\in\mathbb{T}^{T}\Leftrightarrow \left((\sigma,s)\in\mathbb{T} \text{ and } s\leq T \right).
\]
	\end{rem}

An remarkable observation is that the time-inhomogeneity of the tree become a space in-homogeneity for the contour process (see Figure \ref{fig:coloredContour}). Hence, the process remains time-homogeneous but is not a L\'evy  process, as in the case of splitting trees, due to its space-inhomogeneity.
\begin{figure}[ht]
\unitlength 1.3mm % = 2.845pt
\linethickness{0.8pt}
\unitlength 1.4mm

\begin{picture}(250,50)(-35,0)
\put(0,0){\vector(0,1){45}}
\put(3,0){
\newcounter{step}\setcounter{step}{2}
\newcounter{height}\setcounter{height}{30}
\newcounter{base}\setcounter{base}{550}
\newcounter{compteur}
\setcounter{height}{30}
\setcounter{compteur}{\value{base}-\value{step}*\value{height}}
\multiput(0,30)(0,-1){15}{\color[wave]{\thecompteur}\addtocounter{compteur}{\thestep}\line(0,-1){1}}
\setcounter{height}{35}
\setcounter{compteur}{\value{base}-\value{step}*\value{height}}
\multiput(3.5,35)(0,-1){5}{\color[wave]{\thecompteur}\addtocounter{compteur}{\thestep}\line(0,-1){1}}
\setcounter{height}{40}
\setcounter{compteur}{\value{base}-\value{step}*\value{height}}
\multiput(6.5,40)(0,-1){5}{\color[wave]{\thecompteur}\addtocounter{compteur}{\thestep}\line(0,-1){1}}
\setcounter{height}{45}
\setcounter{compteur}{\value{base}-\value{step}*\value{height}}
\multiput(9.5,45)(0,-1){10}{\color[wave]{\thecompteur}\addtocounter{compteur}{\thestep}\line(0,-1){1}}
\setcounter{height}{35}
\setcounter{compteur}{\value{base}-\value{step}*\value{height}}
\multiput(6.5,35)(0,-1){2}{\color[wave]{\thecompteur}\addtocounter{compteur}{\thestep}\line(0,-1){1}}
\setcounter{height}{38}
\setcounter{compteur}{\value{base}-\value{step}*\value{height}}
\multiput(12.5,38)(0,-1){4}{\color[wave]{\thecompteur}\addtocounter{compteur}{\thestep}\line(0,-1){1}}
\setcounter{height}{44}
\setcounter{compteur}{\value{base}-\value{step}*\value{height}}
\multiput(15.5,44)(0,-1){10}{\color[wave]{\thecompteur}\addtocounter{compteur}{\thestep}\line(0,-1){1}}
\setcounter{height}{34}
\setcounter{compteur}{\value{base}-\value{step}*\value{height}}
\multiput(12.5,34)(0,-1){1}{\color[wave]{\thecompteur}\addtocounter{compteur}{\thestep}\line(0,-1){1}}
\setcounter{height}{33}
\setcounter{compteur}{\value{base}-\value{step}*\value{height}}
\multiput(6.5,33)(0,-1){3}{\color[wave]{\thecompteur}\addtocounter{compteur}{\thestep}\line(0,-1){1}}
\setcounter{height}{30}
\setcounter{compteur}{\value{base}-\value{step}*\value{height}}
\multiput(3.5,30)(0,-1){10}{\color[wave]{\thecompteur}\addtocounter{compteur}{\thestep}\line(0,-1){1}}
\setcounter{height}{30}
\setcounter{compteur}{\value{base}-\value{step}*\value{height}}
\multiput(19.5,30)(0,-1){2}{\color[wave]{\thecompteur}\addtocounter{compteur}{\thestep}\line(0,-1){1}}
\setcounter{height}{43}
\setcounter{compteur}{\value{base}-\value{step}*\value{height}}
\multiput(23.5,43)(0,-1){5}{\color[wave]{\thecompteur}\addtocounter{compteur}{\thestep}\line(0,-1){1}}
\setcounter{height}{44}
\setcounter{compteur}{\value{base}-\value{step}*\value{height}}
\multiput(27.5,44)(0,-1){6}{\color[wave]{\thecompteur}\addtocounter{compteur}{\thestep}\line(0,-1){1}}
\setcounter{height}{38}
\setcounter{compteur}{\value{base}-\value{step}*\value{height}}
\multiput(23.5,38)(0,-1){10}{\color[wave]{\thecompteur}\addtocounter{compteur}{\thestep}\line(0,-1){1}}
\setcounter{height}{28}
\setcounter{compteur}{\value{base}-\value{step}*\value{height}}
\multiput(19.5,28)(0,-1){8}{\color[wave]{\thecompteur}\addtocounter{compteur}{\thestep}\line(0,-1){1}}
\setcounter{height}{20}
\setcounter{compteur}{\value{base}-\value{step}*\value{height}}
\multiput(3.5,20)(0,-1){5}{\color[wave]{\thecompteur}\addtocounter{compteur}{\thestep}\line(0,-1){1}}
\setcounter{height}{15}
\setcounter{compteur}{\value{base}-\value{step}*\value{height}}
\multiput(0,15)(0,-1){10}{\color[wave]{\thecompteur}\addtocounter{compteur}{\thestep}\line(0,-1){1}}
\setcounter{height}{20}
\setcounter{compteur}{\value{base}-\value{step}*\value{height}}
\multiput(31.5,20)(0,-1){3}{\color[wave]{\thecompteur}\addtocounter{compteur}{\thestep}\line(0,-1){1}}
\setcounter{height}{37}
\setcounter{compteur}{\value{base}-\value{step}*\value{height}}
\multiput(35.5,37)(0,-1){4}{\color[wave]{\thecompteur}\addtocounter{compteur}{\thestep}\line(0,-1){1}}
\setcounter{height}{43}
\setcounter{compteur}{\value{base}-\value{step}*\value{height}}
\multiput(39.5,43)(0,-1){10}{\color[wave]{\thecompteur}\addtocounter{compteur}{\thestep}\line(0,-1){1}}
\setcounter{height}{33}
\setcounter{compteur}{\value{base}-\value{step}*\value{height}}
\multiput(35.5,33)(0,-1){16}{\color[wave]{\thecompteur}\addtocounter{compteur}{\thestep}\line(0,-1){1}}
\setcounter{height}{17}
\setcounter{compteur}{\value{base}-\value{step}*\value{height}}
\multiput(31.5,17)(0,-1){10}{\color[wave]{\thecompteur}\addtocounter{compteur}{\thestep}\line(0,-1){1}}
\setcounter{height}{17}
\setcounter{compteur}{\value{base}-\value{step}*\value{height}}
\multiput(42.5,17)(0,-1){7}{\color[wave]{\thecompteur}\addtocounter{compteur}{\thestep}\line(0,-1){1}}
\setcounter{height}{14}
\setcounter{compteur}{\value{base}-\value{step}*\value{height}}
\multiput(46.5,14)(0,-1){4}{\color[wave]{\thecompteur}\addtocounter{compteur}{\thestep}\line(0,-1){1}}
\setcounter{height}{10}
\setcounter{compteur}{\value{base}-\value{step}*\value{height}}
\multiput(42.5,10)(0,-1){3}{\color[wave]{\thecompteur}\addtocounter{compteur}{\thestep}\line(0,-1){1}}
\setcounter{height}{7}
\setcounter{compteur}{\value{base}-\value{step}*\value{height}}
\multiput(31.5,7)(0,-1){2}{\color[wave]{\thecompteur}\addtocounter{compteur}{\thestep}\line(0,-1){1}}
\setcounter{height}{5}
\setcounter{compteur}{\value{base}-\value{step}*\value{height}}
\multiput(0,5)(0,-1){5}{\color[wave]{\thecompteur}\addtocounter{compteur}{\thestep}\line(0,-1){1}}
%\put(0,0){\line(0,1){30}}%branche racine
\put(0,0){\line(1,0){2}}%pied racine
\put(0,30){\line(1,0){2}}%tete racine
%R2
\multiput(0,15)(1,0){3}{\line(1,0){0.5}}%r to r2  2.5  
%\put(3.5,15){\line(0,1){20}}% r2 15>35
\put(2.5,15){\line(1,0){2}}%pied r2
\put(2.5,35){\line(1,0){2}}%tete r2
%R22
\multiput(3.5,30)(1,0){3}{\line(1,0){0.5}}%r2 to r22  l=3.5 -> 6.5
%\put(6.5,30){\line(0,1){10}}% r22 30>40
\put(5.5,30){\line(1,0){2}}%pied r22
\put(5.5,40){\line(1,0){2}}%tete r22
%R222
\multiput(6.5,35)(1,0){3}{\line(1,0){0.5}}%r22 to r222  l=6.5 -> 8.5
%\put(9.5,35){\line(0,1){10}}%r222 35>45
\put(8.5,35){\line(1,0){2}}%pied r22
\put(8.5,45){\line(1,0){2}}%tete r22
%R221
\multiput(6.5,33)(1,0){6}{\line(1,0){0.5}}%r22 to r221  l=6.5 -> 12.5
%\put(12.5,33){\line(0,1){5}}%r221 33>38
\put(11.5,33){\line(1,0){2}}%pied r22
\put(11.5,38){\line(1,0){2}}%tete r22
%R2211
\multiput(12.5,34)(1,0){3}{\line(1,0){0.5}}%r221 to r2212  l=12.5 -> 14.5
%\put(15.5,34){\line(0,1){10}}%r221 33>38
\put(14.5,34){\line(1,0){2}}%pied r22
\put(14.5,44){\line(1,0){2}}%tete r22
%R21
\multiput(3.5,20)(1,0){15}{\line(1,0){0.5}}%r2 to r21  l=3.5 -> 21.5
%\put(19.5,20){\line(0,1){10}}%r21 20>30
\put(18.5,20){\line(1,0){2}}%pied r21
\put(18.5,30){\line(1,0){2}}%tete r21
%R211
\multiput(19.5,28)(1,0){3}{\line(1,0){0.5}}%r21 to r211  l=3.5 -> 18.5
%\put(23.5,28){\line(0,1){15}}%r211 28>43
\put(22.5,28){\line(1,0){2}}%pied r21
\put(22.5,43){\line(1,0){2}}%tete r21
%R2111
\multiput(23.5,38)(1,0){3}{\line(1,0){0.5}}%r211 to r2111  l=3.5 -> 18.5
%\put(27.5,38){\line(0,1){6}}%r21 33>38
\put(26.5,38){\line(1,0){2}}%pied r21
\put(26.5,44){\line(1,0){2}}%tete r21
%R1
\multiput(0,5)(1,0){31}{\line(1,0){0.5}}%r to r1  l=3.5 -> 18.5
%\put(31.5,5){\line(0,1){15}}%r1 33>38
\put(30.5,5){\line(1,0){2}}%pied r1
\put(30.5,20){\line(1,0){2}}%tete r1
%R12
\multiput(31.5,17)(1,0){3}{\line(1,0){0.5}}%r1 to r12  l=3.5 -> 18.5
%\put(35.5,17){\line(0,1){20}}%r21 17>37
\put(34.5,17){\line(1,0){2}}%pied r21
\put(34.5,37){\line(1,0){2}}%tete r21
%R121
\multiput(35.5,33)(1,0){3}{\line(1,0){0.5}}%r211 to r2111  l=3.5 -> 18.5
%\put(39.5,33){\line(0,1){10}}%r21 5>20
\put(38.5,33){\line(1,0){2}}%pied r21
\put(38.5,43){\line(1,0){2}}%tete r21
%R11
\multiput(31.5,7)(1,0){10}{\line(1,0){0.5}}%r1 to r11  l=3.5 -> 18.5
%\put(42.5,7){\line(0,1){10}}%r21 5>20
\put(41.5,7){\line(1,0){2}}%pied r21
\put(41.5,17){\line(1,0){2}}%tete r21
%R111
\multiput(42.5,10)(1,0){3}{\line(1,0){0.5}}%r1 to r11  l=3.5 -> 18.5
%\put(46.5,10){\line(0,1){4}}%r21 5>20
\put(45.5,10){\line(1,0){2}}%pied r21
\put(45.5,14){\line(1,0){2}}%tete r21
}
\end{picture}

\unitlength 0.9mm
\vspace{1cm}
\centering
\begin{picture}(175,45)
%%%%%%%%%%%%%AXES%%%%%%%%%%%%%%%%%%%
\put(20,0){\vector(1,0){157}}
\put(0,15){\line(0,1){15}}

\put(0,0){\line(1,0){15}}

\put(15,0){\line(1,0){5}}

\put(20,0){\line(1,0){5}}

\put(25,0){\line(1,0){10}}

\put(35,0){\line(1,0){2}}
\put(0,0){\line(0,1){15}}
\put(0,30){\vector(0,1){15}}
%%%%%%%%%%%%%%%%%%%%
%R AVANT NAISSANCE
%\newcounter{step}\setcounter{step}{1}
%\newcounter{height}\setcounter{height}{30}
%\newcounter{base}\setcounter{base}{500}
%\newcounter{compteur}
\setcounter{height}{30}
\setcounter{compteur}{\value{base}-\value{step}*\value{height}}
\multiput(0,30)(1,-1){15}{\color[wave]{\thecompteur}\addtocounter{compteur}{\thestep}\line(1,-1){1}} %\put(0,30){\line(1,-1){15}}%15,15
\setcounter{height}{35}
\setcounter{compteur}{\value{base}-\value{step}*\value{height}}
\multiput(15,15)(0,1){20}{\line(0,1){0.1}}%15,35
\multiput(15,35)(1,-1){5}{\color[wave]{\thecompteur}\addtocounter{compteur}{\thestep}\line(1,-1){1}}%20,30
\setcounter{height}{30}
\setcounter{compteur}{\value{base}-\value{step}*\value{height}}
\multiput(20,30)(0,1){10}{\line(0,1){0.1}}%20,40
\setcounter{height}{40}
\setcounter{compteur}{\value{base}-\value{step}*\value{height}}
\multiput(20,40)(1,-1){5}{\color[wave]{\thecompteur}\addtocounter{compteur}{\thestep}\line(1,-1){1}} %\put(20,40){\line(1,-1){5}}%25,35

\multiput(25,35)(0,1){10}{\line(0,1){0.1}}%25,45
\setcounter{height}{45}
\setcounter{compteur}{\value{base}-\value{step}*\value{height}}
\multiput(25,45)(1,-1){10}{\color[wave]{\thecompteur}\addtocounter{compteur}{\thestep}\line(1,-1){1}} %\put(25,45){\line(1,-1){10}}%35,35
\setcounter{height}{35}
\setcounter{compteur}{\value{base}-\value{step}*\value{height}}
\multiput(35,35)(1,-1){2}{\color[wave]{\thecompteur}\addtocounter{compteur}{\thestep}\line(1,-1){1}}%\put(35,35){\line(1,-1){2}}%37,33
\setcounter{height}{33}
\setcounter{compteur}{\value{base}-\value{step}*\value{height}}
\multiput(37,33)(0,1){5}{\line(0,1){0.1}}%37,38
\setcounter{height}{38}
\setcounter{compteur}{\value{base}-\value{step}*\value{height}}
\multiput(37,38)(1,-1){4}{\color[wave]{\thecompteur}\addtocounter{compteur}{\thestep}\line(1,-1){1}}%\put(37,38){\line(1,-1){4}}%41,34
\setcounter{height}{44}
\setcounter{compteur}{\value{base}-\value{step}*\value{height}}
\multiput(41,34)(0,1){10}{\line(0,1){0.1}}%41,44
\multiput(41,44)(1,-1){10}{\color[wave]{\thecompteur}\addtocounter{compteur}{\thestep}\line(1,-1){1}}%\put(41,44){\line(1,-1){10}}%51,34
\setcounter{height}{34}
\setcounter{compteur}{\value{base}-\value{step}*\value{height}}
\multiput(51,34)(1,-1){1}{\color[wave]{\thecompteur}\addtocounter{compteur}{\thestep}\line(1,-1){1}}%\put(51,34){\line(1,-1){1}}%52,33
\setcounter{height}{33}
\setcounter{compteur}{\value{base}-\value{step}*\value{height}}
\multiput(52,33)(1,-1){3}{\color[wave]{\thecompteur}\addtocounter{compteur}{\thestep}\line(1,-1){1}}%\put(52,33){\line(1,-1){3}}%55,30
\setcounter{height}{30}
\setcounter{compteur}{\value{base}-\value{step}*\value{height}}
\multiput(55,30)(1,-1){10}{\color[wave]{\thecompteur}\addtocounter{compteur}{\thestep}\line(1,-1){1}}%\put(55,30){\line(1,-1){10}}%65,20
\setcounter{height}{30}
\setcounter{compteur}{\value{base}-\value{step}*\value{height}}
\multiput(65,20)(0,1){10}{\line(0,1){0.1}}%65,30
\multiput(65,30)(1,-1){2}{\color[wave]{\thecompteur}\addtocounter{compteur}{\thestep}\line(1,-1){1}}%\put(65,30){\line(1,-1){2}}%67,28
\setcounter{height}{43}
\setcounter{compteur}{\value{base}-\value{step}*\value{height}}
\multiput(67,28)(0,1){15}{\line(0,1){0.1}}%67,43
\multiput(67,43)(1,-1){5}{\color[wave]{\thecompteur}\addtocounter{compteur}{\thestep}\line(1,-1){1}}%\put(67,43){\line(1,-1){5}}%72,38
\setcounter{height}{44}
\setcounter{compteur}{\value{base}-\value{step}*\value{height}}
\multiput(72,38)(0,1){6}{\line(0,1){0.1}}%72,44
\multiput(72,44)(1,-1){6}{\color[wave]{\thecompteur}\addtocounter{compteur}{\thestep}\line(1,-1){1}}%\put(72,44){\line(1,-1){6}}%78,38
\setcounter{height}{38}
\setcounter{compteur}{\value{base}-\value{step}*\value{height}}
\multiput(78,38)(1,-1){10}{\color[wave]{\thecompteur}\addtocounter{compteur}{\thestep}\line(1,-1){1}}%\put(78,38){\line(1,-1){10}}%88,28
\setcounter{height}{28}
\setcounter{compteur}{\value{base}-\value{step}*\value{height}}
\multiput(88,28)(1,-1){8}{\color[wave]{\thecompteur}\addtocounter{compteur}{\thestep}\line(1,-1){1}}%\put(88,28){\line(1,-1){8}}%96,20
\setcounter{height}{20}
\setcounter{compteur}{\value{base}-\value{step}*\value{height}}
\multiput(96,20)(1,-1){5}{\color[wave]{\thecompteur}\addtocounter{compteur}{\thestep}\line(1,-1){1}}%\put(96,20){\line(1,-1){5}}%101,15
\setcounter{height}{15}
\setcounter{compteur}{\value{base}-\value{step}*\value{height}}
\multiput(101,15)(1,-1){10}{\color[wave]{\thecompteur}\addtocounter{compteur}{\thestep}\line(1,-1){1}}%\put(101,15){\line(1,-1){10}}%111,5
\setcounter{height}{20}
\setcounter{compteur}{\value{base}-\value{step}*\value{height}}
\multiput(111,5)(0,1){15}{\line(0,1){0.1}}%111,20
\multiput(111,20)(1,-1){3}{\color[wave]{\thecompteur}\addtocounter{compteur}{\thestep}\line(1,-1){1}}%\put(111,20){\line(1,-1){3}}%114,17
\setcounter{height}{37}
\setcounter{compteur}{\value{base}-\value{step}*\value{height}}
\multiput(114,17)(0,1){20}{\line(0,1){0.1}}%114,37
\multiput(114,37)(1,-1){4}{\color[wave]{\thecompteur}\addtocounter{compteur}{\thestep}\line(1,-1){1}}%\put(114,37){\line(1,-1){4}}%118,33
\setcounter{height}{43}
\setcounter{compteur}{\value{base}-\value{step}*\value{height}}
\multiput(118,33)(0,1){10}{\line(0,1){0.1}}%118,43
\multiput(118,43)(1,-1){10}{\color[wave]{\thecompteur}\addtocounter{compteur}{\thestep}\line(1,-1){1}}%\put(118,43){\line(1,-1){10}}%128,33
\setcounter{height}{33}
\setcounter{compteur}{\value{base}-\value{step}*\value{height}}
\multiput(128,33)(1,-1){16}{\color[wave]{\thecompteur}\addtocounter{compteur}{\thestep}\line(1,-1){1}}%\put(128,33){\line(1,-1){16}}%144,17
\setcounter{height}{17}
\setcounter{compteur}{\value{base}-\value{step}*\value{height}}
\multiput(144,17)(1,-1){10}{\color[wave]{\thecompteur}\addtocounter{compteur}{\thestep}\line(1,-1){1}}%\put(144,17){\line(1,-1){10}}%154,7
\setcounter{height}{17}
\setcounter{compteur}{\value{base}-\value{step}*\value{height}}
\multiput(154,7)(0,1){10}{\line(0,1){0.1}}%154,17
\multiput(154,17)(1,-1){7}{\color[wave]{\thecompteur}\addtocounter{compteur}{\thestep}\line(1,-1){1}}%\put(154,17){\line(1,-1){7}}%161,10
\setcounter{height}{14}
\setcounter{compteur}{\value{base}-\value{step}*\value{height}}
\multiput(161,10)(0,1){4}{\line(0,1){0.1}}%161,14
\multiput(161,14)(1,-1){4}{\color[wave]{\thecompteur}\addtocounter{compteur}{\thestep}\line(1,-1){1}}%\put(161,14){\line(1,-1){4}}%165,10
\setcounter{height}{10}
\setcounter{compteur}{\value{base}-\value{step}*\value{height}}
\multiput(165,10)(1,-1){3}{\color[wave]{\thecompteur}\addtocounter{compteur}{\thestep}\line(1,-1){1}}%\put(165,10){\line(1,-1){3}}%168,7
\setcounter{height}{7}
\setcounter{compteur}{\value{base}-\value{step}*\value{height}}
\multiput(168,7)(1,-1){2}{\color[wave]{\thecompteur}\addtocounter{compteur}{\thestep}\line(1,-1){1}}%\put(168,7){\line(1,-1){2}}%170,5
\setcounter{height}{5}
\setcounter{compteur}{\value{base}-\value{step}*\value{height}}
\multiput(170,5)(1,-1){5}{\color[wave]{\thecompteur}\addtocounter{compteur}{\thestep}\line(1,-1){1}}%\put(170,5){\line(1,-1){5}}%175,0

\end{picture}
\caption{IST and its contour with varying birth-rate. The color on both contour and tree denotes the variation of the birth-rate.}
\label{fig:coloredContour}
\end{figure}
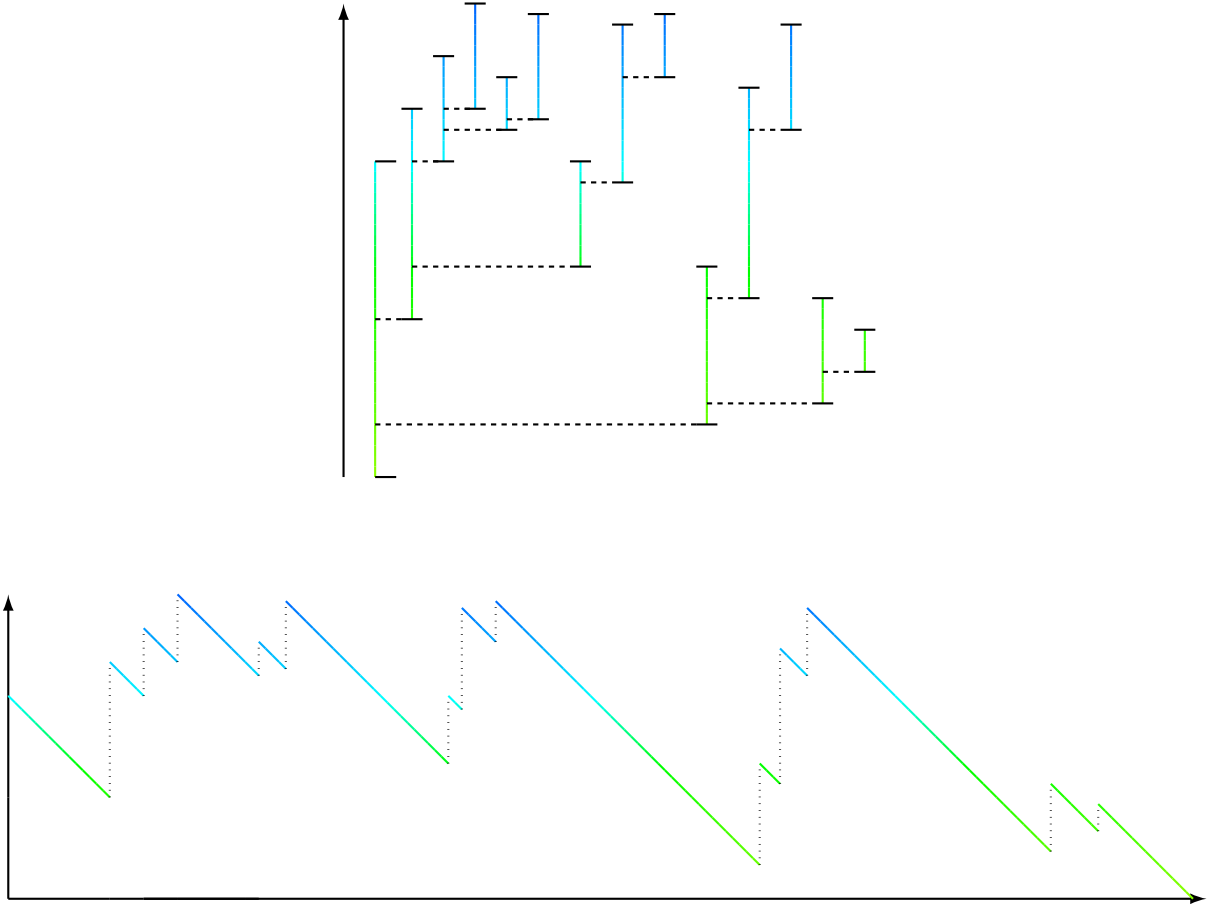

\subsection{Markov and Feller properties}
\label{ssec:markov}

The purpose of this section is to study the law of the contour process of an IST.
So let $\mathbb{T}$ be an IST with birth-rate $b$ and birth-kernel $K$. To study the contour process of $\mathbb{T}$, we need to introduce a modification of the tree $\mathbb{T}$ called contracted tree. This modification is useful in order to understand dependencies in the tree. Let $(\sigma,s)$ be an element of $\mathbb{T}$, the contracted tree $\mathbb{T}_{(\sigma,s)}$ of $\mathbb{T}$ at point $(\sigma,t)$ is defined by (see also Figure \ref{fig:contracted})
\[
\forall (\delta,t)\in \mathbb{T},\, \left\{(\delta,t)\in\mathbb{T}_{(\sigma,s)} \Leftrightarrow 
%\left\{
%\begin{array}{ll}
\exists n\in\mathbb{N}, (\sigma^{n}.\delta,t)\in\mathbb{T} \text{ and } B_{\sigma^{n}.\delta^{1}}\leq B_{\sigma^{n+1}}\wedge s \right\}. %\text{if}\, \delta\neq\emptyset,\\
% t\leq s& \text{if}\, \delta=\emptyset.\\
% \end{array}
% \right.
\]
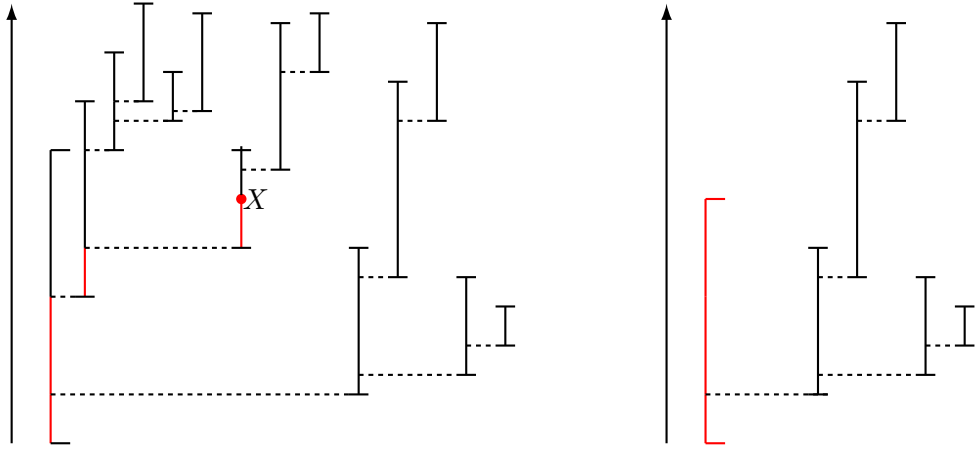
\begin{figure}[ht]
	\unitlength 1.3mm % = 2.845pt
	\linethickness{0.8pt}
	
	\begin{picture}(80,45)(-15,0)
\put(0,0){\vector(0,1){45}}
\put(3,0){
\color{black}
\color{red}
\put(0,0){\line(0,1){5}}%branche racine
\put(0,5){\line(0,1){10}}
\color{black}
\put(0,15){\line(0,1){15}}

\color{black}
\put(0,0){\line(1,0){2}}%pied racine

\put(0,30){\line(1,0){2}}%tete racine
%R2
\multiput(0,15)(1,0){3}{\line(1,0){0.5}}%r to r2  2.5  
\color{red}
\put(3.5,15){\line(0,1){5}}% r2 15>35
\color{black}
\put(3.5,20){\line(0,1){10}}

\put(3.5,30){\line(0,1){5}}
\color{black}
\put(2.5,15){\line(1,0){2}}%pied r2
\put(2.5,35){\line(1,0){2}}%tete r2
%R22
\multiput(3.5,30)(1,0){3}{\line(1,0){0.5}}%r2 to r22  l=3.5 -> 6.5

\put(6.5,30){\line(0,1){3}}% r22 30>40

\put(6.5,33){\line(0,1){2}}

\put(6.5,35){\line(0,1){5}}
\color{black}
\put(5.5,30){\line(1,0){2}}%pied r22
\put(5.5,40){\line(1,0){2}}%tete r22
%R222
\multiput(6.5,35)(1,0){3}{\line(1,0){0.5}}%r22 to r222  l=6.5 -> 8.5

\put(9.5,35){\line(0,1){10}}%r222 35>45
\color{black}
\put(8.5,35){\line(1,0){2}}%pied r22
\put(8.5,45){\line(1,0){2}}%tete r22
%R221
\multiput(6.5,33)(1,0){6}{\line(1,0){0.5}}%r22 to r221  l=6.5 -> 12.5
\put(12.5,33){\line(0,1){5}}%r221 33>38
\put(11.5,33){\line(1,0){2}}%pied r22
\put(11.5,38){\line(1,0){2}}%tete r22
%R2211
\multiput(12.5,34)(1,0){3}{\line(1,0){0.5}}%r221 to r2212  l=12.5 -> 14.5
\put(15.5,34){\line(0,1){10}}%r221 33>38
\put(14.5,34){\line(1,0){2}}%pied r22
\put(14.5,44){\line(1,0){2}}%tete r22
%R21

\multiput(3.5,20)(1,0){15}{\line(1,0){0.5}}%r2 to r21  l=3.5 -> 21.5
\put(21,25){\makebox(0,0)[cc]{$X$}}
\color{red}
\put(19.5,25){\circle*{1.061}}
\put(19.5,20){\line(0,1){4.5}}%r21 20>30
\color{black}
\put(19.5,25.4){\line(0,1){5}}
\put(18.5,20){\line(1,0){2}}%pied r21
\put(18.5,30){\line(1,0){2}}%tete r21
%R211

\multiput(19.5,28)(1,0){3}{\line(1,0){0.5}}%r21 to r211  l=3.5 -> 18.5
\put(23.5,28){\line(0,1){15}}%r211 28>43
\put(22.5,28){\line(1,0){2}}%pied r21
\put(22.5,43){\line(1,0){2}}%tete r21
%R2111
\multiput(23.5,38)(1,0){3}{\line(1,0){0.5}}%r211 to r2111  l=3.5 -> 18.5
\put(27.5,38){\line(0,1){6}}%r21 33>38
\put(26.5,38){\line(1,0){2}}%pied r21
\put(26.5,44){\line(1,0){2}}%tete r21
%R1
\multiput(0,5)(1,0){31}{\line(1,0){0.5}}%r to r1  l=3.5 -> 18.5
\put(31.5,5){\line(0,1){15}}%r1 33>38
\put(30.5,5){\line(1,0){2}}%pied r1
\put(30.5,20){\line(1,0){2}}%tete r1
%R12
\multiput(31.5,17)(1,0){3}{\line(1,0){0.5}}%r1 to r12  l=3.5 -> 18.5
\put(35.5,17){\line(0,1){20}}%r21 17>37
\put(34.5,17){\line(1,0){2}}%pied r21
\put(34.5,37){\line(1,0){2}}%tete r21
%R121
\multiput(35.5,33)(1,0){3}{\line(1,0){0.5}}%r211 to r2111  l=3.5 -> 18.5
\put(39.5,33){\line(0,1){10}}%r21 5>20
\put(38.5,33){\line(1,0){2}}%pied r21
\put(38.5,43){\line(1,0){2}}%tete r21
%R11
\multiput(31.5,7)(1,0){10}{\line(1,0){0.5}}%r1 to r11  l=3.5 -> 18.5
\put(42.5,7){\line(0,1){10}}%r21 5>20
\put(41.5,7){\line(1,0){2}}%pied r21
\put(41.5,17){\line(1,0){2}}%tete r21
%R111
\multiput(42.5,10)(1,0){3}{\line(1,0){0.5}}%r1 to r11  l=3.5 -> 18.5
\put(46.5,10){\line(0,1){4}}%r21 5>20
\put(45.5,10){\line(1,0){2}}%pied r21
\put(45.5,14){\line(1,0){2}}%tete r21
}
\end{picture}
	\begin{picture}(-40,45)(0,0)
\put(0,0){\vector(0,1){45}}
\put(3,0){
\color{black}
\color{red}
\put(0,0){\line(0,1){5}}%branche racine
\put(0,5){\line(0,1){10}}

\put(0,15){\line(0,1){10}}

\put(0,0){\line(1,0){2}}%pied racine

\put(0,25){\line(1,0){2}}%tete racine
%R2
\color{black}

%R1
\multiput(0,5)(1,0){13}{\line(1,0){0.5}}%r to r1  l=3.5 -> 18.5
\put(-20,0){
\put(31.5,5){\line(0,1){15}}%r1 33>38
\put(30.5,5){\line(1,0){2}}%pied r1
\put(30.5,20){\line(1,0){2}}%tete r1
%R12
\multiput(31.5,17)(1,0){3}{\line(1,0){0.5}}%r1 to r12  l=3.5 -> 18.5
\put(35.5,17){\line(0,1){20}}%r21 17>37
\put(34.5,17){\line(1,0){2}}%pied r21
\put(34.5,37){\line(1,0){2}}%tete r21
%R121
\multiput(35.5,33)(1,0){3}{\line(1,0){0.5}}%r211 to r2111  l=3.5 -> 18.5
\put(39.5,33){\line(0,1){10}}%r21 5>20
\put(38.5,33){\line(1,0){2}}%pied r21
\put(38.5,43){\line(1,0){2}}%tete r21
%R11
\multiput(31.5,7)(1,0){10}{\line(1,0){0.5}}%r1 to r11  l=3.5 -> 18.5
\put(42.5,7){\line(0,1){10}}%r21 5>20
\put(41.5,7){\line(1,0){2}}%pied r21
\put(41.5,17){\line(1,0){2}}%tete r21
%R111
\multiput(42.5,10)(1,0){3}{\line(1,0){0.5}}%r1 to r11  l=3.5 -> 18.5
\put(46.5,10){\line(0,1){4}}%r21 5>20
\put(45.5,10){\line(1,0){2}}%pied r21
\put(45.5,14){\line(1,0){2}}%tete r21
}}
\end{picture}
	\caption{An original tree $\mathbb{T}$ (left) and its associated contracted tree $\mathbb{T}_{X}$ (right).}
		\label{fig:contracted}
\end{figure}
Our first step is to show that the contour process is a Feller Markov process. To this end, we need the following proposition which use the contracted tree to understand the dependencies in the tree seen from the exploration process.
\begin{prop}
\label{prop:condIndep}
Let $\mathbb{T}$ be an IST with birth-kernel $K$ and birth-rate $b$ satisfying Assumption~\ref{ass:b,K}. Let $X$ be a random variable with value in\, $\mathcal{U}\times\mathbb{R}_{+}$ almost surely in $\mathbb{T}$, and $(\mathcal{G}_{s})_{ s\geq 0}$ be the natural filtration associated to $(\mathcal{C}_{s}(\mathbb{T}))_{ s\geq 0}$. Then, $\mathbb{T}_{X}$ is independent of $\mathcal{G}_{\mathcal{E}^{-1}_{X}(\mathbb{T})}$ conditionally on $P_{\mathbb{R}}(X)$.
\end{prop}
\begin{proof}
In order to lighten notations, let us denote (only in the present proof), for any $x\in\mathbb{T}$,  
\[
\xi_{x}=\xi_{P_{\mathcal{U}}(x)} \text{ and } \mathcal{N}_{x}=\mathcal{N}_{P_{\mathcal{U}}(x)},
\]
where we recall that $\xi_{\sigma}$ and $\mathcal{N}_{\sigma}$ stands respectively for the lifetime and the birth point process of individual $\sigma\in\mathbb{T}$, as defined in Section \ref{sec:model}.

Now, let $Z$ be a random variable measurable with respect to
\[
\sigma\left\{\xi_{\mathcal{E}_{s}(\mathbb{T})},\, \mathcal{N}_{\mathcal{E}_{s}(\mathbb{T})}\left(P_{\mathbb{R}_{+}}(\{x\in\mathbb{T}\mid x\leq X \})\cap \cdot \right) \mid\, s\leq \mathcal{E}^{-1}_{X}(\mathbb{T}) \right\}\supset \mathcal{G}_{\mathcal{E}^{-1}_{X}(\mathbb{T})},
\]
where the inclusion follows from \eqref{eq:inclu}.
Hence, there exists a measurable function $\psi$ such that
\[
Z=\psi\left(\left(\xi_{\mathcal{E}_{s}(\mathbb{T})}\right)_{s\leq \mathcal{E}^{-1}_{X}(\mathbb{T})},\, \left(\mathcal{N}_{\mathcal{E}_{s}(\mathbb{T})}\left(P_{\mathbb{R}_{+}}(\{x\in\mathbb{T}\mid x\leq X \})\cap \cdot \right) \right)_{s\leq\mathcal{E}^{-1}_{X}(\mathbb{T})} \right).
\]

By construction, if $y\in\mathbb{T}\backslash\{x\in\mathbb{T},\, x\leq X \}$ and $P_{\mathcal{U}}(y)\neq P_{\mathcal{U}}(X)$, then $\xi_{y}$ and $\mathcal{N}_{y}$ are independents of $Z$. Consequently, for any $y$ in $\mathbb{T}_{X}\setminus{\{\emptyset \}\times \mathbb{R}_{+}}$, $V_{y}$ and $\mathcal{N}_{y}$ are independents of $Z$. So, it only remains to understand the dependencies between $Z$ and the root of $\mathbb{T}_{X}$.
Now, denote by $\mathcal{P}$ the point process of births associated to the root of $\mathbb{T}_{X}$. By construction, for any measurable set $A$,
\[
\mathcal{P}(A)=\sum_{v\preceq P_{\mathcal{U}}(X)}\mathcal{N}_{v}\left(A\cap\{s\geq 0\mid (v,t)\preceq X \} \right),
\]
where $(v,t)\preceq X $ if and only if $v\preceq P_{\mathcal{U}}(X)$ and $X\leq (v,t)$. In particular, the set $\{x\in\mathbb{T}\mid x\preceq X \}$ corresponds to the path in the tree which link the bottom of the root to point $X$ (see Figure \ref{fig:contracted}).
Now, since $\{x\in\mathbb{T}\mid x\preceq X \}$ and $\{x\in\mathbb{T}\mid x\leq X \}$ are disjoint sets, wet get,
using the independence properties of Poisson random measure, that $\mathcal{P}$ is independent of $Z$. Hence, the dependence between $Z$ and $\mathbb{T}_{X}$ only relies on the lifetime of the root. Finally, since the lifetime of the root of $\mathbb{T}_{X}$ is given by $P_{\mathbb{R}}(X)$, we can conclude that $\mathbb{T}_{X}$ is independent of $Z$ conditionally on $P_{\mathbb{R}}(X)$. This ends the proof.
\end{proof}
We can now show that the contour process of a truncated IST is a Feller process.

\begin{prop}
		Let $\mathbb{T}$ be an IST  with birth-kernel $K$ and birth-rate $b$ satisfying Assumptions \ref{ass:b,K}.  Let $(\mathbb{P}_{x})_{x\in\mathbb{R}_{+}}$ a family of probability measure such that $\mathbb{P}_{x}(\xi_{\emptyset}=x)=1$.  Denote, for any positive real number $T$, $\mathbb{T}^{T}$ the truncated tree above $T$ as defined in Remark \ref{rem:trunc}. Then,  $\left((\mathcal{C}_{s}(\mathbb{T}^{T})_{s\geq 0},\mathbb{P}_{x}\right)$ is a Feller process.
\end{prop}
\begin{proof}

	We begin the proof by showing that $\mathcal{C}_{s}(\mathbb{T}^{T})_{s\geq 0}$  is a Markov process. To this end let us denote by $\mathcal{F}=(\mathcal{F}_{s})_{s\geq 0}$ the natural filtration associated with $\mathcal{C}_{s}(\mathbb{T}^{T})_{s\geq 0}$ and, as above, $\mathcal{G}^{T}=(\mathcal{G}^{T}_{s})_{s\geq 0}$ the natural filtration of $(\mathcal{E}_{s}(\mathbb{T}^{T}))_{s\geq 0}$. Let $f$ be some bounded measurable function. We have, for any positive real numbers $s$ and $h$,
	\[
	\mathbb{E}\left[f\left(\mathcal{C}_{s+h}\left(\mathbb{T}^{T}\right)\right)\mid \mathcal{G}^{T}_{s} \right]=\mathbb{E}\left[f\left(\mathcal{C}_{h}\left(\mathbb{T}^{T}_{\mathcal{E}_{s}(\mathbb{T}^{T})}\right)\right)\mid \mathcal{G}^{T}_{s} \right].
	\]
	
	However, according to Proposition \ref{prop:condIndep}, $\mathbb{T}^{T}_{\mathcal{E}_{s}(\mathbb{T}^{T})}$ is independent of $\mathcal{G}^{T}_{s}$ conditionally on $P_{\mathbb{R}}(\mathcal{E}_{s}(\mathbb{T}^{T}))=\mathcal{C}_{s}(\mathbb{T}^{T})$. Consequently,
	\[
	\mathbb{E}\left[f\left(\mathcal{C}_{h}\left(\mathbb{T}^{T}_{\mathcal{E}_{s}(\mathbb{T}^{T})}\right)\right)\mid \mathcal{G}^{T}_{s} \right]=\mathbb{E}\left[f\left(\mathcal{C}_{h}\left(\mathbb{T}^{T}_{\mathcal{E}_{s}(\mathbb{T}^{T})}\right)\right)\mid \mathcal{C}_{s}(\mathbb{T}^{T}) \right],
	\]
	which gives the Markov property.
	As a consequence of the Markov property, 
	\[
	f\mapsto \mathbb{E}_{x}\left[f\left(\mathcal{C}_{s}\left(\mathbb{T}^{T}\right)\right)\right]
	\]
	defines a semigroup $(P_{s})_{s\geq 0}$ on the space $\mathcal{B}_{b}([0,T])$ of bounded measurable function on $[0,T]$. To prove the Feller property, it remains to show that the subspace $C([0,T])$ of continuous function on $[0,T]$ is invariant under the action of $P_{s}$, for any $s\geq 0$.
	
	 Let $f\in C([0,T])$ and $y<x\leq t$. Using Markov property and denoting by $J$ the time of first jump of $(\mathcal{C}_{s}(\mathbb{T}^{T}))_{s\geq 0}$, we have, with $s>(x-y)$, that
	\begin{align*}
	\mathbb{E}_{x}\left[f\left(\mathcal{C}_{s}\left(\mathbb{T}^{T}\right)\right) \right]&=\mathbb{E}_{x}\left[f\left(\mathcal{C}_{s}\left(\mathbb{T}^{T}\right)\right)\mathds{1}_{J>x-y} \right]+\mathbb{E}_{x}\left[f\left(\mathcal{C}_{s}\left(\mathbb{T}^{T}\right)\right)\mathds{1}_{J\leq x-y} \right]\\
	&=\mathbb{E}_{y}\left[f\left(\mathcal{C}_{s-(x-y)}\left(\mathbb{T}^{T}\right)\right) \right]\mathbb{P}_{x}\left(J>x-y \right)+\mathbb{E}_{x}\left[f\left(\mathcal{C}_{s}\left(\mathbb{T}^{T}\right)\right)\mathds{1}_{J\leq x-y} \right].\\
	\end{align*}
	Now, by construction
	\[
	\mathbb{P}_{x}(J>x-y)=\mathbb{P}\left(\mathcal{N}_{\emptyset}([x,y])=0\right)=\exp\left(-\int_{y}^{x}b(u)\ du\right).
	\]
	Consequently,
	\begin{align*}
	\left|	\mathbb{E}_{x}\left[f\left(\mathcal{C}_{s}\left(\mathbb{T}^{T}\right)\right) \right]-	\mathbb{E}_{y}\left[f\left(\mathcal{C}_{s}\left(\mathbb{T}^{T}\right)\right) \right]\right|\leq& \left|P_{s-(x-y)}f(y)\exp\left(-\int_{y}^{x}b(u)\ du\right)-P_{s}f(y)\right|\\&+\|f\|_{\infty}\left(1-\exp\left(-\int_{y}^{x}b(u)\ du\right) \right)\\
	\leq&\left|P_{s-(x-y)}f(y)-P_{s}f(y)\right|+2\left(1-\exp\left(-\int_{y}^{x}b(u)\ du\right) \right)\|f\|_{\infty}.
	\end{align*}
	Now, using once again the Markov property, we have that
	\begin{multline*}
	\mathbb{E}_{y}\left[f(\mathcal{C}_{s-(x-y)}(\mathbb{T}^{T})) \right]=\mathbb{E}_{y}\left[f\left(\mathcal{C}_{s}(\mathbb{T}^{T})+(x-y)\right) \exp\left(-\int_{\mathcal{C}_{s}(\mathbb{T}^{T})}^{\mathcal{C}_{s}(\mathbb{T}^{T})+(x-y)}b(u)\ du\right)  \right]\\
	+\mathbb{E}_{y}\left[f\left(\mathcal{C}_{s-(x-y)}(\mathbb{T}^{T})\right) \left(1-\exp\left(-\int_{\mathcal{C}_{s}(\mathbb{T}^{T})}^{\mathcal{C}_{s}(\mathbb{T}^{T})+(x-y)}b(u)\ du\right)\right)  \right],
	\end{multline*}
	leading to
	\begin{multline*}
	\left|P_{s-(x-y)}f(y)-P_{s}f(y)\right|\\\leq \left|\mathbb{E}_{y}\left[f\left(\mathcal{C}_{s}(\mathbb{T}^{T})\right)-f\left(\mathcal{C}_{s}(\mathbb{T}^{T})+(x-y)\right) \right]\right|+2\mathbb{E}\left[\left(1-\exp\left(-\int_{\mathcal{C}_{s}(\mathbb{T}^{T})}^{\mathcal{C}_{s}(\mathbb{T}^{T})+(x-y)}b(u)\ du\right) \right)\right]\|f\|_{\infty}.
	\end{multline*}
	Finally, we obtain, for any $x>y$ and $s>(x-y)$,
	\begin{multline}
	\label{eq:FelEstim}
\left|	P_{s}f(x)-P_{s}f(y)\right|\\\leq \left|\mathbb{E}_{y}\left[f\left(\mathcal{C}_{s}(\mathbb{T}^{T})\right)-f\left(\mathcal{C}_{s}(\mathbb{T}^{T})+(x-y)\right) \right]\right|+4\left(1-\exp\left(-(x-y)\|b\|_{\infty}\right) \right)\|f\|_{\infty}.
	\end{multline}
 Now, we have thanks to Equation \eqref{eq:FelEstim} and Lebesgue's dominated convergence theorem that
	\[
	\left|	P_{s}f(x)-P_{s}f(y)\right|\xrightarrow[x\to y]{}0,
	\]
	for any $f$ in $C([0,T])$ and any $s\geq 0$. Hence, $P_{s}f$ is continuous on $(0,T]$.
	 So it remains to prove that $P_{s}f$ is continuous at $0$. To this end, let $x$ be a positive real number, we have
	\begin{align*}
	\left|\mathbb{E}_{x}\left[f\left(\mathcal{C}_{s}\left(\mathbb{T}^{T}\right)\right) \right]-f(0)\right|&\leq\left|f(0)\mathbb{P}_{x}(T>x) -f(0)\right|+\|f\|_{\infty}(1-\mathbb{P}_{x}(T>x)).
		\end{align*}
So according to the above computation, this gives the continuity of $P_{s}f$ at point $0$. Hence, the space $C([0,T])$ is invariant under the action of $P_{s}$ for any $s\geq 0$, which gives the Feller property.
\end{proof}
Since the contour process is a Feller process, it induces a strongly continuous semigroup on the space $C([0,T])$ of continuous function on $[0,T]$. Our next goal is to obtain the generator of the process as well as its domain (namely the strong generator on $C([0,T])$, see the appendix for the different notions of generators). Looking at the behavior of the tree and, thus, of the contour, a natural candidate is given by 
\[
L^{(T)}f(x)=-f'(x)+b(x)\int_{\mathbb{R}_{+}}\, (f((x+y)\wedge T)-f(x))\, K(x,dy)
\]
with domain
\[
D(L^{(T)})=\left\{f\in C^{1}([0,T]),\ L^{(T)}f(0)=0 \right\}.
\]
Our strategy relies on two steps: the first step is to show that
\[
\lim\limits_{s\to0}\frac{1}{s}\left(\mathbb{E}_{x}\left[f\left(\mathcal{C}_{s}(\mathbb{T}^{T}) \right) \right]-f(x) \right)= L^{(T)}f,
\]
in $C([0,T])$. But this is not enough to ensure that $(L^{(T)},D(L^{(T)}))$ is the generator of the contour process. Hence, our second step is to show that $(L^{(T)},D(L^{(T)}))$ generates a strongly continuous semigroup on  $C([0,T])$. As consequence, using the maximality property of such operator (see \cite[Lemma 17.12]{Kal}), we obtain that $(L^{(T)},D(L^{(T)}))$ is indeed the generator of the contour. We begin by showing that $(L^{(T)},D(L^{(T)}))$ generate a strongly continuous semigroup.
\begin{prop}
Let $L^{(T)}$ be the operator on $C([0,T])$ defined by
\[
\forall x\geq 0, \quad L^{(T)}f(x)=-f'(x)+b(x)\int_{\mathbb{R}_{+}}\, (f((x+y)\wedge T)-f(x))\, K(x,dy),
\]
for all function $f$ in the domain
\[
D(L^{(T)})=\left\{f\in C^{1}([0,T]),\ L^{(T)}f(0)=0 \right\}.
\]
Then, under Assumption~\ref{ass:b,K}, $(L^{(T)},D(L^{(T)}))$ generate a strongly continuous semigroup on $C([0,T])$.
\end{prop}
\begin{proof}
In the following, the  supremum norm $\|\cdot\|_{\infty}$ refers to the one of $C([0,T])$. Let us first consider the operator
\[
\mathcal{A}f=-f'-2\|b\|_{\infty}f,
\]
defined for function $f$ in the domain 
\[
D(\mathcal{A})=\left\{f\in C^{1}([0,T]),\ L^{(T)}f(0)=0 \right\}.
\]
It is easily seen that $L^{(T)}$ is a perturbation of $\mathcal{A}$ by a bounded operator. Hence, deducing that $L^{(T)}$ generate a strongly continuous semigroup is easy if $\mathcal{A}$ does. Hence, we begin to show, using Lumer-Philips Theorem (see \cite[Chapter 1, Theorem 4.3]{pazy}), that $\mathcal{A}$ generates a strongly continuous semigroup.

Let $f$ in $D(\mathcal{A})$ and $x^{\ast}$ be a global maximum of $|f|$.

It is hence easily seen that the measure $f(x^{\ast}) \delta_{x^{\ast}}$  belongs to the duality set of $f$ \cite[Chapter 1, Section 1.4]{pazy}. In addition, we have
\[
\langle -f'-2\|b\|_{\infty}f,\ f(x^{\ast}) \delta_{x^{\ast}}\rangle = \left(-f'(x^{\ast})-2\|b\|_{\infty}f(x^{\ast}) \right)f(x^{\ast})=-\frac{1}{2}(f^2)'
(x^{\ast})-2\|b\|_{\infty}f(x^{\ast})^{2}. \]
Since $x^\ast$ is an extremum of $f$, \begin{equation}\label{eq:disBound}\langle -f'-\|b\|_{\infty}f,\ f(x^{\ast}) \delta_{x^{\ast}}\rangle\leq -\|b\|_{\infty}f(x^{\ast})^{2},\end{equation} as soon as $x^{\ast}\in (0,T)$. On the other hand, $x^{\ast}$ is a maximum for $f^{2}$, so \eqref{eq:disBound} also holds if $x^{\ast}=T$. Finally, if $x^{\ast}=0$, we have 
\begin{multline*}\left(-f'(x^{\ast})-2\|b\|_{\infty}f(x^{\ast}) \right)f(x^{\ast})=-\|f\|_{\infty} b(0)\int_{\mathbb{R}_{+}}f(y)\ K(0,dy)-(2\|b\|_{\infty}-b(0))\|f\|_{\infty}^{2}\\\leq (b(0)-\|b\|_{\infty})\|f\|_{\infty}^{2}.\end{multline*}
Finally, in any case, $\langle -f'-\|b\|_{\infty}f,\ f(x^{\ast}) \delta_{x^{\ast}}\rangle\leq 0$ which implies that $\mathcal{A}$ is a dissipative operator; see \cite[Chapter 1, Section 1.4]{pazy}. In addition, operator $L^{(T)}$ is unbounded, its kernel is dense in $C([0,T])$ (see Prop. 1.7.16 in \cite{Banach}). To apply Lumer-Philips theorem, it remains to show that the range of $\lambda-\mathcal{A}$  is dense in $C([0,T])$. However, taking $g\in C([0,T])$, solving the problem
\[
\left\{ 
\begin{array}{l}
\lambda f+f'=g\\
L^{(T)}f(0)=0,
\end{array}
\right.
\]
is straightforward using elementary ODE solving techniques. As a consequence, $\mathcal{A}$ generate a strongly continuous semigroup over $C([0,T])$.

Finally, since 
\[
f\mapsto b(x)\int_{\mathbb{R}_{+}}\, (f((x+y)\wedge T)-f(x))\, K(x,dy)
\]
is a bounded operator over $C([0,T])$,  $L^{(T)}$ is obtained from $\mathcal{A}$ through the perturbation by a bounded operator. As a consequence, $L^{(T)}$ is a closed operator with the same domain as $\mathcal{A}$ and generating a strongly continuous semigroup (see \cite[Chapter 3, Theorem 1.1]{pazy}).
	\end{proof}
We can finally show that $(L^{(T)},D(L^{(T)}))$ is the generator of the truncated contour process.
\begin{thm}
\label{thm:genContour}
Let $\mathbb{T}$ with birth-kernel $K$ and birth-rate $b$ satisfying Assumption~\ref{ass:b,K}. Let $(\mathbb{P}_{x})_{x\in\mathbb{R}_{+}}$ a family of probability measure such that $\mathbb{P}_{x}(\xi_{\emptyset}=x)=1$.  Denote, for any positive real number $T$, $\mathbb{T}^{T}$ the truncated tree above $T$. Then,  $((\mathcal{C}_{s}(\mathbb{T}^{T}))_{s\geq 0},\mathbb{P}_{x})$ is a Feller process, with generator 
\[
L^{(T)}f(x)=-f'(x)+b(x)\int_{\mathbb{R}_{+}}\, (f((x+y)\wedge T)-f(x))\, K(x,dy),
\]
on the domain
\[
D(L^{(T)})=\left\{f\in C^{1}([0,T]),\ L^{(T)}f(0)=0 \right\}.
\]

%In addition, $\left(\mathcal{C}_{s}(\mathbb{T}^{\mathbf{t}}),\, s\geq 0\right)$ converges almost surely to a Markov process with generator given by 
%\[
%Lf(x)=-f'(x)+b(x)\int_{\mathbb{R}_{+}}\, f((x+y))-f(x)\, K(x,dy),
%\]
%on the domain
%\[
%D(L)=\left\{f\in C^{1}_{0}(\mathbb{R}_{+}),\ Lf(0)=0 \right\}.
%\]
%If $\mathbb{T}$ is almost surely finite, then this Markov process is $\left(\mathcal{C}_{s}(\mathbb{T}),\, s\geq 0\right)$.
\end{thm}

\begin{proof}
	 We can now derive the generator of this process.  Assume now that $f$ belongs to the set $D(L^{(T)})$ defined above.  In order to lighten notation, $\mathcal{E}_{s}(\mathbb{T}^{T})$ is now simply denoted $\mathcal{E}_{s}$.

\[
\mathbb{E}_{x}\left[f(\mathcal{C}_{h}) \right]=\mathbb{E}_{x}\left[f(\mathcal{C}_{h})\mathds{1}_{\Delta \mathcal{C}([0,h])=0 } \right]+\mathbb{E}_{x}\left[f(\mathcal{C}_{h})\mathds{1}_{\Delta \mathcal{C}([0,h])=1 } \right]+\mathbb{E}_{x}\left[f(\mathcal{C}_{h})\mathds{1}_{\Delta \mathcal{C}([0,h])>1 } \right],
\]
where $\Delta C([0,h])$ denotes the number of jump of $\mathcal{C}$ over the time interval $[0,h]$.
By construction of the exploration process $\Delta C([0,h])=1$ if and only if $\mathcal{N}_{\emptyset}([x-h,x])>0$ and $\mathcal{N}_{N_{\emptyset}}([x+\xi_{N_{\emptyset}}-h,x+\xi_{N_{\emptyset}}-J])=0$, where $N_{\emptyset}$ stands for $\mathcal{N}_{\emptyset}([0,\xi_{\emptyset}])$ and $J=\inf\{s>0\mid \mathcal{N}_{\emptyset}([\xi_{\emptyset}-s,\xi_{\emptyset}])>0 \}$. In particular, let us highlight that $N_{\emptyset}$ is the number of children of individual $\emptyset$ and that individual $N_{\emptyset}$ is the first one to be explored by the contour process after exploring individual $\emptyset$. These lead to
\[
\mathbb{E}_{x}\left[f(\mathcal{C}_{h})\mathds{1}_{\Delta \mathcal{C}([0,h])=1 } \right]=\mathbb{E}_{x}\left[f(x+V_{N_{\emptyset}})\mathds{1}_{\mathcal{N}_{\emptyset}([x-h,x])>0 }\mathds{1}_{\mathcal{N}_{N_{\emptyset}}([x+\xi_{N_{\emptyset}}-h,x+\xi_{N_{\emptyset}}-J])=0 } \right].
\]
By construction, $\xi_{N_{\emptyset}}$, $\mathcal{N}_{N_{\emptyset}}$ and $\mathcal{N}_{\emptyset}$ are independent. It follows that
\begin{align*}
\mathbb{E}_{x}\left[f(\mathcal{C}_{h})\mathds{1}_{\Delta \mathcal{C}([0,h])=1 } \right]&=\mathbb{E}_{x}\left[f(x+\xi_{N_{\emptyset}})\mathds{1}_{\mathcal{N}_{\emptyset}([x-h,x])>0 }\mathds{1}_{\mathcal{N}_{N_{\emptyset}}([x+\xi_{N_{\emptyset}}-h,x+\xi_{N_{\emptyset}}-J])=0 } \right]\\
&=\mathbb{E}_{x}\left[\int_{\mathbb{R}_{+}}f(x+y)\mathds{1}_{J\leq h }\mathds{1}_{\mathcal{N}_{N_{\emptyset}}([x+y-h,x+y-J])=0 }\, K(x-J,dy) \right]\\
&=\int_{[0,h]}\int_{\mathbb{R}_{+}}f(x+y)e^{-\int_{x+y-h}^{x+y-z} b(u)\ du }\, K(x-z,dy) b(x-z)e^{-\int_{x-z}^{x}b(u)\ du}\ dz,
\end{align*}
where the last equality is obtained using that $\mathbb{P}_{x}(J>s)=\exp(\int_{x-s}^{x}b(u)\ du)$.
Now, set 
\[
I_{h}=\int_{[0,h]}b(x-z)e^{-\int_{x-z}^{x}b(u)\ du}\ dz\int_{\mathbb{R}_{+}}f(x+y)\, K(x,dy).
\]
Now, it is easily seen that
\begin{align*}
&\Bigg|\int_{[0,h]}\int_{\mathbb{R}_{+}}f(x+y)e^{-\int_{x+y-h}^{x+y-z} b(u)\ du }\, K(x-z,dy) b(x-z)e^{-\int_{x-z}^{x}b(u)\ du}\ dz-I\Bigg|\\\leq&\ 3 \Bigg|\int_{[0,h]}\int_{\mathbb{R}_{+}}f(x+y) \left(K(x-z,dy)-K(x,dy)\right) b(x-z)e^{-\int_{x-z}^{x}b(u)\ du}\ dz\Bigg|\\\leq&\ 3h\|b\|_{L^{\infty}([0,T])}\sup_{z\in [0,h]}\left|(Kf(x)-Kf(x-z)\right|.
\end{align*}
But since $x\mapsto K(x,dy)$ is weakly continuous, from Heine-Cantor Theorem we get that 
\[
h\|b\|_{L^{\infty}([0,T])}\sup_{z\in [0,h]}\left| Kf(x)-Kf(x-z)\right|=o(h).
\]
Moreover,
\[
\lim\limits_{h\to 0}\frac{1}{h} I_{h}=b(x)\int_{\mathbb{R}_{+}}f(x+y)\ K(x,dy),
\]
we obtain the same limit for $h^{-1}\mathbb{E}_{x}\left[f(\mathcal{C}_{h})\mathds{1}_{\Delta \mathcal{C}([0,h])=1 } \right]$.
Using similar argument, one can also get that
\[
\mathbb{E}_{x}\left[f(\mathcal{C}_{h})\mathds{1}_{\Delta \mathcal{C}([0,h])=0 } \right]=f(x)e^{\int_{x-h}^{x}b(u)\ du },
\]
and
\[
\mathbb{E}_{x}\left[f(\mathcal{C}_{h})\mathds{1}_{\Delta \mathcal{C}([0,h])>1 } \right]\leq \|f\|_{\infty}\mathbb{P}\left( \mathcal{P}([0,h])>1\right),
\]
where $\mathcal{P}$ is a Poisson random measure on $[0,T]$ with constant intensity given by $\|b\|_{L^{\infty}([0,T])}$, which implies that
\[
\mathbb{E}_{x}\left[f(\mathcal{C}_{h})\mathds{1}_{\Delta \mathcal{C}([0,h])>1 } \right]=o(h),
\]
uniformly in $x$.
Hence, we get that
\begin{multline*}
\left|\frac{1}{h}\left(\mathbb{E}_{x}\left[f(\mathcal{C}_{h}) \right]-f(x)\right)-L^{(T)}f(x)\right|\leq \left|\frac{1}{h}\left(f(x-h) -f(x)\right)e^{-\int_{x-h}^{x}b(u)\ du}+f'(x)\right|\\+\left|b(x)-\frac{1}{h}\left(1-e^{-\int_{x-h}^{x}b(u)\ du} \right)\right|\|f\|_{\infty}+o(1).
\end{multline*}
Using the uniform continuity of $f'$ and $\int b$ on $[0,T]$ finally gives the desired result.
\end{proof}
\subsection{A jump process point of view for the contour}
\label{ssec:jumpProc}

Let us fix $T\in \mathbb{R}_+ \cup \{+\infty\}$. In this section, we define a Markov process  $(X^{(T)}_t)_{t\geq 0}$, which is distributed as $(\mathcal{C}_s(\mathbb{T}^{T}))_{s\geq 0}$ but whose construction is separated from the tree. The case $T=+\infty$ corresponds to the non-truncated tree $\mathbb{T}$ and we omit to write the symbol $(T)$, i.e.\ $(X^{(\infty)}_{t})_{t\geq 0}=(X_t)_{t\geq 0}$ . In addition, for being relevant with the tree notation, we also denote by $\mathbb{P}_{x}$ the probability measure on $\mathbb{D}(\mathbb{R}_{+})$, the Skorokhod space (see \cite{EK86}),  satisfying $\mathbb{P}_{x}(X_{0}^{T}=x)=1$, for $0<T\leq\infty$.

This process has $[0,T]$ (or $\mathbb{R}_+$ when $T=+\infty$) as state space and evolves deterministically and linearly (with slope $-1$) between some positive jumps. These jumps arise at the non-homogeneous rate $b$ and the sizes are given by the kernel $K_{T}$ satisfying for any function $f$ and $x \in [0, T]$,
$$
\int_{\mathbb{R}_{+}} f(y)\ K_{T}(x,dy) = \int_{\mathbb{R}_{+}} f(y\wedge T)\ K(x,dy).
$$
The process is also supposed to be absorbed at $0$. Let us give further explanations on its dynamics. Starting from some $x>0$, $(X^{(T)})_{t\geq 0}$ experiences its first jump at some random time $\tau \in \mathbb{R}_+ \cup \{+\infty \}$ such that
\begin{equation*}
\label{eq:loitau}
\forall t\geq 0, \quad \mathbb{P}(\tau>t)= \exp\left(-\int_0^{t\wedge x} b(x-s) ds \right),
\end{equation*}
and $\mathbb{P}(\tau= \infty)= \exp\left(-\int_0^{x} b(u) du\right)$, corresponding to the absorption of the process at $0$. For $x\in [0,\tau)$, $X^{T}_s=(x-s)\mathbf{1}_{s\leq x}$, and conditionally on the event $\{\tau < \infty\}=\{\tau < x\}$, the jump size $(X^{(T)}_{\tau}-X^{(T)}_{\tau-})$ is distributed according to $K_{T}(X^{(T)}_{\tau-}, \cdot)=K_{T}(x-\tau, \cdot)$.
After this jump time, the process $(X^{(T)}_t)_{t\geq 0}$ recursively follows these preceding steps.

If $T<+\infty$ and $b$ is locally bounded (as in Assumptions \ref{ass:b,K}) and jumps are bounded (as in Assumptions \ref{ass:b,K}), it is easy to see that the process is well-defined. Namely, there is no explosion, that is the sequence of jump times does not converge to some finite value. More precisely, if $(T_k)_{k\geq 0}$ is the sequence of jump times of $(X_t)_{t\geq 0}$, we have that $\lim_{k\to \infty} T_k=+\infty$ (it can also be a finite sequence whose last term is $+\infty$).  This may not be the case when $T=+\infty$. Note that this definition of explosion is consistent with the definition of \cite{MT93III}.

Due to the exponential-like distribution of $\tau$ and its iterated construction, $(X^{(T)}_t)_{t\geq 0}$ is a strong Markov process.  
 It is neither a diffusion process (because it is not continuous) nor a L\'{e}vy process (because it is space-inhomogeneous). This type of process belongs to the class of Piecewise Deterministic Markov
Process (PDMP). Introduced in \cite{D84},  this class of processes has recently motivated  a considerable amount of research in various context (see for instance \cite{ABGMKZ,CDGMMY,M15,SDZ} for surveys). Among many results, simulation algorithms can be found in \cite{D93,SDZ,LTT}.

The present section is devoted to the study of this particular process. In particular, we can derive its generator as above and, even, characterize its extended generator.

\begin{thm}[Strong generator]
Under Assumptions \ref{ass:b,K}, if $T<+\infty$, then $(X_t^{(T)})_{t\geq0}$ is a Feller process with full generator
\begin{align*}
L^{(T)} f(x)&=-f'(x)+b(x)\int_{\mathbb{R}_{+}}\, (f((x+y)\wedge T)-f(x))\, K(x,dy)\\
&=-f'(x)+b(x)\int_{\mathbb{R}_{+}}\, (f(x+y)-f(x))\, K_{T}(x,dy),
\end{align*}

 on $C([0,T])$ defined
on the domain
\[
D(L^{(T)})=\left\{f\in C^{1}([0,T]),\ L^{(T)} f(0)=0 \right\}.
\]
\end{thm}
\begin{proof}
The Feller property comes from \cite[Theorem (27.6) p 27]{D93}. The remaining of the proof is analogous to the one of Theorem \ref{thm:genContour}.
\end{proof}

An easy consequence is the following (see for instance \cite[Proposition 1.2.9]{EK86}).

\begin{cor}[The PDMP and the contour process]
\label{cor:mmloi}
Under Assumptions \ref{ass:b,K}, processes $(\mathcal{C}_{t}(\mathbb{T}^{(T)}))_{t\geq 0}$ and $(X^{(T)}_t)_{t\geq 0}$ have the same law.
\end{cor}

In addition, we can also obtain the extended generator of the process (see the appendix for the definition).

\begin{thm}[Extended generator: case $T<+\infty$]
\label{th:genext}
Under Assumptions \ref{ass:b,K}, if $T<+\infty$, then the domain $\widehat{D}(L^{(T)})$ of the extended generator of $(X_t^{(T)})_{t\geq0}$ on $C([0, T])$ consists of the absolutely continuous functions $f$. Moreover, for such function $f$, we have
\begin{align*}
L^{(T)} f(x)=-f'(x)+b(x)\int_{\mathbb{R}_{+}}\, (f(x+y)-f(x))\, K_{T}(x,dy),
\end{align*}
where $f'$ has to be understood in the sense of Radon-Nikodym.
\end{thm}	
\begin{proof}
This is a direct consequence of \cite[Theorem (26.14) p.69]{D93}.
\end{proof}

The extended generator of the non-truncated process $(X_t)_{t\geq 0}$ is also characterized by  \cite[Theorem (26.14) p.69]{D93}.

In what follow, $(X_t)$ will be refereed as the contour process due to Corollary~\ref{cor:mmloi}

\iffalse

\begin{thm}[Extended generator: case $T=+\infty$]
Under Assumptions \ref{ass:b,K}, if $T=+\infty$. Moreover, if we assume that the process is non-explosive. Then the domain $\widehat{D}(L)$ of the extended generator of $(X_t)_{t\geq0}$ on $C([0, +\infty))$ consists of the absolutely continuous functions $f$ such that
$$
\forall n\geq 1, \quad \mathbb{E}\left[ \left| \sum_{k\geq 1, T_k \leq \tau_n}  f(X_{T_k}) - f(X_{T_k-})\right|\right]<\infty,
$$
for at least one sequence of stopping time $\tau_n$ converging to infinity.
Moreover, for such function, we have
\begin{align*}
Lf(x)=-f'(x)+b(x)\int_{\mathbb{R}_{+}}\, (f(x+y)-f(x))\, K(x,dy).
\end{align*}

\end{thm}	
\begin{proof}
This is a direct consequence of \cite[Theorem (26.14) p.69]{D93}.
\end{proof}

\fi

\section{Scale-type functions and hitting times}
\label{sec:scale}

In this section, we define our scale function for the contour process and study some applications. It is then structured through three subsection devoted to its definition, its applications for the tree and finally some particular cases with explicit formulas.

\subsection{Definition of the scale function}
\label{ssec:scale-def}
Here, we will define and study a \textit{scale} type function for the contour process. Before introducing it, let us recall briefly the definition of the scale function in the context of diffusion processes and L\'{e}vy processes. The interested reader may also take a look at \cite{RY} for the diffusion case and \cite{scaleFunctionKyp} for the L\'evy case. For a diffusion processes $(Z_t)_{t\geq 0}$ on $\mathbb{R}$, the (or actually a) scale function $s:\mathbb{R} \to \mathbb{R}$ verifies, for every $a< b$,
\begin{equation}
\label{eq:scale-eds}
 \forall x\in [a,b], \quad  \mathbb{P}_x(\tau_a > \tau_b)  = \mathbb{P}_x(\tau_{(-\infty,a]} > \tau_{[b,+\infty)}) = \frac{s(x) - s(a)}{s(b)-s(a)},
\end{equation}
where $\tau_A$ is the hitting time of the set or the point $A$ for the process $(Z_t)_{t\geq 0}$ (see for instance \cite[Chapter VII, Definition (3.3)]{RY}). This function is known to be continuous, strictly increasing, unique up to an affine function (see \cite[Chapter VII, Proposition (3.2)]{RY}). Moreover, from \cite[Equation (5.42) 339]{KS}, we have
$$
s:x\mapsto \int_c^x \exp\left(-2\int_c^y \frac{b(z)}{\sigma^2(z)} dz \right) dy,
$$
for any constant $c\in \mathbb{R}$, where $b,\sigma$ are respectively the drift and the diffusion coefficients of the diffusion process. Moreover, in general, $(s(Z_{t\wedge T_{[a,b]^c}}))_{t\geq 0}$ is a martingale and $\mathcal{L} s = 0$, where $\mathcal L$ is the generator of $(Z_t)_{t\geq 0}$ (see \cite[Equation (5.43)]{KS}). For a L\'{e}vy process $(L_t)_{t\geq 0}$, the definition of the scale function, generally denoted by $W$, is closely related to the previous one but slightly differ. Indeed, from \cite[Theorem 8 p.194]{B98}, we have, for every $a\leq b$,
$$
 \forall x\in [a,b], \quad   \mathbb{P}_x(\tau_{(-\infty,a]} > \tau_{[b,+\infty)}) = \frac{W(x-a)}{W(b-a)},
$$
where $\tau_A$ is now the hitting time associated to the L\'{e}vy process $(L_t)_{t\geq 0}$. In particular, in the case of L\'evy process with no positive jumps, the map $x\mapsto W(x)$ is increasing and verifies 
$$
\int_0^\infty e^{-\lambda x} W(x) dx = \frac{1}{\psi(\lambda)},
$$
for large $\lambda$ and $\psi$ is the Laplace exponent of $(L_t)$; see \cite[p.188]{B98}. Again, $(W(L_t))$ is a martingale and $\mathcal{L} W=0$, where $\mathcal{L}$ is the generator of $(L_t)_{t\geq 0}$ (this is a rewriting of the last expression, up to a Laplace transform).  
%This is for this reason that Bertoin, in \cite{B92}, gives this name for $W$\cmb{a reformuler}. 
The map $W$ is fundamental in the study of homogeneous splitting tree and related models \cite{CH16, CLR,H15,  Lambert2010,R11,Ricthese}. %In our case, it appears for instance in the parameter of the geometric distribution associated to the number of individuals (see \cite{Lambert2010}).

Both these functions are particular cases of harmonic functions of Markov processes (see \cite[Chapter 6]{K11}) that are solutions to the Dirichlet problem. More precisely, for an open set $U$ and a function $h\in C(\partial U)$, the associated Dirichlet problem for a Markov process $(X_t)_{t\geq 0}$, with generator $L$ corresponds to the existence of a function $u$ of the following equation:
$$\left\{
\begin{array}{ll}
  Lu(x)=0,\quad &\forall x \in U \\
u(x)=h(x),\quad &\forall x\in \partial U.
\end{array}
\right.$$

From Feynman-Kac type formula, under irreducibility assumption $x\mapsto u(x)=\mathbb{E}_x[h(X_\tau)]$ is a solution of this equation (see \cite[Theorem 6.2.3]{K11}), where $\tau$ equals $\tau_{\partial U}$, the hitting time of the boundary of $U$.  In particular, regularity of $u$ is unknown and the expression $Lu$ is taken in some weak/abstract sense. For instance, it is not clear that for a diffusion processes that $Lu= au''+bu'$, where $u',u''$ are the usual derivatives of $u$. Moreover, 
%in addition to the unknown regularity, explicit equations and solutions,
 scale functions do not depend on the boundary $a,b$ of the interval. Indeed, here $u$ depends on $U$ and $h$, and $S,W$ should depend on $U=[a,b]$ and $h= \mathbf{1}_{[b,+\infty)}$.

Before introducing our definition of a scale function, let us point out some differences between diffusion processes, (spectrally negative) L\'{e}vy processes and our PDMP.  In contrast with both the others processes, the diffusion process has infinite quadratic variation and then $\lim_{x\to a} \mathbb{P}_x(\tau_a > \tau_b)=0$ and $\lim_{x\to b} \mathbb{P}_x(\tau_a > \tau_b)=1$. Also, the L\'{e}vy and the diffusion processes are stochastically monotonous (see \cite[Section 5.9]{K11}), although our PDMP is not in general. These properties entail, for these two processes, the monotonicity of $x\mapsto \mathbb{P}_x(\tau_{(-\infty,a]} > \tau_{[b,+\infty)})$ and the continuity property of the associated scale function. 

We can now give our definition of the scale function for the process $(X_{t})_{t\geq 0}$.

\begin{thm}[Definition of the scale function]
\label{th:def-scale}
Under Assumptions \ref{ass:b,K}, for every $T \in \mathbb{R}_+$, there exists a bounded function $S_{T}$ on $\mathbb{R}_+$ such that $S_{T}(t)=0$, for all $t\geq T$, $S_T$ is absolutely continuous on $[0,T)$\footnote{Usually absolutely continuity is defined on compact intervals, here it means that $S_\mf{T}$ is absolutely continuous over all compact subintervals of $[0,T)$. } and
$$
\forall \mf{t}\in [0,\mf{T}), \quad S_\mf{T}(\mf{t})= e^{-\int_0^\mf{t} b(\mf{u}) d\mf{u}} + e^{-\int_0^\mf{t} b(\mf{u}) d\mf{u}} \int_0^\mf{t} b(\mf{s}) e^{\int_0^\mf{s} b(\mf{u})\, d\mf{u}} \int_{[0,\mf{T}-\mf{s})} S_\mf{T}(\mf{s}+\mf{v})\, K(\mf{s},d\mf{v}) d\mf{s}.
$$
In particular, $S_{\mf{T}}$ belongs to the domain of the extended generator of $(X_{t \wedge \tau_{(-\infty,a]} \wedge \tau_{[b,+\infty)}})_{t\geq 0}$ and $LS_{\mf{T}}(\mf{t})=0$ for almost all $\mf{t}\in [0,\mf{T}]$. In addition, for all $ \mf{s}\leq \mf{t} \leq \mf{T}$, $(S_{\mf{T}} (X_{t \wedge \tau_{(-\infty,\mf{s}]} \wedge \tau_{[\mf{T},+\infty)}}))$ is a Martingale and
\begin{equation}
\label{eq:scale-proba}
\mathbb{P}_\mf{t} (\tau_{[0,\mf{s}]} > \tau_{[\mf{T},+\infty)}) = \frac{S_{\mf{T}}(\mf{t})-S_{\mf{T}}(\mf{s})}{S_{\mf{T}}(\mf{T})-S_{\mf{T}}(\mf{s})}=\frac{S_{\mf{T}}(\mf{s}) - S_{\mf{T}}(\mf{t})}{S_{\mf{T}}(\mf{s})}.
\end{equation}
\end{thm}

\begin{proof}
%Let us extend $K$ and $R$ by setting $R(x)=R(T)$ and $K(x,dy)= K(T,dy)$.
Set $\mathcal{C}= \left\{H \in C[0,\mf{T}] \ | \ H(0)=1, \quad \Vert H \Vert_\infty \leq 1, H\geq 0  \right\}.$
This is a non-empty convex closed bounded subset in the uniformly convex Banach space of continuous functions $C([0,+\infty))$. Let $\mathcal{A}:\mathcal{C} \to C([0,+\infty))$, defined, for every $x\in [0,b]$ and $S\in \mathcal{C}$ by
%$$
%TS(x)= e^{-\int_0^x R(u) du} S(0) + e^{-\int_0^x R(u) du} \int_0^x R(y) e^{\int_0^y R(u) du} \int_0^{b-y} S(y+z) K(y,dz) dy.
%$$
$$
\mathcal{A}S(\mf{t}) = e^{-\int_0^\mf{t} b(\mf{u}) d\mf{u}} + e^{-\int_0^\mf{t} b(\mf{u}) d\mf{u}} \int_0^\mf{t} b(\mf{s}) e^{\int_0^\mf{s} b(\mf{u}) d\mf{u}} \int_{[0,\mf{T}-\mf{s})} S_\mf{T}(\mf{s}+\mf{v}) K(\mf{s},d\mf{v}) d\mf{s}
$$
We have $\mathcal{A}S(0)=1$ and $\Vert \mathcal{A} S \Vert_\infty \leq 1$, thus $\mathcal{A}S \in \mathcal{C}$. The map $\mathcal{A}$ is also non expansive, namely for all $S_1,S_2\in \mathcal{C}$,
$$
\Vert \mathcal{A} S_1 - \mathcal{A} S_{2} \Vert_\infty \leq \Vert S_1 - S_2\Vert_\infty.
$$
Then, by the Browder fixed point theorem $\mathcal{A}$ admits a fixed point $S_\star$ in $\mathcal{C}$. Set $S_{\mf{T}}(\mf{t}) =S_\star(\mf{t}) \mathbf{1}_{\mf{t}\in [0,\mf{T})}$. It is an absolutely continuous function on $[0,\mf{T})$ and then belongs to the extended generator of $(X_{t \wedge \tau_{(-\infty,\mf{s}]} \wedge \tau_{[\mf{T},+\infty)}})$; see \cite[Theorem (26.14) p.69]{D93}) or Theorem \ref{th:genext}. In particular, by the It\^o-Dynkin formula, see \cite[Section 26, page 66]{D93}, the process $(S_{\mf{T}} (X_{t \wedge T_{(-\infty,\mf{s}]} \wedge T_{[\mf{T},+\infty)}}))$ is a martingale. Since hitting times are finite  and this martingale is bounded, then using the stopping time theorem and the dominated convergence theorem, we find the expected formula \eqref{eq:scale-proba}.
%Moreover if $K$ is strong Feller (or at least $K(\mathbf{1}_{<b})$ Feller pour tt $b$, ou $x\mapsto K(x,[0,b-x])$ continue) then $H_\star$ is $C^1$.
\end{proof}

For L\'{e}vy processes, we can set $W_{\mf{T}}:\mf{t}\mapsto S_{\mf{T}}(\mf{T}-\mf{t}) \mathbf{1}_{\mf{t}<\mf{T}}$ and using the invariance by translation property of these processes, we recover that $W_{\mf{T}}$ does actually not depend on $\mf{T}$.

\iffalse
For L\'{e}vy processes, we have $S_{\mf{T}}:\mf{t}\mapsto W(\mf{T}-\mf{t}) \mathbf{1}_{x<\mf{T}}$(which actually depends on $b$). However, from the last theorem, setting $W(t) = S_\mf{T}(\mf{T}-\mf{t})$ in the L\'{e}vy case, one can use the invariance by translation property to recover that $W$ does actually not depend on $\mf{T}$.
\fi

We also naturally recover the natural property that $S_{\mf{T}}$ is constant (namely $\mathbb{P}_\mf{t}(\tau_{[0,\mf{s}]} > \tau_{[\mf{T},+\infty)}) =0$) on $[0,\mf{T})$ if and only if $K(\mf{t},[0,\mf{T}))=1$ for all $\mf{t}\in [0,\mf{T}]$.

\begin{cor}[Integro-differential equation for $S_{\mf{T}}$]
\label{cor:Eqscale}
If $K$ is Feller, $b$ and $\mf{t}\mapsto K(\mf{t},[0,\mf{T}-\mf{t}))$ are continuous then $S_{\mf{T}} \in C^1([0,\mf{T}) \cup (\mf{T}, + \infty))$ and
\begin{equation}
\label{eq:gen-edo-scale}
\forall \mf{t} \in [0,\mf{T}), \quad -S'_{\mf{T}}(\mf{t}) + b(\mf{t}) \left(\int_{[0, \mf{T}-\mf{t})} S_{\mf{T}}(\mf{t}+\mf{s}) K(\mf{t},d\mf{s}) -S_{\mf{T}}(\mf{t}) \right)=0.
\end{equation}
\end{cor}

Note that if $K(\mf{t},d\mf{s})= \varphi(\mf{t},\mf{s}) d\mf{s}$, with a continuous density $\varphi$, then $K$ is Feller and $\mf{t}\mapsto K(\mf{t},[0,\mf{T}))$ is continuous.

\begin{cor}[Monotony properties of the scale function]
\label{cor:scal-mon}
Under Assumption~\ref{ass:b,K}, for every $\mf{T}\geq \mf{t} \geq 0$, we have 
$$
S_{\mf{T}}(\mf{t})=\mathbb{P}_\mf{t} (\tau_{0} < \tau_{[\mf{T},+\infty)}).
$$
In particular $\mf{t} \mapsto S_{\mf{T}}(\mf{t})$ is decreasing and  $\mf{T} \mapsto S_{\mf{T}}(\mf{t})$ is increasing.
\end{cor}

\subsection{Tree properties though scale function}
\label{ssect:scale-tree}

We can now express certain properties of the tree through this new defined scale function. We set $\Xi_t$ the number of individuals at time $\mf{t}$.

\begin{prop}[Number of individuals at a fixed time]
Under Assumption~\ref{ass:b,K}, for every  $ \mf{t_0}\leq \mf{t}$, we have
\begin{equation}
\label{eq:CMJ0}
\mathbb{P}_{\mf{t}_0} \left( \Xi_{\mf{t}}=0\right) =\mathbb{P}_{\mf{t}_0} \left( \mathcal{H}(\mathbb{T}) \leq \mf{t}\right) =S_{\mf{t}}(\mf{t}_0).%\frac{S_{\mf{t}}(\mf{t}_0)}{S_{\mf{t}}(0)}.%==\frac{S_{\mf{T}}(\mf{s})}{S_{\mf{T}}(\mf{s})}%= P(T_0< T_{(\tau,+\infty)}).
\end{equation}
Moreover, conditionally on $\{\Xi_{\mf{t}}\neq 0\}$, $\Xi_\mf{t}$ is geometrically distributed; that is
\begin{equation}
\label{eq:CMJk}
\mathbb{P}_{\mf{t}_0} \left( \Xi_{\mf{t}}=k \ | \ \Xi_{\mf{t}}\neq 0 \right) = S_{\mf{t}}(\mf{t} -) (1-S_{\mf{t}}(\mf{t}-))^k
\end{equation}
\end{prop}

This result is a direct adaptation of \cite[Proposition 5.6]{Lambert2010}. The proof is written for sake of completeness.

Using \eqref{eq:CMJ0}, we recover the formula  established in \cite[Proposition 4]{LS13} and \cite[Proposition 3.2.6]{L17}. We even improve this result, because we establish the regularity of the function $S_{\mf{T}}$ and precise the sense of the derivative. We also detail more precisely the law of $(\Xi_{\mf{t}})_{t\geq 0}$.

\begin{proof}
From the definition ,of the contour process, the heigh of the tree is lower than $\mf{t}$ if and only if $(X_{\mf{t}})_{\mf{t}\geq 0}$ hits $0$ before $(\mf{t},\infty)$. Theorem~\ref{th:def-scale} then entails Equation~\eqref{eq:CMJ0}.
Now, as $\Xi_{\mf{t}}$ correspond of the number of times that $(X_t)_{t\geq 0}$ hits $\mf{t}$. By the Markov property, it is then a sequence of i.i.d. excursions of $(X_{\mf{t}})_{\mf{t}\geq0}$ from $\mf{t}$ on $(0,\mf{t}]$, stopped at the first one that exits it from the bottom.  Hence, from Theorem~\ref{th:def-scale}, we obtain \eqref{eq:CMJk}. 
\end{proof}

\begin{cor}[Extinction probability]
\label{cor:ext}
Under Assumption~\ref{ass:b,K}, for every $\mf{t}_0\geq 0$, we have
$$
\mathbb{P}_{\mf{t}_0} \left( \text{Ext}\right)= \mathbb{P}_{\mf{t}_0} \left( \exists \mf{t}\geq0, \ \Xi_{\mf{t}}=0\right)=  \mathbb{P}_{\mf{t}_0} \left( \mathcal{L}(\mathbb{T}) <\infty \right)= \lim_{\mf{t}\to \infty} S_{\mf{t}}(\mf{t}_0).
$$
In particular, $S:\mf{t_0}\mapsto \lim_{\mf{t}\to \infty} S_{\mf{t}}(\mf{t}_0)$ exists, is decreasing and it is solution to the functional equation
$$
\forall\mf{t}\geq 0, \ S(\mf{t})= e^{-\int_0^\mf{t} b(\mf{u}) d\mf{u}} + e^{-\int_0^\mf{t} b(\mf{u}) d\mf{u}} \int_0^\mf{t} b(\mf{s}) e^{\int_0^\mf{s} b(\mf{u}) d\mf{u}} \int_{[0,+\infty)} S(\mf{s}+\mf{v}) K(\mf{s},d\mf{v}) d\mf{s},
$$
\end{cor}
\begin{proof}
Function $S$ is well-defined and decreasing thanks to Corollary~\ref{cor:scal-mon} and the functional equation holds because of the Beppo Levi Theorem.
\end{proof}

Again under the regularity assumptions of Corollary~\ref{cor:Eqscale}, $S$ is solution to the integro-differential equation 
\begin{equation}
\forall \mf{t} \geq 0, \quad S'(\mf{t}) = b(\mf{t}) \left(\int_{[0, +\infty)} S(\mf{t}+\mf{s}) K(\mf{t},d\mf{s}) -S(\mf{t}) \right).
\end{equation}

In the extinction case, as in the homogeneous setting (\cite[Proposition 5.8]{Lambert2010}, it is possible to establish Yaglom-type limiting result; that means convergence to a a quasi-stationary type distribution. This is not rigorously the case because we recall that $(\Xi_{\mf{t}})_{t\geq 0}$ is not a Markov process.

\begin{cor}[Quasi-limiting behavior of $(\Xi_t)$]
\label{cor:qsd}
Under Assumption~\ref{ass:b,K}, we have
\begin{itemize}
\item If $\lim_{\mf{t}\to \infty} S_{\mf{t}}(\mf{t}-)=q\in (0,1)$ then
\begin{equation}
\label{eq:CMJ-qsd}
\lim_{\mf{t}\to \infty} \mathbb{P}_{\mf{t}_0} \left( \Xi_{\mf{t}}=k \ | \ \Xi_{\mf{t}}\neq 0 \right) = q(1-q)^k
\end{equation}
\iffalse
\item If $\lim_{\mf{t}\to\infty} S_{\mf{t}}(\mf{t}-)=1$ then
\begin{equation}
\label{eq:CMJ-qsd-crit}
\forallx>0, \quad \lim_{\mf{t}\to \infty} \mathbb{P}_{\mf{t}_0} \left( \Xi_{\mf{t}} (1-S_{\mf{t}}(\mf{t}-))^{-1}>x \ | \ \Xi_{\mf{t}}\neq 0 \right) = e^{-x}
\end{equation}
\fi
\item If $\lim_{\mf{t}\to\infty} S_{\mf{t}}(\mf{t}-)=0$ then
\begin{equation}
\label{eq:CMJ-qsd-surcrit}
\forall x>0, \quad \lim_{\mf{t}\to \infty} \mathbb{P}_{\mf{t}_0} \left( \Xi_{\mf{t}} S_{\mf{t}}(\mf{t}-)>x \ | \ \Xi_{\mf{t}}\neq 0 \right) = e^{-x}.
\end{equation}
\end{itemize}
\end{cor}

In case of L\'{e}vy process, it is direct that $\mf{t} \mapsto S_{\mf{t}}(\mf{t}-)$ is decreasing and then converges although it is not generally the case. Thus, the previous corollary  does not give an exhaustive description of possible behaviors.

However thanks to Corollary~\ref{cor:scal-mon}, we have, for all $\mf{s}\geq 0$,
$$
\lim_{t \to \infty} S_{\mf{t}} (\mf{t}-)\leq \lim_{t \to \infty} S_{\mf{t}} (\mf{s})=S(s)=\mathbb{P}_{\mf{s}}(\text{Ext}),
$$
and then $\lim_{t \to \infty} S_{\mf{t}} (\mf{t}-)\leq \lim_{t \to \infty} \mathbb{P}_{\mf{t}}(\text{Ext})$. In particular, this gives a sufficient condition to verify \eqref{eq:CMJ-qsd-surcrit} or an upper bound for $q$ in \eqref{eq:CMJ-qsd}.

Finally Corollary~\ref{cor:qsd} (Equation~\eqref{eq:CMJ-qsd-surcrit}) gives also the deterministic growth under non-extinction. Indeed, as 
$$
\mathbb{P}_{\mf{t}_0} \left( \text{Ext}\right)=\lim_{\mf{t}\to \infty} \mathbb{P}_{\mf{t}_0} \left( \Xi_{\mf{t}}=0\right)= \lim_{\mf{t}\to \infty} S_{\mf{t}}(\mf{t}_0),
$$
we have
$$
\lim_{\mf{t}\to \infty} \mathbb{P}_{\mf{t}_0} \left( \Xi_{\mf{t}} S_{\mf{t}}(\mf{t}-)>x \ | \ \Xi_{\mf{t}}\neq 0 \right)=\lim_{\mf{t}\to \infty} \mathbb{P}_{\mf{t}_0} \left( \Xi_{\mf{t}} S_{\mf{t}}(\mf{t}-)>x \ | \ \text{Ext}^c \right).
$$
Hereafter, we will say that the tree is supercritical when $\mathbb{P}_x(\mathcal{L}(\mathbb{T})=+\infty)>0$, for every $x\geq 0$. Several details are given in the next section. In the supercritical case, one can go further that this result by looking the tree conditioning on being finite or infinite. Indeed, From Corollary~\ref{cor:ext} (see also the proof of Theorem~\ref{prop:htrans} below), $S$ is an harmonic function, which is non-trivial, under the supercritical assumption, and one can then use a Doob $h$-transform (a Cameron-Martin-Girsanov type change of measure, see \cite{D57} or \cite[Page 83]{RW00}) to define a new (Markov) process which is distributed as the process conditioned on hitting $0$ (or never hit it).

Before going further, let us introduce some notation.
For any positive measurable function $f$ and $g$, we set
\[
\forall x\geq 0, %\qquad Kf(x)=\int_{\mathbb{R}_{+}}f(x+y)\, K(x,dy), \qquad
\quad K^g(f):=K(gf)/K(g).
\]
We can now state our result on conditioned ISTs.
\iffalse
\begin{prop}[Conditioned tree]
\label{prop:htrans}
Suppose that the tree is supercritical and Assumption~\ref{ass:b,K} holds. Then
\begin{itemize}
\item The law of $(X_t)_{t\geq 0}$ conditionally on $\text{Ext}$\cmb{$=\mathcal{L}(\mathbb{T})<+\infty)$} is the same as the PDMP decreasing linearly at rate $1$, with jumps rate $bKS/S$ and jump kernel $K^S$. 
\item The law of $(X_t)_{t\geq 0}$ conditionally on $\text{Ext}^c$\cmb{$=\mathcal{L}(\mathbb{T})<+\infty)$} is the same as the PDMP decreasing linearly at rate $1$, with jumps rate $b(1-KS)/(1-S)$ and jump kernel $K^{(1-S)}$. 
\end{itemize}
\end{prop}
\fi

\begin{thm}[Conditioned supercritical tree]
\label{prop:htrans}
Suppose that $\mathbb{T}$ is supercritical and Assumption~\ref{ass:b,K} holds. 
We have
\begin{itemize}
\item The law of $\mathbb{T}$ conditionally on $\text{Ext}=\{\mathcal{L}(\mathbb{T})<+\infty\}$ is the same as an IST with birth rate $bKS/S$ and death kernel $K^S$. 
\item The law of $\mathbb{T}$ conditionally on $\text{Ext}^c=\{\mathcal{L}(\mathbb{T})=+\infty\}$ is the same as a IST with birth rate $b(1-KS)/(1-S)$ and death kernel $K^{(1-S)}$. 
\end{itemize}
\end{thm}

\begin{proof}
Let $(P_t)_{t\geq 0}$ be the Markov semigroup of $(X_t)_{t\geq 0}$. Using the Markov property, we easily see that for every $\mf{x}\geq 0$ that 
$$
\forall t\geq0, \quad P_{t} S(\mf{x})= \mathbb{E}_{\mf{x}}\left[\mathbb{P}_{X_{t}}(\exists s\geq 0, \ X_{s}=0)\right] = \mathbb{P}_{\mf{x}}\left(\exists s\geq 0, \ X_{s}=0)\right) = S(\mf{x}).
$$
This shows that the function $S$ is an harmonic function. Let $(P^S_t)_{t\geq0}$ defined, for every continuous and bounded functions $f:\mathbb{R}_+^* \to \mathbb{R}$ vanishing at $0$ (and that we extend in $0$ by $f(0)=0$),   by
$$
P^S_t f=\frac{P_t(S f)}{P_t S} = \frac{P_t(S f)}{S}
$$
is a Markov semigroup on $\mathbb{R}_+^*=(0,+\infty)$. We have, for all $f$ such that $f/S\in D(L)$, $\partial_t P^S_t f= P^S_t L^S f$ where,
$$
L^S f(\mf{x})= \frac{L(fS)(\mf{x})}{S(\mf{x})}=-f'(\mf{x}) + \frac{b(\mf{x}) KS(\mf{x})}{S(\mf{x})} \left(\frac{K(S f)(\mf{x})}{KS (\mf{x})} - f(\mf{x})  \right).
$$
 This semigroup then correspond to the semigroup of the PDMP which decreases linearly (with rate $1$) and jumps at rate $b KS/S$ with kernel $K^S$. Indeed this is a slight variation of Corollary~\ref{cor:mmloi}.
 
 Using again the Markov property, we have, for every continuous and bounded functions $f: \mathbb{R}_+^* \to \mathbb{R}$, $t>0$ and $\mf{x}>0$,
$$
P_r^S f (\mf{x})= \frac{\mathbb{E}_{\mf{x}}\left[ f(X_t) \mathbb{P}_{X_{t}}(\exists s\geq 0, \ X_{s}=0)\right] }{\mathbb{P}_{\mf{x}}\left(\exists s\geq 0, \ X_{s}=0\right) }= \frac{\mathbb{E}_{\mf{x}}\left[ f(X_t) \mathbf{1}_{\{\exists s\geq t, \ X_{s}=0\}} \right] }{\mathbb{P}_{\mf{x}}\left(\exists s\geq 0, \ X_{s}=0\right) } = \mathbb{E}_{\mf{x}}\left[ f(X_t) \ | \ \exists s\geq 0, \ X_{s}=0\right].
$$
In other words, the law of $X_t$ conditioned to the extinction is described by $\delta_t P_{t}$. By successive conditioning, it gives that the law of the process $(X_t)_{t\geq 0}$ conditionally on extinction is the same as the Markov process with semigroup $(P_t)_{t\geq 0}$.

The proof is similar when conditioning on non-extinction with $h=1-S$ as harmonic function. We then set $(P_t^h)$ and $\mathcal{L}^h$ the operators $(P_t(\cdot \times h)/h)_{t\geq 0}$ and $\mathcal{L}(\cdot \times h)/h$.  The only difference is that $h$ is not bounded around $0$, and we need to prove that the associated Markov process does not explode. But as $\mathcal{L}^h(1/h) =0$, \cite[Theorem 2.1]{MT93III} ends the proof.
% It correspond to the la
\end{proof}

We also have an equivalent theorem when conditioning any tree on finite biologic time window.

\begin{thm}[Conditioned tree on finite windows]
\label{prop:htrans-fixe}
Let $T>0$ and suppose Assumption~\ref{ass:b,K} holds, we have
\begin{itemize}
\item The law of $\mathbb{T}$ conditionally on  $\mathcal{H}(\mathbb{T}) \leq T$ is the same as an IST with birth rate $bKS_T/S_T$ and death kernel $K^{S_T}$. 
\item Assume further that $S_T(x)<1$, for all $x\geq 0$, the law of $\mathbb{T}$ conditionally on $\mathcal{H}(\mathbb{T}) > T$ is the same as a IST with birth rate $b(1-KS_T)/(1-S_T)$ and death kernel $K^{(1-S_T)}$. 
\end{itemize}%\cmb{verifier si les $>$ ou $\geq$ sont les bons pour $H\geq T$}
\end{thm}

Note that sufficient (and almost necessary) condition for $S(x)<1$ for all $x\geq 0$ will be given by Theorem~\ref{th:dichotomy}.

\begin{proof}
The proof is the same as Theorem~\ref{prop:htrans}.
\end{proof}

\subsection{Applications to particular cases}
\label{ssec:example-scale}

For L\'evy processes, there exists a large amount of cases where the scale function is explicit, see for instance \cite{scaleFunctionKyp}. In this subsection, we give some explicit solutions of Equation~\eqref{eq:gen-edo-scale} for time-inhomogeneous versions of classical examples but also for specially time-inhomogeneous tree. We also detail some outcome of Subsection~\ref{ssect:scale-tree} where the scale function are known.

\subsubsection{Deterministic life duration}
Here we assume that $K(\mf{t},d\mf{u}) = \delta_1$, meaning that individuals live a fixed time equals to $1$. The map $\mf{t}\mapsto K(\mf{t},[0,\mf{T}-\mf{t}]) = \mathbf{1}_{\{\mf{t}+1\leq \mf{T}\}}$ is not continuous. We have
\begin{align*}
 S_{\mf{T}} (\mf{t}) 
 &=e^{- \int_0^\mf{t} R(\mf{u}) d\mf{u}} + e^{- \int_0^\mf{t} R(\mf{u}) d\mf{u}} \int_0^\mf{t} R(\mf{s}) e^{\int_0^\mf{s} R(\mf{u}) d\mf{u}} S_{\mf{T}}(\mf{s}+1) \mathbf{1}_{\mf{s}+1 < \mf{T}} d\mf{s} \\
 &=e^{- \int_0^\mf{t} R(\mf{u}) d\mf{u}}+ e^{- \int_0^\mf{t} R(\mf{u}) d\mf{u}} \int_0^{\mf{t}\wedge \mf{T}-1}  R(\mf{s}) e^{\int_0^\mf{s} R(\mf{u}) d\mf{u}} S_{\mf{T}}(\mf{s}+1)d\mf{s}.
\end{align*}
In particular, for $\mf{t} \in [\mf{T}-1,\mf{T})$, $S_{\mf{T}}(\mf{t}) = e^{- \int_0^\mf{t} R(\mf{u}) d\mf{u}} C$, for some $C>0$ and then we recover the property $\mathbb{P}_\mf{t}(T_{[0,\mf{T}-1]} < T_{[\mf{T},+\infty)})=e^{- \int_{(\mf{T}-1)}^\mf{t} R(\mf{u}) d\mf{u}}= S_{\mf{T}}(\mf{t})/S_{\mf{T}}(\mf{T}-1)$ that can be obtain by regarding if the process has a jump before hitting $\mf{T}-1$.

\subsubsection{Fixed death moment}
Let us consider fixed death times; for instance at times $1$ or $2$.  For $x\in [0,1]$,
$$
K(\mf{t},d\mf{u}) = \frac{1}{2} \delta_{1-\mf{t}} (d\mf{u}) + \frac{1}{2} \delta_{2-\mf{t}}(d\mf{u}) 
$$
and for $\mf{t}\in (1,2]$, $K(\mf{t},d\mf{u}) =  \delta_{2-\mf{t}}(d\mf{u}).$ Then for $\mf{t}\leq1$,%\cmb{Ce modèle doit s etudier plus simplement car si on regarde les individus à des temps discret, ici $1,2$, il y a tout d'abord un nombre Poissonnien d enfants puis une binomiale de morts, ca doit faire une chaine facile a etudier}
\begin{align*}
S_{\mf{T}} (x) 
&=e^{- \int_0^\mf{t} R(\mf{u}) d\mf{u}} + e^{- \int_0^\mf{t} R(\mf{u}) d\mf{u}} \int_0^\mf{t} R(\mf{s}) e^{- \int_0^\mf{s} R(\mf{u}) d\mf{u}} \frac{1}{2} (S(1) \mathbf{1}_{\mf{T}\geq 1} + S(2) \mathbf{1}_{\mf{T}\geq 2}) d\mf{s}\\
&= e^{- \int_0^\mf{t} R(\mf{u}) d\mf{u}} + \left(1 - e^{- \int_0^\mf{t} R(\mf{u}) d\mf{u}} \right) C_\mf{T},
\end{align*}
for some $C_\mf{T}>0$. For $\mf{t}\in [1,2)$ %(à voir les bords=)
\begin{align*}
S_{\mf{T}}(x) &=e^{- \int_0^\mf{t} R(\mf{u}) d\mf{u}} + e^{- \int_0^\mf{t} R(\mf{u}) d\mf{u}}\left( \int_0^1 R(\mf{s}) e^{- \int_0^\mf{s} R(\mf{u}) d\mf{u}} C_\mf{T} d\mf{s} + \int_1^\mf{t} R(\mf{s}) e^{- \int_0^\mf{s} R(\mf{u}) d\mf{u}} S(2) \mathbf{1}_{\mf{T}\geq 2} d\mf{s} \right) \\
&= e^{- \int_0^\mf{t} R(\mf{u}) d\mf{u}} +  \left(1 - e^{- \int_0^1 R(\mf{u}) d\mf{u}} \right) C_\mf{T} + \left(e^{- \int_0^1 R(\mf{u}) d\mf{u}} - e^{- \int_0^\mf{t} R(\mf{u}) d\mf{u}} \right) c_\mf{T},
\end{align*}
for some $c_\mf{T}>0$. If $S(2)=0$ then $c_\mf{T}=0, C_\mf{T}=2 e^{- \int_0^1 R(\mf{u}) d\mf{u}}  /\left(1 + e^{- \int_0^1 R(\mf{u}) d\mf{u}} \right)$, and for all $\mf{t}\in [0,2)$,
$$
S_{\mf{T}}(x) = e^{- \int_0^\mf{t} R(\mf{u}) d\mf{u}} + 2 e^{- \int_0^1 R(\mf{u}) d\mf{u}}  \frac{ 1 - e^{- \int_0^{\mf{t}\wedge 1} R(\mf{u}) d\mf{u}} }{1 + e^{- \int_0^1 R(\mf{u}) d\mf{u}} }.
$$
In particular $S$ is not differentiable at $\mf{t}=1$.

\subsubsection{Inhomogeneous-time Markovian tree}
Let $K(\mf{t},d\mf{u}) =d(\mf{t}+\mf{u}) e^{-\int_\mf{t}^{\mf{t}+ \mf{u}} d(\mf{v}) d\mf{v} } d\mf{u}$ and the equation of Corollary \ref{cor:Eqscale} reads
\begin{align}
S'_{\mf{T}}(x) 
&= b(\mf{t}) \left(\int_0^{\mf{T}-\mf{t}} S_{\mf{T}}(\mf{t}+\mf{s}) d(\mf{t}+\mf{s}) e^{-\int_\mf{t}^{\mf{t}+ \mf{s}} d(\mf{u}) d\mf{u} } d\mf{s} -S_{\mf{T}}(\mf{t}) \right)\nonumber\\
&=b(\mf{t}) \left(\int_\mf{t}^{\mf{T}} S_{\mf{T}}(\mf{s}) d(\mf{s}) e^{-\int_\mf{t}^{\mf{s}} d(\mf{u}) d\mf{u} } d\mf{s} -S_{\mf{T}}(\mf{t}) \right).\label{eq:scalemarkov}
\end{align}
In particular, if $b,d$ are $C^1$ then
 \iffalse
\begin{align*}
S''_{\mf{T}}(\mf{t}) &= b'(\mf{t}) \left(\int_0^{T-\mf{t}} S_{\mf{T}}(\mf{t}+\mf{s}) d(\mf{t}+\mf{s}) e^{-\int_\mf{t}^{\mf{t}+ \mf{s}} d(\mf{u}) d\mf{u} } d\mf{s} -S_{\mf{T}}(\mf{t}) \right)\\
&- b(x) S_{T}(x) d(x) + b(x) \int_x^{T} S_{T}(y) d(y) d(x) e^{-\int_x^{y} d(u) du } dy -b(x) S_T'(x).
\end{align*}
\fi
\begin{align*}
S''_{\mf{T}}(\mf{t}) &= b'(\mf{t})  \left(\int_\mf{t}^{\mf{T}} S_{\mf{T}}(\mf{s}) d(\mf{s}) e^{-\int_\mf{t}^{\mf{s}} d(\mf{u}) d\mf{u} } d\mf{s} -S_{\mf{T}}(\mf{t}) \right)\\
&- b(\mf{t}) S_{\mf{T}}(\mf{t}) d(\mf{t}) + b(\mf{t}) \int_\mf{t}^{\mf{T}} S_{\mf{T}}(\mf{s}) d(\mf{s}) d(\mf{t}) e^{-\int_\mf{t}^{\mf{s}} d(\mf{u}) d\mf{u} } d\mf{s} -b(\mf{t}) S_\mf{T}'(\mf{t}).
\end{align*}
Using Equation \eqref{eq:scalemarkov}, this can be simplified to
$
S''_\mf{T}(\mf{t})= \frac{b'(\mf{t})}{b(\mf{t})} S_\mf{T}'(\mf{t}) +d(\mf{t}) S_\mf{T}'(\mf{t})  - b(\mf{t}) S'_{\mf{T}}(\mf{t})
$
with also the following boundary condition:
$
S_\mf{T}(0)=1, \ S'_\mf{T}(\mf{T})= -b(\mf{T})S(\mf{T}).
$
This simple equation can be solved. Indeed
$$
S'_\mf{T}(\mf{t}) =-b(\mf{T}) S_\mf{T}(\mf{T}) e^{-\int_\mf{t}^\mf{T} \frac{b'(\mf{u})}{b(\mf{u})} +d(\mf{u}) - b(\mf{u}) d\mf{u} },
$$
and
$$
S_\mf{T}(\mf{t})= 1 -b(\mf{T}) S(\mf{T}) \int_0^\mf{t} e^{-\int_\mf{s}^\mf{T} \frac{b'(\mf{u})}{b(\mf{u})} +d(\mf{u}) - b(\mf{u}) d\mf{u} } d\mf{s}.
$$
Using $\mf{t}=\mf{T}$, we finally have
\begin{align*}
S_\mf{T}(\mf{t})
&= 1 -\frac{b(\mf{T}) \int_0^\mf{t} e^{-\int_\mf{s}^\mf{T} \frac{b'(\mf{u})}{b(\mf{u})} +d(\mf{u}) - b(\mf{u}) d\mf{u} } d\mf{s}}{1+b(\mf{T}) \int_0^\mf{T} e^{-\int_\mf{s}^\mf{T} \frac{b'(\mf{u})}{b(\mf{u})} +d(\mf{u}) - b(\mf{u}) d\mf{u} } d\mf{s}}
=\frac{1+b(\mf{T}) \int_\mf{t}^\mf{T} e^{-\int_\mf{s}^\mf{T} \frac{b'(\mf{u})}{b(\mf{u})} +d(\mf{u}) - b(\mf{u}) d\mf{u} } d\mf{s}}{1+b(\mf{T}) \int_0^\mf{T} e^{-\int_\mf{s}^\mf{T} \frac{b'(\mf{u})}{b(\mf{u})} +d(\mf{u}) - b(\mf{u}) d\mf{u} } d\mf{s}}\\
&=\frac{1+ \int_\mf{t}^\mf{T} b(\mf{s}) e^{-\int_\mf{s}^\mf{T} d(\mf{u}) - b(\mf{u}) d\mf{u} } d\mf{s}}{1+ \int_0^\mf{T} b(\mf{s}) e^{-\int_\mf{s}^\mf{T} d(\mf{u}) - b(\mf{u}) d\mf{u} } d\mf{s}}.
\end{align*}
In particular, we recover \cite[Exercise 3.2.8]{L17}. When $b\neq d$ are constant, we have
$$
S_\mf{T}(\mf{t}) =  \frac{d-b e^{(b-d)(\mf{T}-\mf{t})}}{d-be^{(b-d)\mf{T}}},
$$
so we recover the classical function \cite{CLR},
$$W(t)=\frac{d-b e^{(b-d)\mf{t}}}{d-b}.$$

\subsubsection{Supercritical time-homogeneous tree}

Let us consider the case of a standard homogeneous splitting tree. In such situation, $b$ is constant and the kernel $K(dy)$ is not time-dependent. Moreover, if it is supercritical then we have $S(t)=e^{-\alpha t}$ for some $\alpha>0$.

In such situation the semigroup of the contour process of the tree conditioned on $\text{Ext}$ provided by Theorem~\ref{prop:htrans} is given by $L^S$ reading as
\begin{align*}
L^Sf(x) &= -f'(y)+  b\int_{\mathbb{R}_{+}} e^{-\alpha y} K(dy) \left(\frac{\int_{\mathbb{R}_{+}} e^{-\alpha y} f(x+y) K(dy)}{\int e^{-\alpha y} K(dy)} - f(x) \right).\\
\end{align*}
In this case, we recover, the results of \cite{Lambert2010}. A more interesting situation is the case of splitting tree conditioned on non-extinction. In this case we get that the generator of the contour of the conditioned tree reads
$$
L^h f(x) = -f'(y)+  b\frac{\int_{\mathbb{R}_{+}} (1-e^{-\eta (x+y)}) K(dy)}{(1-e^{-\eta x})} \left(\frac{\int_{\mathbb{R}_{+}} (1-e^{-\eta(x+y)}) f(x+y) K(dy)}{\int (1-e^{-\eta (x+y)}) K(dy)} - f(x) \right),
$$
which shows that the conditioned time-homogeneous tree becomes time-inhomogeneous. However, it tends to recover its homogeneity in the long time limit (in a heuristic sense). Also, a Markovian tree $K(dy)= de^{-dy}$ looses the Markov property and becomes an IST (when conditioned on non-extinction).

%\section{Lyapounov functions, hitting times and criticality}
\section{Lyapounov functions: sufficient condition and tail estimates}
\label{sec:lyap}

Properties of the tree given in Subsection~\ref{ssect:scale-tree} depends crucially on the scale function. In this section, we adopt another point of view based on usual reducibility conditions for Markov processes. They are detailed in the following section. A subsection with some application ends the present section.

\subsection{General condition}

The aim of this section is to generalize \cite[Proposition 2.2 p. 12]{Lambert2010} which links the probability to extinction (finiteness of the tree) to the drift of the contour process. However, in contrast with the case where the contour is a L\'evy process (see for instance \cite[Corollary 2 p. 190]{B98}), there is no simple formulation on the drift describing the long time behaviour of a general PDMP. 

We begin by introducing the following necessary condition for supercriticality.

%\noindent
\begin{assu}[No age barrier]
 \label{hyp:ageBarrier}
$
\mathbb{P}_{x}\left(\exists t>0, \ X_t=y \right)>0,\quad  \forall x,y>0.
$ 

\end{assu}

\iffalse
\begin{rem}
This last assumption can fail in some pathological situations. For instance, if new individual lifetime became shorter and shorter or if $b$ is null on large time interval (or a combination of these two phenomena). It is difficult to write a necessary and sufficient condition on $b$ and $K$ for this assumption. A sufficient one is 
\[
\left\{
\begin{array}{ll}
b(x)>0,&\quad \forall x\geq 0,\\
K(x,[y,\infty))>0,&\quad \forall x,y\geq 0.
\end{array}
\right.
\]
\end{rem}

\fi

As usual in branching theory, we still have

\begin{thm}[Extinction and explosion: the merciless dichotomy]
\label{th:dichotomy}
Let $\tau_{0}$ be the hitting time of $0$. Then, under Assumption~\ref{ass:b,K} and Assumption~\ref{hyp:ageBarrier}, 
$$
\mathbb{P}_x( \tau_0 < + \infty \text{ or } \lim_{t\to \infty} X_t=+\infty)=1.
$$
Equivalently either the height $\mathcal{H}(\mathbb{T})$ of the tree is infinite (and then the length $\mathcal{L}(\mathbb{T})$ is also infinite) or $\mathcal{L}(\mathbb{T})$ is finite. Moreover, 
$$
\exists x>0, \ \mathbb{P}_x( \tau_0 = + \infty)>0 \Leftrightarrow \forall x>0, \ \mathbb{P}_x( \tau_0 = + \infty)>0.
$$
Equivalently the definition of the supercriticality of the tree does not depend on the initial age.
 \end{thm}

\begin{proof}
From Theorem \ref{thm:genContour}, it is enough to prove that either the hitting time $\tau_0$ of $0$ is finite or $\lim_{t\to \infty} X_t=+\infty$. But as the process $(X_t)_{t\geq 0}$ is a $\delta_0-$ irreducible (from Assumption~\ref{hyp:ageBarrier}) $T-$process, this is a direct consequence of  \cite[Theorem 3.2]{MT93II} ($\delta_0-$Harris recurrence is equivalent to absorption in finite time). The second point is direct.

\end{proof}
\begin{rem}
Remark that Assumption~\ref{ass:b,K} implies the local boundedness of $b$. This implies that there can only be a finite number of individuals in the truncated tree $\mathbb{T}^{(T)}$, which implies $\mathcal{L}(\mathbb{T}^{(T)})<\infty$. In particular, $\mathcal{L}(\mathbb{T})<\infty$ if and only if $\mathcal{H}(\mathbb{T})<\infty$. Let us insist on the fact that this relies on the local boundedness of $b$.
\end{rem}

From the Theorem~\ref{th:dichotomy} and from the classical definition for Galton-Watson processes or homogeneous splitting tree, we can define the typical behaviors of the population.

\begin{enumerate}
\item The supercritical case:  if one of the following equivalent assertions holds
\begin{itemize}
\item For all $x>0$, $\mathbb{P}_x(\lim_{t\to \infty} X_t=+\infty)=1-\mathbb{P}_x(\tau_0<+\infty)>0,$
\item $\mathbb{P}_{x}(\mathcal{L}(\mathbb{T}) =+\infty)>0,$
\item $\mathbb{P}_{x}\left( \mathcal{H}(\mathbb{T})=+\infty \right)>0$.
\end{itemize}
\item The critical case:  if one of the following equivalent assertions holds
\begin{itemize}
\item  $\tau_0<+\infty$ almost surely but $\mathbb{E}_x[\tau_0]= + \infty$, for all $x>0$.
\item $\mathbb{P}_{x}(\mathcal{L}(\mathbb{T}) =+\infty))$ but $\mathbb{E}_{x}[\mathcal{L}(\mathbb{T})]<+\infty$, for all $x>0$.
\end{itemize}
\item The subcritical case:  if one of the following equivalent assertions holds
\begin{itemize}
\item $\tau_0<+\infty$ almost surely and $\mathbb{E}_x[ \tau_0]< + \infty$, for all $x>0$.
\item $\mathbb{E}_{x}[\mathcal{L}(\mathbb{T})]<+\infty$, for all $x>0$.
\end{itemize}
\end{enumerate}
This division also comes from the different recurrence notions for the contour. Note that the situation is not as homogeneous-time Galton-Watson for whose we have
$$
\mathbb{E}[\tau_0]<\infty \Rightarrow \exists \theta>0, \ \mathbb{E}_\mf{t}[e^{\theta \tau_0}]< + \infty.
$$
Let us also highlight that, as in the Galton-Watson case, a non-supercritical population almost surely extinct. Furthermore, even if the notion of supercritical is clear (in view of our results), it is not clear for the moment if we chose the good notion of subcriticality or criticality.

\bigskip

As pointed out in the beginning of this section, it is not easy to have simple condition to ensure in which case we are. Let us end this section by giving sufficient condition using drift-type conditions. These conditions are based on classical Lyapunov functions. Even they may not be optimal, their proof gives a method to show how exploit the contour to answer some questions on the size of the tree.

First of all, let us set, for all $x\geq 0$,
\begin{equation}
\label{eq:moments}
%\left\{
%\begin{array}{l}
m(x) = \int_{\mathbb{R}_{+}}y\,  K(x,dy), \quad m_p(x) = \int_{\mathbb{R}_{+}} |y|^p \, K(x,dy), \ p\geq 1.
%\end{array}
%\right.
\end{equation}
These functions take values in $\mathbb{R}_+ \cup \{ + \infty \}$. We can now state our first drift condition.
%that for all $\mf{t}$ in a neighborood of $0$ and all $\mf{s}>0$ then $b(\mf{t})$ and $K(\mf{t},[\mf{s},+\infty))>,$ are positive. 

\begin{prop}[(Continuous) Drift conditions for extinction/survival]
\label{th:lyap}
Let us suppose Assumption~\ref{ass:b,K} holds.
\begin{enumerate}
\item If there exists a positive $V\in \widehat{D}(L)$ such that $LV\leq 0$ outside a compact set and $$\lim\limits_{x\to\infty} V(x)=+\infty,$$ then $\tau_0<+\infty$ almost surely, which means that $\mathbb{T}$ is not supercritical.
\item If Assumption~\ref{hyp:ageBarrier} holds and there exists a positive $V\in \widehat{D}(L)$ and a compact $K\supset\{ 0\}$ such that 
\[
\left\{
\begin{array}{l}
LV(x)\leq 0, \quad \forall x\not\in K,\\
V(x)<\inf_{y\in K} V(y),\quad \forall x\not\in K,
\end{array}
\right.
\]
 then $\mathbb{T}$ is supercritical.
\end{enumerate}
\end{prop}

\begin{proof}
For the first point, $V$ satisfies $(CD1)$  of \cite{MT93III} and then by the result of \cite[Section 3]{MT93III}, the process does not drift to infinity, and then goes to $0$ according to Theorem \ref{th:dichotomy}. For the second point, it is classic that if $V$ is a map such that $L V\leq 0$ then
\begin{equation}
\label{eq:boundproba}
\mathbb{P}_x (\tau_0<+ \infty) \leq \frac{V(x)}{V(0)}, \quad \forall x\in\mathbb{R}_{+.}
\end{equation}
Indeed, as $(V(X_{t\wedge \tau_0})_{t\geq 0}$ is a super-martingale (from Theorem \ref{th:genext}), it is a consequence of the stopping time theorem and Fatou Lemma. Then for $x\not\in K$, $\mathbb{P}_x (\tau_0<+ \infty)<1$. Using Assumption~\ref{hyp:ageBarrier} ends the proof.
\end{proof}

\begin{cor}[Sufficient asymptotic conditions for extinction/survival]
\label{cor:lyap}
Suppose Assumption~\ref{ass:b,K} holds.
\begin{enumerate}
\item If $\limsup_{x\to \infty} b(x)m(x) <1$, then $\mathbb{T}$ is not supercritical.
\item If Assumption~\ref{hyp:ageBarrier} and $\liminf_{x\to \infty}  b(x)m(x) >1$ and $\sup_{x\geq 1} \frac{b(x)m_2(x)}{x}<+\infty$, then $\mathbb{T}$ is supercritical.
\end{enumerate}
\end{cor}
\begin{proof}
For the first point use $V:x\mapsto x$ in Proposition~\ref{th:lyap}. For the second point, set $V:x\mapsto M x^{-\alpha}\mathbf{1}_{x\geq M} + \mathbf{1}_{x< M}$, for some $M,\alpha>0$ fixed hereafter. A rapid calculation using Jensen inequality shows that, for $x>M$,
\begin{align*}
L V(x) 
\leq \frac{\alpha}{x^{\alpha+1}} \left[ \left(1- b(x)m(x) \right) + \frac{1}{\alpha +1} b(x) \frac{1}{x} m_2(x) \right].
\end{align*}
Then for $M$ and $\alpha$ large enough, we have $L V \leq 0$.
\end{proof}

\begin{rem}[Assumption]
The assumption $\sup_{x\geq 1} \frac{b(x)m_2(x)}{x}<+\infty$ is a technical but no confining assumption. Moreover, it is not clear if it is only a technical one (that can be removed) or if it has a real sense. Indeed, let us recall that fluctuation of the environment has an effect of the probability of extinction in classical Galton-Watson chain (see \cite[Section 2.9.2 p.49]{HJV} for instance).

Anyway, it can be weakened in such a way: if $V$ is decreasing with increasing derivative function $V'$ then it is enough that $\sup_{x\in K^c} b(x) m_2(x) V''(x)<+\infty$ for some compact set $K$. 
\end{rem}

\begin{rem}[Direct coupling approach]
Corollary~\ref{cor:lyap} is intuitive and it seems possible to prove it directly on the tree. If $K(\cdot,dy)=K(dy)$ is constant, $b$ varies and $m \sup_{x} b(x)= m\mathbf{b} <1$ then it is easy to couple the IST with a time-homogeneous splitting tree (with parameters $K$, $\mathbf{b}$) and then deduce the extinction from the homogeneous case. One can also extend this argument by coupling after a certain moment to have the condition $m \limsup_{x\to+\infty} b(x) <1$ and also prove survival condition with the same arguments. However when $K$ varies this argument totally fails and it is not easy (at least for us) to see how one can prove this result directly on the tree by simple argument.
\end{rem}

Piecewise deterministic Markov process are almost discrete objects (in contrast with diffusion processes) because there is no randomness between jumps. It is then inviting to consider discrete criteria on the post-jump Markov chain. That has been done in \cite{CD08} for instance, and our case this reads:

\begin{prop}[(Discrete) Drift conditions for extinction/survival]
\label{th:lyap2}%\cmb{relire le bouquin de MT pour revoir ces cdtions voir aussi Dufour et le citer}
Suppose Assumption~\ref{ass:b,K} holds. Let $P$ be the transition kernel defined, for every positive function $f$, by $Pf(0)=f(0)$ and  for all $x>0$ by 
$$
Pf(x)= \int_0^{x }\int_{\mathbb{R}_{+}}  f(y+z)\,  K(y,dz) b(y) e^{-\int_y^{x} b(s)ds}\, dy + f(0) e^{-\int_0^{x} b(s)\, ds}.
$$
\begin{enumerate}
\item If there exists a positive function $V$ such that $PV -V\leq 0$ outside a compact set and $$\lim\limits_{x\to\infty} V(x)=+\infty,$$ then $\mathbb{T}$ is not supercritical.
\item If Assumption~\ref{hyp:ageBarrier} holds and there exists a positive function $V$ and a compact set $K$ such that
\[
\left\{
\begin{array}{l}
PV(x)\leq V(x), \quad \forall x\not\in K,\\
V(x)<\inf_{y\in K} V(y),\quad \forall x\not\in K,
\end{array}
\right.
\]
then the tree is supercritical.
\end{enumerate}
\end{prop}

\begin{proof}
Let us consider $(\widehat{X}_n)_{n\geq0}$ be the (post-jump) embedded (or skeleton) chain associated to $(X_t)_{t\geq0}$; namely it is the Markov chaine defined by $\widehat{X}_n=X_{T_n}$, where $T_n$ is the $n$th jump time of $(X_t)_{t\geq0}$. Its transition kernel is $P$ and similarly to Proposition~\ref{th:lyap}, the statement is a consequence of the classical result \cite[Theorem 8.4.3 p. 191]{MT09} and \cite[Proposition 8.4.1 p.189]{MT09} (in discrete times now).
\end{proof}

\begin{cor}[Sufficient integral condition for extinction]
\label{cor:lyapd}
If Assumption~\ref{ass:b,K} holds and
%$$
%\limsup_{\mf{t}\to \infty} \int_0^{\mf{t}}  \left((\mf{s} + m(\mf{s})) b(\mf{s}) e^{-\int_\mf{s}^{\mf{t}} b(\mf{u})d\mf{u}} - 1 \right) d\mf{s} <0
%$$ 
%or equivalently
$$
\limsup_{x\to \infty} \int_0^{x}  \left(( m(s)b(s)-1)  e^{-\int_s^{x} b(u)du} \right) ds <0,
$$ 
then $\mathbb{T}$ is not supercritical.
\end{cor}
\begin{proof}
Again function $V:x \mapsto x$ satisfies
\begin{align*}
Pf(x)
%&= \int_0^{\mf{t}}  (\mf{t}-\mf{s} + m(\mf{t}-\mf{s})) b(\mf{t}-\mf{s}) e^{-\int_0^\mf{s} b(\mf{t}-\mf{u})d\mf{u}} d\mf{s}
= \int_0^{x}  (s + m(s)) b(s) e^{-\int_s^x b(u)\, du}\, ds.
\end{align*}
An integration by parts ends the proof.
\end{proof}

In case of finite tree, it is possible to give some bounds on the tail of the length of the tree $\mathcal{L}(\mathbb{T})$ using again drift conditions. To our knowledge such results were never established in the case of non-Markovian tree (even in time-homogeneous setting). Also, with the help of Theorem~\ref{prop:htrans}, this leads some bound for some. conditioned trees.

\begin{prop}[Tail estimate of the tree length]
\label{prop:tail}

Suppose Assumption~\ref{ass:b,K} holds and \\ $\limsup_{x\to \infty} b(x)m(x) <1.$
\begin{enumerate}
\item If there exists $a>0$, such that for any large enough $x$, $\int_{\mathbb{R}_{+}} e^{a y} K(x,dy)<+\infty$ then
\begin{equation}
\label{eq:expoTail}
\mathbb{P}_{x}(\mathcal{L}(\mathbb{T}) \geq t) \leq C (1+ e^{\theta x}) e^{-\lambda t},\quad \forall t\geq 0,
\end{equation}
and some constants $C,\lambda, \theta>0$.
\item If there exists $p \in \mathbb{N}, p\geq 2$, such that for any  $x$ large enough, $m_p(x) <+\infty$ and $$\lim_{x\to \infty} x^{k-p+1} m_{p-k}(x)=0,$$ then
\begin{equation}
\label{eq:geoTail}
\mathbb{P}_{x}(\mathcal{L}(\mathbb{T}) \geq t) \leq C x^p t^{-1/p+1/p^2},\quad \forall t\geq 0, %(t^{-1/p} x^p + e^{\theta x}t^{-1/p+1/p^2}),\quad \forall t\geq 0,
\end{equation}
and some constant $C>0$.
\end{enumerate}
\end{prop}

\begin{proof}
First, from Theorem \ref{thm:genContour}, we have
$$
\mathbb{P}(\mathcal{L}(\mathbb{T}) \geq 2t) = \mathbb{P}(\tau_0>t)=\| \mathcal{L}(X_t) - \delta_0 \|_{\textrm{TV}},
$$
where  $\| \cdot \|_{\textrm{TV}}$ is the classical total variation distance. 
Hence, it is enough to work with classical results on the convergence of Markov processes to equilibrium. In particular \eqref{eq:expoTail} is a consequence of \cite[Theorem 6.1]{MT93III} and \eqref{eq:geoTail} is a consequence of \cite[Theorem 3.10, (3.6)]{DFG}. For the first point, let us check \cite[(CD3)]{MT93III} with $V_\theta(x) :\mf{t} \mapsto e^{\theta x}$. In the first hand, we have  
$$
L V_\theta(x) = V(x) \left(-\theta  +  b(x) \int_{\mathbb{R}_{+}} (e^{\theta y} -1) K(x,dy)\right) = \lambda_\theta V(x),\quad \forall x\geq 0.
$$
On the second hand $\lambda_0=0$, and $\partial_\theta \lambda_\theta(x) =_{\scriptscriptstyle{\vert \theta=0}} b(x)m(x) -1 $ which is negative for large $x$. Then there exists a small $\theta$ such that  \cite[(CD3)]{MT93III} and \cite[Theorem 6.1]{MT93III} hold. For the second point now, the map $V_p:x\mapsto x^p$ satisfies
$$
\mathcal{L} V_p (x)%= p\mf{t}^{p-1} + b(\mf{t}) \left( \int (\mf{t}+\mf{u})^p K(\mf{t},d\mf{u})-\mf{t}^p \right)
= (b(x)m(x)-1) p V_p^{1-1/p}(x) + \sum_{k=0}^{p-2}  \dbinom{p}{k} x^{k-p+1} m_{p-k}(x)
$$
and then the drift condition of \cite[Theorem 3.11, (3.10)]{DFG} (see also \cite[Theorem 4.1 point 3]{H16} or \cite{DFMS}) holds. This ends the proof.
\end{proof}

\vspace{1cm}

\subsection{Applications to particular cases}
\label{ssec:example}

\subsubsection{Periodic environement}
We will assume that $b=\beta$ is constant and $m$ is periodic (with period $\mf{T}$). Set
$$
\psi:\mf{s} \mapsto \beta m(\mf{s}) -1, \qquad \Psi:\mf{t} \mapsto e^{-\mf{t}} \int_0^\mf{t} \psi(\mf{s}) e^{\beta \mf{s}} d\mf{s}.
$$
From Corollary~\ref{cor:lyapd}, the process goes to extinction if $\limsup_{\mf{t}\to\infty} \Psi(t)<0$. Using the periodicity assumption, we have, for every $n\geq 0$,
\begin{align*}
\Psi(n\mf{T}) 
&= e^{-\beta \mf{t}-\beta (n-1)\mf{T}} \Psi(\mf{T}) +\Psi(\mf{t})= e^{-\beta \mf{t}} \Psi(\mf{T}) \sum_{k=0}^{n-1}e^{-\beta k \mf{T}} + \Psi(\mf{t}) \\
&=\Psi(\mf{T}) \frac{e^{-\beta \mf{t}} (1-e^{-\beta n \mf{T}}) }{(1-e^{-\beta k \mf{T}})} + \Psi(\mf{t}).
\end{align*}
In particular, if we set
$$
\Phi:\mf{t} \mapsto \Psi(\mf{T}) \frac{e^{-\beta \mf{t}} }{(1-e^{-\beta k \mf{T}})} + \Psi(\mf{t})=e^{-\beta \mf{t}}\left( \frac{e^{-\beta \mf{T}} }{(1-e^{-\beta k \mf{T}})} \int_0^\mf{T} \psi(\mf{s}) e^{\beta \mf{s}} d\mf{s}+ \int_0^\mf{t} \psi(\mf{s}) e^{\beta \mf{s}} d\mf{s} \right),
$$
then $\Phi$ is periodic and $\lim_{\mf{t}\to \infty} |\Phi(\mf{t}) -\Psi(\mf{t})|=0$. To verify Corollary~\ref{cor:lyapd} (1), it is then enough to have $\sup_{\mf{t}\in [0, \mf{T}]} \Phi(\mf{t})<0.$ As an instance, when $\psi:\mf{t}\mapsto \cos(\mf{t})+c$ ($c$ fixed hereafter), we have
$$
\forall \mf{t} \in [0,2\pi], \ \Phi(\mf{t})=\frac{1}{1+\beta^2}\left[-\beta e^{-\beta \mf{t}}(1+e^{-2\pi\beta})+\beta \cos(\mf{t}) +\sin(\mf{t}) \right]
%= \frac{-e^{-\mf{t}}- e^{-\mf{t}-2\pi} + \cos(\mf{t}) + \sin(\mf{t}) }{2} +c
$$
For $\beta=1$, $\sup_{\mf{t}\in [0, \mf{T}]} \Phi(\mf{t}) \approx  0.5072555 +c$. Then the process goes to extinction for $c< 0.5073$. Note that for such $c$, we can have $\limsup_{\mf{t}\to \infty} \psi(\mf{t})=1+c>0$ (this is then a better result than those of Corollary~\ref{cor:lyap} (1))

\subsubsection{Time-homogeneous splitting tree with heavy-tail life distribution}

Let us assume that $b$ is constant and $K(dy)$ is given by a Pareto distribution; that is
$$
K(dy)=\frac{k}{x^{k+1}},
$$
for some $k\in \mathbb{N}$. Assume $mb=kb/(k-1)<1$. We have $m_p<+\infty$ if and only if $k \geq 1+p$. Under this condition, assumption of Proposition~\ref{prop:tail} 2. holds, and we have the bound \eqref{eq:geoTail} on the tail. On the other side, we have (with the notation of Section~\ref{sec:model})
\begin{align*}
\mathbb{P}_x(\mathcal{L}(\mathbb{T}) \geq t)
&\geq \mathbb{P}_x(P_{\mathcal{U}}(\mathbb{T}) \neq \{\emptyset\}, B_1 \geq t)
= e^{-b x} t^{-k}.
\end{align*}
Even if we do not find the same bound as in \eqref{eq:geoTail} , this show that the tail of the tree length is certainly heavy and the bound \eqref{eq:geoTail} is not too rough.

\section{Scaling limits of non-homogeneous splitting trees}
\label{sec:scaling}
Let us finally end the paper with the present section by some scaling limits of ISTs. The first subsection is concerned by the overall shape of asymptotically critical IST. In the second part, we consider the asymptotic behavior of ISTs as the birth-rate goes to $\infty$ whereas the birth-kernel $K(x,dy)$ tends to have zero means. In such situation, we consider two sequences $(b_{n})_{n\geq 1}$ and $(K_{n})_{n\geq 1}$ converging point-wise receptively to $\infty$ and $\delta_{0}$. %We show that one cannot obtain scaling limit for such trees without endowing it in the larger space of TOM trees. This is quite a frustrating result since, intuitively speaking, the limiting tree (from a topological point of view) should not depend on a measured structure.
 
\subsection{Scaling limit of asymptotically critical IST}
In this section we look up to the global shape of asymptotically critical IST. More precisely, we suppose that $b$ and $K$ satisfy

\[
\lim\limits_{x\to\infty}b(x)\int_{\mathbb{R}_{+}} y\ K(x,dy)=1.
\] 
To this end, we consider the time contraction $\mathbb{T}_{n}$ of $\mathbb{T}$ on the scale $(a_{n},c_{n})_{n\geq 1}$ by
\[
 (\delta,t)\in\mathbb{T}_{n} \Leftrightarrow (\delta,c_{n}t)\in\mathbb{T}%(c_{n}x),
\]
where $(a_{n},c_{n})\to\infty$  as $n$ goes to infinity. We also look at the law of $\mathbb{T}_n$ under $\mathbb{P}_{c_n x}$ for fixed $x>0$.
Our aim is to show that the sequence a tree converges to some tree in the sense of the Gromov-Hausdorff-Prokhorov topology using the contour processes.
Since the triplet $(\mathbb{T},d,\lambda)$ is a TOM tree, we can define a sequence of TOM trees by setting
\begin{equation}
\label{eq:scaledl}
\left\{
\begin{array}{l}
d_{n}((\sigma, s),(\delta, t))=d((\sigma, s),(\sigma, t)),\quad \forall(\sigma,s),(\delta, t)\in\mathbb{T}_{n}\\
\lambda_{n}(A)=\frac{a_{n}}{c_{n}}\lambda(A),
\end{array}
\right.
\end{equation}
From this construction, we have that the contour process $(\mathcal{C}(\mathbb{T}_{n}),\ s\geq 0)$ of $\mathbb{T}_{n}$ has generator given by
\[
L_{n}f(x)=-\frac{a_{n}}{c_{n}}f'(x)+a_{n}b(c_{n}x)\int_{\mathbb{R}_{+}}f\left(x+\frac{y}{c_{n}}\right)-f(x)\ K(c_{n}x,dy).
\]
In particular, if we assume that $f$ has third derivative we have, using Taylor expansion,
\begin{multline}
\label{eq:devLim}
L_{n}f(x)=\frac{a_{n}}{c_{n}}\left(b(c_{n}x)m(c_{n}x) -1\right)f'(x)+\frac{1}{2}f''(x)\frac{a_{n}b(c_{n}x)m_{2}(c_{n}x)}{c_{n}^{2}}\\+ a_{n}b(c_{n}x)\int_{\mathbb{R}_{+}}\int_{x}^{x+y/c_{n}}\frac{f^{(3)}(s)}{2}\left(s-x-\frac{y}{c_{n}}\right)^{2}\ ds\ K(c_{n}x,dy).
\end{multline}
Hence, to obtain a non-degenerate limit in $n$, one should require that
\[
\left\{
\begin{array}{lll}
\lim\limits_{n\to\infty}\frac{a_{n}}{c_{n}}\left(b(c_{n}x)m(c_{n}x) -1\right)&= &g_{d}(x),\\
\lim\limits_{n\to\infty}\frac{a_{n}}{c_{n}^{2}}b(c_{n}x)m_{2}(c_{n}x)&=&g_{v}(x),
\end{array}
\right.
\]
for some functions $g_{d}$ and $g_{v}$. In particular, this implies that the functions $x\mapsto\left(b(x)m(x) -1\right)$ and $x\mapsto b(x)m_{2}(x)$ are regularly varying. As a consequence, there exists two slowly varying functions $S_{d}$ and $S_{v}$, and two real numbers $\beta$ and $\gamma$, such that
\[
\left\{
\begin{array}{l}
\left(b(x)m(x) -1\right)=x^{\beta}S_{d}(x),\\
b(x)m_{2}(x)=x^{\gamma}S_{v}(x).
\end{array}
\right.
\]
In addition, the only case where the limit can hold is $\beta=\gamma-1$. Now, in view of \eqref{eq:devLim}, one should expect that the limiting generator has form given by
\[
x^{\beta}f'(x)+x^{\beta+1}f''(x).
\]

 This leads to the following assumptions on $b$ and $K$.

\begin{assu}[Asymptotically critical IST]
	\label{hyp:regul}
	~	
\begin{enumerate}
\item$\left(b(x)m(x) -1\right)=\frac{c}{x}S_{d}(x)$ for some constant $c \in \mathbb{R}$ and such that $S_{d}$ is bounded and satisfies $\lim\limits_{x\to\infty}S_{d}(x)=1$.
\item$b(x)m_{2}(x)=S_{v}(x)$ such that $S_{v}$ is bounded and satisfies $\lim\limits_{x\to\infty}S_{v}(x)=1$.
\item $x\mapsto b(x)m_{3}(x)$ is bounded.
\end{enumerate}
\end{assu}

\begin{rem}
	\begin{itemize}
\item There are many examples of parameters which satisfy these hypotheses. Among the simpler ones, one could, for instance, think of
\[
\left\{
\begin{array}{l}
b(x)=1+\frac{c}{1+x}, \text{for some constant }c,\\
K(x,dy)=\delta_{1}(dy).
\end{array}
\right.
\]
See also Figure~\ref{fig:Bessel-intro}.
\item One may think that assuming $\beta=-1$ is a bit restrictive but the above proof holds for any choice of $\beta$. Moreover, it does not change limit up to a change of measure.

\end{itemize}
\end{rem}

In addition, in order to lighten notations, we also assume that $c^2_{n}=a_{n}=n$. The proof relies on the two following lemma. The idea is to show the convergence of the sequence $(L_{n})_{n\geq 1}$ of generators using a wise choice of core for the expected limit of this sequence. The good choice appears in the following lemma where we show that this set is dense.

In the Lemma that follows and others results of the present section, the topology is the uniform topology.

\begin{lem}
\label{lem:goodSpace}
The subspace
\[
\mathcal{D}=\left\{f\in C^{\infty}_{c}(\mathbb{R}_{+})\mid f^{(n)}(0)=0, \quad \forall n\geq 1 \right\},
\]
is dense in $C_{0}(\mathbb{R}_{+})$.
\end{lem}
%\begin{proof}
%Let $f$ be a function of $C_{0}(\mathbb{R}_{+})$ and $(\varphi_{n})_{n\geq 1}$ be a sequence of $C^{\infty}_{c}(\mathbb{R}_{+})$ converging to $f$. Now, set, for any integer $n$ and any positive real number $x$,
%\[
%\tilde{\varphi}_{n}(x)=
%\left\{
%\begin{array}{ll}
%\varphi(x), \quad &\text{if }x\in [1/n,\infty),\\
%\varphi(1/N), \quad &\text{if } x\in[0,1/n].
%\end{array}
%\right.
%\]
%Clearly, the sequence $(\tilde{\varphi}_{n})_{n\geq 1}$ belongs to $\mathcal{D}$ \textcolor{red}{lol pas du tout}. In addition,
%\[
%\|\tilde{\varphi}_{n}-f\|_{\infty}\leq \sup_{x\in [0,1/n]} \left|\varphi_{n}(1/n)-\varphi_{n}(x)\right|+\|\varphi_{n}(x)-f(x)\|_{\infty}.
%\]
%Now, since $(\varphi_{n})_{n\geq 1}$ is a converging sequence of $C_{0}(\mathbb{R}_{+})$, it forms a  locally uniformly equicontinuous family of functions. Hence,
%\[
%\sup_{x\in [0,1/n]} \left|\varphi_{n}(1/n)-\varphi_{n}(x)\right|
%\]
%converges to $0$ as $n$ goes to infinity. This ends the proof.
%\end{proof}
\begin{proof}
Let us fix $f\in C_{0}(\mathbb{R}_{+})$. Let $\psi \in C_c^\infty(\mathbb{R}_{+})$ such that 
$$
\psi(0)=1, \quad \forall n\geq 0, \ \psi^{(n)}(0)=0,\footnote{It can be construct as $\psi=B \circ A$, with $A:x\mapsto (1-x e^{-1/x^2}) \mathbf{1}_{\{x>0\}}$ and $B:x\mapsto \exp\left(1-1/(1-x^2)\right)\mathbf{1}_{\{x\leq 1\}}$.}
$$ and set $g:x\mapsto f(x)-f(0)\psi(x)$ restricted on $(0,+\infty)$. Function $g$ belongs to 
$$
C_{0}((0,+\infty))= \{h\in C((0,+\infty)) \ | \ \lim_{x\to 0} h(x) = \lim_{x\to + \infty} h(x) =0\}.
$$
As $C^{\infty}_{c}((0,+\infty))$ is dense in $C_{0}((0,+\infty))$ then there exists a sequence of functions $(\varphi_n)_{n\geq 0}$ of $C^{\infty}_{c}((0,+\infty))$ converging to $g$ and then (the continuous extension) $(\varphi_n + f(0)\psi)_{n\geq 0}$ converges to $f$. It remains to show that this sequence belongs to $\mathcal{D}$ but as functions $\varphi_n$ are null in a neighborhood of $0$, this is direct.

\end{proof}
The next lemma shows that $\mathcal{D}$ is indeed a core for the expected generator.
\begin{lem}
Let $(A^{c},D(A^{c}))$ be the generator of a Bessel process with dimension $2c+1$ absorbed at $0$. In particular
\[
A^{c}f(x)=\frac{c}{x}f'(x)+\frac{1}{2}f''(x),
\]
for all twice differentiable functions $f$ in $C_{0}(\mathbb{R}_{+})$. Then, the space $\mathcal{D}$ defined in Lemma \ref{lem:goodSpace} is a core for $(A^{c},D(A^{c}))$.
\end{lem}

Note that in this Lemma (and all the paper), Bessel process are generalized Bessel process with possibly negative dimension; see \cite{GJY03}.

\begin{proof}
First of all, it is well-known that the generator of the Bessel process in any dimension contains the space of twice differentiable functions on $(0,\infty)$ (see \cite{Kent,GJY03}). As a consequence, $D(A^{c})$ contains $\mathcal{D}$. Now, according to Lemma \ref{lem:goodSpace} and Proposition 17.9 of \cite{Kal}, we just needs to show that $\mathcal{D}$ is invariant under the action of $A^{c}$, which is clearly the case.  Consequently, $\mathcal{D}$ is indeed a core for $(A^{c},D(A^{c}))$.
\end{proof}
We can now prove the main result.
\begin{thm}
\label{th:scaling}
Under Assumptions~\ref{ass:b,K} and \ref{hyp:regul}, the sequence of contour processes $\mathcal{C}(\mathbb{T}_{n})_{n\geq 1}$ associated with the sequence of TOM trees $(\mathbb{T}_{n})_{n\geq 1}$ converges weakly to a (possibly absorbed) Bessel processes with dimension $(c-1)/2$.
\end{thm}
\begin{proof}
The proof relies on Theorem 2.2.5 and Theorem 1.6.1 of \cite{EK86}, showing that the convergence of the sequence of generators implies the weak convergence of the associated processes. To this end, let $f$ be an element of $\mathcal{D}$.
Now set for every $n\geq 1$,
\[
f_{n}(x)=f(x)-\int_{\mathbb{R}_{+}}\left(f\left(\frac{y}{\sqrt{n}}\right)-f(0)\right) \ K(0,dy).
\]
Clearly, for any integer $n\geq 1$, $f_{n}$ belongs to the domain $D(L_n)$ of the generator $L_{n}$. Moreover, it is easily seen that $f_{n}$ converges to $f$ in $C_{0}(\mathbb{R}_{+})$. Now, according to \eqref{eq:devLim}, we have, for $x>0$,
\begin{align*}
\left|L_{n}f_{n}(x)-A^{c}f(x)\right|\leq& |L_{n}f(0)|+\left|L_{n}f(x)-A^{c}f(x)\right|\\\leq&\left|\sqrt{n}\left(b(\sqrt{n}x)m(\sqrt{n}x) -1\right)-\frac{c}{x}\right |\left|f'(x)\right|+\frac{1}{2}|f''(x)|\left|b(\sqrt{n}x)m_{2}(\sqrt{n}x)-1\right|\\&+\left| nb(\sqrt{n}x)\int_{\mathbb{R}_{+}}\int_{x}^{x+y/\sqrt{n}}\frac{f^{(3)}(s)}{2}\left(s-x-\frac{y}{\sqrt{n}}\right)^{2}\ ds\ K(\sqrt{n}x,dy)\right|+|L_{n}f(0)|.
\end{align*}
Now, using Assumptions \ref{hyp:regul}, we get
\begin{multline}
\label{eq:supToBound}
\left|L_{n}f_{n}(x)-A^{q}f(x)\right|\leq\left|S_{d}(\sqrt{n}x)-1\right |\frac{c\left|f'(x)\right|}{x}+\frac{1}{2}\left|f''(x)\right|\left|S_{v}(\sqrt{n}x)-1\right|\\+\left| nb(\sqrt{n}x)\int_{\mathbb{R}_{+}}\int_{x}^{x+y/\sqrt{n}}\frac{f^{(3)}(s)}{2}\left(s-x-\frac{y}{\sqrt{n}}\right)^{2}\ ds\ K(\sqrt{n}x,dy)\right|+|L_{n}f(0)|.
\end{multline}
Hence, we have, for all $x>1$,
\begin{align*}
\left|L_{n}f_{n}(x)-A^{q}f(x)\right| \leq& |L_{n}f(0)|+\mathcal{C}\Bigg(c\left|S_{d}(nx)-1\right |+\frac{1}{2}\left|S_{v}(nx)-1\right|+ \frac{nb(\sqrt{n}x)m_{3}(\sqrt{n}x)}{2n^{3/2}}\Bigg),
\end{align*}
where \[
\mathcal{C}=\max\left(\sup_{x\geq1}\left|\frac{f'(x)}{x}\right|,\|f''\|_{\infty},\|f^{(3)}\|_{\infty} \right).
\]
Now, using that $\left|S_{d}(nx)-1\right |$ and $\left|S_{v}(nx)-1\right|$ converges to $0$, as $n$ goes to infinity, uniformly on any set of the form $[a,\infty)$ (for $a>0$). We get the uniform convergence on $[1,\infty)$.
In addition, $\left|L_{n}f_{n}(0)-A^{q}f(0)\right|=0$, so we only have to study the supremum over $(0,1]$. To this end we use Equation \eqref{eq:supToBound}. Let us assume that
\[
\limsup_{n\to \infty}\sup_{x\in(0,1]}\left|S_{d}(nx)-1\right|\frac{c\left|f'(x)\right|}{x}=\alpha.
\]
This implies that there exists two sequences $(k_{n})_{n\geq 1}\subset \mathbb{N}$ and $(x_{n})_{n\geq 1}\subset(0,1]$ such that
\[
\left|S_{d}(k_{n}x_{n})-1\right|\frac{c\left|f'(x_{n})\right|}{x_{n}}\geq\alpha, \quad \forall n\geq 1,
\]
where $k_{n}\geq n$, for all positive integer $n$.
Now, if $$\liminf_{n\to\infty}x_{n}>0,$$ using the uniform convergence  property of $\left|S_{d}(k_{n}x)-1\right|$ and the boundedness of $\frac{c\left|f'(x)\right|}{x}$, $\alpha$ has to be equal to $0$. On the other hand, if $$\liminf_{n\to\infty}x_{n}=0,$$
extracting a subsequence if needed, we have
\[
\alpha\leq\limsup_{n\to \infty}\left|S_{d}(k_{n}x_{n})-1\right|\frac{c\left|f'(x_{n})\right|}{x_{n}}\leq M \limsup_{n\to\infty} \frac{\left|f'(x_{n})\right|}{x_{n}}=0,
\]
where $M$ is an upper bound for $\left|S_{d}(x)-1\right|$. Let us also point out that $\frac{\left|f'(x_{n})\right|}{x_{n}}$ converges to $0$ because $f'(x)$ goes to $0$ as $x$ goes to $0$ faster than any polynomial.

The study of the other terms of \eqref{eq:supToBound} follows the same lines. Consequently, we obtain the convergence of the sequence of generators and thus the result.
\end{proof}
This last result implies, in virtue of the result of \cite{lambertTOM}, the convergence of the sequence of trees.
\begin{cor}
Under Assumption~\ref{ass:b,K} and \ref{hyp:regul}, the sequence of TOM trees $(\mathbb{T}_{n})_{n\geq 1}$ converges weakly, in the Gromov-Hausdorf-Prokorhov topology, to a TOM tree whose contour process is a Bessel process with dimension $(c-1)/2$.
\end{cor} 

\begin{proof}
This is a direct consequence of \cite[Appendix A.4]{lambertTOM} or \cite[Proposition 2.3.13]{L17}
\end{proof}

\begin{rem}[Bessel tree]
The construction of $(\mathbb{T},d,\lambda)$ of the Bessel tree from its contour process is described in \cite[Section 1]{lambertTOM}; see also \cite{duqu,LeGallTree}.
\end{rem}

\begin{rem}[On the proof]
Instead of using a functional approach to prove the scaling limit, one can think of using a Martingale approach as in \cite[Chapter 7, Theorem 4.1 p.354]{EK86}. Indeed, Assumptions are simpler in the sense that it is not required to know the explicit domains of the generator, and we only need some martingales. However, to prove the assumptions of \cite[Chapter 7, Theorem 4.1 p.354]{EK86} needs to control uniformly the process in $(0,+\infty)$, that is how it is far from $0$. We do write this here but with this approach we only arrive to prove Theorem~\ref{th:scaling} in the case $c\geq -1/2$. 
\end{rem}

\begin{rem}[Applications]
It was quite surprising that Bessel process appears naturally in this context; even if it is related to others random trees in others contexts \cite{LG15}.
There exists a lot of results on Bessel processes; see for instance \cite{RY, GJY03} and it then leads to new properties. For instance, \cite[Chapter XI]{RY} gives that, under Assumption~\ref{hyp:regul} the tree is almost surely finite when $c\leq 1/2$ and it is supercritical for $c>1/2$. In a certain sense, this refines Corollary~\ref{cor:lyap}. Among many other results, Theorem 3 and Theorem 4 of \cite{PY99} permit for instance to have some asymptotics of $\mathcal{H}(\mathbb{T})$ when Assumption~\ref{hyp:regul} and $c\to 0$.
\end{rem}

\subsection{No scaling limits for naive JCCP}
In this section, we show that the convergence of a sequence of chronological tree satisfying $b$ goes to infinity as $t$ increase and that of $K(t,dy)$ converges to $\delta_{0}(dy)$ cannot be done using JCCP without considering a renormalization of the metric and the measure. This is due to the following result.
\begin{thm}
\label{prop:noscaling}
Under Assumption~\ref{ass:b,K}, there is no diffusive limits to any sequence of processes in $\mathbb{D}([0,1])$ with generators of the form
$$L_{n}f(x)=-f(x)+b_{n}(x)\int_{\mathbb{R}_{+}}(f(x+y)-f(x))\ K_{n}(x,dy).$$
\end{thm}
The proof of this result relies on the two following Lemmas.
\begin{lem}Define for any $f\in\mathbb{D}([0,1])$, the negative variations of $f$ by
$$
V_-(f)=\sup_\sigma \sum_{i=1}^{|\sigma|}|f(\sigma_i)-f(\sigma_{i-1})|\mathbf{1}_{\{f(\sigma_i)<f(\sigma_{i-1})\}},
$$
where the supremum is taken on the set of all partition $\sigma$ of $[0,1]$ with $\sigma_{i}<\sigma_{i+1}$ for all integer $i\leq |\sigma|$, where  $|\sigma|$ is the cardinality of the partition $\sigma$.  Then, the set
$$
V_1=\{f \in D([0,1]) \mid V_{-}(f)\leq 1 \}
$$
is closed for the Skorokhod topology.
\end{lem}
\begin{proof}
Let $f_{n}$ be a sequence of functions of $\mathbb{D}([0,1])$ converging to $f$ in the Skorokhod space $\mathbb{D}([0,1])$. Hence, there exists a sequence of strictly increasing continuous function $(\lambda_{n})$ such that
\[
\left\{
\begin{array}{l}
\lim\limits_{n\to\infty}\|\lambda_{n}-I\|_{\infty}=0,\\
\lim\limits_{n\to\infty}\|f_{n}\circ \lambda_{n}-f\|_{\infty}=0,
\end{array}
\right.
\]
\iffalse
where $\|.\|_{\infty}$ stands for the uniform norm over $f\in\mathbb{D}([0,1])$ and $I$ is the identity function on $[0,1]$.\fi

where $I$ is the identity function on $[0,1]$. Now, let $\sigma$ be some partition of $[0,1]$, then 
\begin{equation}
\label{eq:1}
\sum_{i=1}^{|\sigma|}|f(\sigma_i)-f(\sigma_{i-1})|\mathbf{1}_{\{f(\sigma_i)<f(\sigma_{i-1})\}}=\lim\limits_{n\to \infty}\sum_{i=1}^{|\sigma|}|f_{n}(\lambda_{n}(\sigma_i))-f_{n}(\lambda_{n}(\sigma_{i-1}))|\mathbf{1}_{\{f_{n}(\lambda_{n}(\sigma_i))<f(\lambda_{n}(\sigma_{i-1}))\}}.
\end{equation}
On the other hand, since $\lambda_{n}$ is strictly increasing, there exists some partition $\tilde{\sigma}$ of $[0,1]$ such that $\tilde{\sigma}_{i}=\lambda_{n}(\sigma_{i})$, for all $i$. In particular,
\begin{equation}
\label{eq:2}
\sum_{i=1}^{|\sigma|}|f_{n}(\lambda(\sigma_i))-f_{n}(\lambda_{n}(\sigma_{i-1}))|\mathbf{1}_{\{f_{n}(\lambda_{n}(\sigma_i))<f(\lambda_{n}(\sigma_{i-1}))\}}\leq \sup_\sigma \sum_{i=1}^{|\sigma|}|f_n(\sigma_i)-f_n(\sigma_{i-1})|\mathbf{1}_{\{f_n(\sigma_i)<f_n(\sigma_{i-1})\}}\leq 1.
\end{equation}
Finally, equations \eqref{eq:1} and \eqref{eq:2} ends to proof.
\iffalse
leads to
\[
V_{-}(f)=\sup_{\sigma}\sum_{i=1}^{|\sigma|}|f(\sigma_i)-f(\sigma_{i-1})|\mathbf{1}_{f(\sigma_i)<f(\sigma_{i-1})}\leq \liminf_{n\to \infty} \sup_\sigma \sum_{i=1}^{|\sigma|}|f_{n}(\sigma_i)-f_{n}(\sigma_{i-1})|\mathbf{1}_{f_{n}(\sigma_i)<f_{n}(\sigma_{i-1})}\leq 1,
\]
which \fi
\end{proof}
\begin{lem}
Denote
\[
V(f,x)=\sup_\sigma \sum_{i=1}^{|\sigma|}|f(\sigma_i)-f(\sigma_{i-1})|\mathbf{1}_{\{f(\sigma_i)<f(\sigma_{i-1})\}},
\]
where the supremum is taken on the set of all partition $\sigma$ of $[0,x]$. Let $f$ be a function of $\mathbb{D}([0,1])$ such that
$$
\forall x \in[0,1), \ V(f,x)<\infty, \ V_{-}(f)<\infty, \quad \lim\limits_{x\to1}V(f,x)=\infty.
$$
\iffalse
\[
\left\{
\begin{array}{l}
V(f,x)<\infty, \quad \forall x \in[0,1),\\
V_{-}(f)<\infty,\\
V(f,x)\xrightarrow[x\to 1]{}\infty.\\
\end{array}
\right.
\]
\fi
Then $
\lim\limits_{x\to 1}f(x)=\infty\
$
\end{lem}
\begin{proof}
Let $f$ be such a function.
Without loss of generality, we may assume that $f(0)=0$. 
Because of $V(f,x)\xrightarrow[x\to 1]{}\infty$, we have that, for any $M>0$, there exists some partition $\sigma$ of $[0,1]$ such that
\[
\sum_{i=1}^{|\sigma|}|f(\sigma_i)-f(\sigma_{i-1})|>M,
\]
and $\sigma_{1}=0$ and $\sigma_{|\sigma|}=1$.
Now, we have
\begin{multline*}
f(1)=\sum_{i=1}^{|\sigma|}f(\sigma_{i})-f(\sigma_{i-1})=\sum_{i=1}^{|\sigma|}\left(f(\sigma_{i})-f(\sigma_{i-1})\right)\mathbf{1}_{\{f(\sigma_i)\geq f(\sigma_{i-1})\}}-\sum_{i=1}^{|\sigma|}\left|f(\sigma_{i})-f(\sigma_{i-1})\right|\mathbf{1}_{\{f(\sigma_i)<f(\sigma_{i-1})\}}\\
\geq \sum_{i=1}^{|\sigma|}\left(f(\sigma_{i})-f(\sigma_{i-1})\right)\mathbf{1}_{\{f(\sigma_i)\geq f(\sigma_{i-1})\}}-V_{-}(f).
\end{multline*}
Moreover, it is easily seen that
\begin{multline*}
\sum_{i=1}^{|\sigma|}\left(f(\sigma_{i})-f(\sigma_{i-1})\right)\mathbf{1}_{\{f(\sigma_i)\geq f(\sigma_{i-1})\}}=\sum_{i=1}^{|\sigma|}\left|f(\sigma_{i})-f(\sigma_{i-1})\right|-\sum_{i=1}^{|\sigma|}\left|f(\sigma_{i})-f(\sigma_{i-1})\right|\mathbf{1}_{\{f(\sigma_i)<f(\sigma_{i-1})\}}\\ \geq \sum_{i=1}^{|\sigma|}\left|f(\sigma_{i})-f(\sigma_{i-1})\right|-V_{-}(f).
\end{multline*}
Hence,
\[
f(1)=\sum_{i=1}^{|\sigma|}f(\sigma_{i})-f(\sigma_{i-1})\geq
\sum_{i=1}^{|\sigma|}\left|f(\sigma_{i})-f(\sigma_{i-1})\right|-2V_{-}(f)\geq M-V_{-}(f).
\]
Finally, since $M$ is arbitrary, we obtain the desired result.
\end{proof}
In particular, let us highlight that a diffusion has unbounded variation on any compact set. So, such a diffusion as one in the preceding lemma would instantaneously explode.
We can now prove our theorem.
\begin{proof}[Proof of Theorem~\ref{prop:noscaling}]
Let $(\mu_{n})_{n\geq 1}$ be the sequence of measure on $\mathbb{D}([0,1])$ corresponding to the laws the sequence $(X^{n})_{n\geq 1}$. Assume that $(\mu_{n})_{n\geq 1}$ converges weakly to some measure $\mu$. Since, the only negative variation of $X^{n}$ is provided by the negative drift (with slope $1$), we have that 
\[
V_{-}(X^{n})\leq 1,\quad \text{alsmot surely}.
\]
Hence, $\mu_{n}(V_{1})=1$ for all $n\geq 1$. Now, since $V_{1}$ is closed, the portemanteau theorem entails that
\[
1=\limsup_{n\to\infty}\mu_{n}(V_{1})\leq \mu(V_{1}).
\]
Hence, we necessarily have $\mu(V_{1})=1$. This ends the proof.
\end{proof}

\section*{Appendix: a quick reminder on Generator}%\cmb{Si tu parles de générateur avant il faut peut-être mettre cette sou%-section avant/ ou en annexe}%{Generator and Martingale properties}%{Definition and first properties of the contour process}

Let us begin by reminding some notions on the generator of Markov processes (which also enable to precise our notation). 

Given any Markov process $(Z_t)_{t\geq 0}$ on $\mathbb{R}_+$ (or any more general state space). One can define its associated semigroup $(P_t)_{t\geq 0}$ by
$$
\forall t\geq 0, \forall x\in \mathbb{R}_+, \qquad P_tf(x) = \mathbb{E}[f(Z_t) \ | \ Z_0=x],
$$
for every function $f$ in some space of functions $L$. This space of functions can be chosen as the space of continuous functions vanishing at infinity $C_0(\mathbb{R}_+)$, the space of bounded and continuous functions $C_b(\mathbb{R}_+)$ or the space $L^2(\mathbb{R}_+,\pi)$ of functions which are square-integrable functions with respect to the invariant measure of the semigroup (when it exists); see for instance \cite{B94,EK86}.  when $(L, \| \cdot \|_L)$ is a Banach space (for some norm $\| \cdot \|_L$), we say that $(P_t)$ is strongly continuous if
$$
\lim_{t\to 0} \Vert P_t f- f \Vert_L =0,
$$ for every function $f\in L$.  Recall that when $L=C_0$ then the process $(Z_t)_{t\geq 0}$ is called a Feller process.

One can define the \textit{ infinitesimal generator} or \textit{strong generator} $A$ of the semigroup $(P_t)_{t\geq 0}$ on $L$  by
$$
Af= \lim_{t \to 0} \frac{1}{t}(T_t f- f),
$$
for every $f$ in $\mathcal{D}_L (A) \subset L$, the domain of $A$, which corresponds to the functions $f\in L$ such that the previous limit holds. When there is no ambiguity on $L$ we note $\mathcal{D} (A)$. A powerful property of the generator is given by the Dynkin formula and reads as the process $(M_t)_{t\geq 0}$, given by
\begin{equation}
\label{eq:dinkyn}
M_t = f(X_t) - \int_0^t Af(X_s) ds,
\end{equation}
is a martingale for any $f\in L$. see in particular \cite[Theorem (14.13) p.31]{D93}. In order to generalize this property, one can define the \textit{full generator} %\cmb{D après le thm 3.4 de la ref generator, si un semigroupe est fortement continu alors le full domain est le strong c'est le même}
of $(Z_t)_{t\geq 0}$ on the domain 
$$
\overline{D}_L(A)= \left\{f \in L \ | \ \exists g_f \in L, \forall t\geq 0 \quad P_t f=f + \int_0^t P_t g ds \right\}. 
$$ 
For $f\in \overline{D}_L(A)$, we set $Af=g_f$. Again, $(M_t)$ defined by \eqref{eq:dinkyn} is a martingale for all $f\in \overline{D}(A)$; see \cite[Proposition 1.7 p.162]{EK86}. We can again weaken the assumptions on $f$ by considering directly the space $\widehat{D}(A)$ of functions $f$ such that there exists a function $g_f$, such that $(M_t)_{t\geq 0}$, defined by
$$
 \forall t\geq 0, \quad M_t = f(Z_t) - \int_0^t g_f(X_s) ds, 
$$ 
is a local Martingale. Again we set $Af=g_f$, and $A$ with its domain $\widehat{D}(A)$ is called the extended generator. Note that the extended and full generator are multi-valued, in the sense that it is possible to find different $g_f$ for the same $f$. However, these functions will be different only on a negligible set (see \cite[page 32-33]{D93}). Also, there is no problem of notation because $D_L(A)\subset \overline{D}_L(A) \subset \widehat{D}(A)$ and if $L'\subset L$ then $D_{L'}(A)\subset D_L(A)$, $\overline{D}_{L'}(A) \subset \overline{D}_L(A)$.

\section*{Acknowledgments}

The authors have received the support of the Chair ``Mod\'elisation Math\'ematique et Biodiversit\'e'' of VEOLIA-Ecole Polytechni\-que-MnHn-FX.
B.C has been supported by ANR projects, funded by the French Ministry of Research ANR PIECE (ANR-12-JS01-0006-01).%, ANR ANSWER (???) and PROMMECE???

\bibliographystyle{plain}
\bibliography{biblio.bib}
\end{document}